\documentclass[UTF8]{ctexart}

\usepackage[a4paper, margin = 2.33cm, centering, includefoot]{geometry}

\usepackage[svgnames, dvipsnames]{xcolor} 

\usepackage[unicode=true, bookmarksnumbered=true, colorlinks=true]{hyperref}
\colorlet{linkblue}{blue!45!black}

\hypersetup{
    citecolor=magenta, 
    linkcolor=linkblue, 
    urlcolor=gray
}

\usepackage{xurl} 

\usepackage{latexsym,mathtools,mathalfa,amssymb,amsmath,amsbsy,amsopn,amstext,amsthm,amsxtra,bm, bbm,calc,ifpdf,cancel,extarrows,manfnt}
\usepackage{enumerate,graphicx,wrapfig,subfig,diagbox,listings,multicol,multirow,makeidx,fontspec,marvosym}
\usepackage{booktabs,tabularx,array,float}

\newtheorem{Lem}{Lemma}[section]
\newtheorem{Thm}{Theorem}[section]
\newtheorem{Prop}{Proposition}[section]
\newtheorem{Ass}{Assumption}[section]
\newtheorem{Rmk}{Remark}[section]
\newtheorem{Eg}{Example}[section]
\newtheorem{IEg}{Illustrative Example}[section]
\newtheorem{corollary}{Corollary}[section] 
\newtheorem{condition}{Condition}[section]
\newtheorem{definition}{Definition}[section] 

\usepackage{fancyhdr}
\renewcommand{\sectionmark}[1]{}

\usepackage{caption}
\usepackage[section]{placeins}

\usepackage{lmodern}
\usepackage[T1]{fontenc}
\usepackage{anyfontsize}
\usepackage{tikz}
\usetikzlibrary{decorations.pathreplacing, calligraphy, patterns, arrows.meta, positioning}

\ctexset{
 section = {
    format = \Large\bfseries\raggedright,
    numberformat = {},
    nameformat = {}  
  }, 
  abstractname = {Abstract},
  contentsname = {Contents},
  part = {
    name   = {Part~},
    number = \Roman{part}
  },
  bibname = {References},
  figurename = {Figure},  
  tablename  = {Table}   
}

\usepackage[authoryear,round, sort&compress]{natbib} 

\newcommand{\V}{\mathsf{V}}  
\newcommand{\U}{\mathsf{U}}  
\newcommand{\Sf}{\mathsf{S}} 

\newcommand{\1}{\mathbf{1}}				
\renewcommand{\P}{\mathbb{P}}				
\newcommand{\E}{\mathbb{E}}					
\newcommand{\Bro}{\mathbb{B}}               
\newcommand{\Var}{\mathbb{V}}		
\newcommand{\Lp}{\mathcal{L}}

\newcommand{\R}{\mathbb{R}}					
\renewcommand{\d}{\mathrm{d}}				

\DeclareMathOperator*{\argmin}{argmin}

\newcommand{\abs}[1]{\left\lvert#1\right\rvert}		
\newcommand{\norm}[1]{\left\lVert#1\right\rVert}

\newcommand{\oas}[1]{o_{\mathrm{a.s.}}(#1)}                   
\newcommand{\op}[1]{o_{\P}(#1)}                   
\newcommand{\T}[2]{\mathcal{T}_{#1}[#2]}  

\newcommand{\eps}{\varepsilon}              
\newcommand{\block}{\mathcal{B}}            

\usepackage[en-US]{datetime2} 

\title{\textbf{Time-Uniform Dependent Central Limit Theory}\footnote{The authors are listed in alphabetical order. We acknowledge the use of ChatGPT and Claude for assistance in draft editing, coding, proofreading and finding useful literature. We thank Alessandro Casini, Nabarun Deb, Florian Gunsilius, Christian Hansen, Shiwei Huang, Shubhanshu Shekhar, Weibiao Wu, Kaspar Wüthrich, and seminar participants at Michigan Econometrics Workshop for helpful comments and feedback. We also thank Yuchen Hu and Robert de Jong for insightful discussions.  All errors are our own.}}
\author{Mingxuan Ge\thanks{University of Michigan, Ross School of Business: \href{mkgg@umich.edu}{mkgg@umich.edu}} \quad Yanlong Liu\thanks{University of Chicago, Booth School of Business: \href{yliu420@chicagobooth.edu}{yliu420@chicagobooth.edu}}}
\date{\small This draft version: \today}

\begin{document}
\maketitle

\begin{abstract}
We establish a near-optimal almost sure invariance principle for nonstationary time series under the functional dependence measure (FDM) framework. This result yields a universal time-uniform dependent central limit theorem, which serves as the basis for anytime-valid inference. In particular, we instantiate the theorem through asymptotic confidence sequences (AsympCSs), time-uniform counterparts of asymptotic confidence intervals obtained via central limit theorems, extending those originally proposed in \cite{waudby2024time} to time series with serial dependence. We further refine the original definition because, as we show, it can admit AsympCSs that fail to achieve asymptotic time-uniform coverage even for independent and identically distributed settings. We address this issue by adding a simple eventual-inclusion condition, which in turn permits more flexible variance estimation. In time-series settings, this relaxation can substantially ease restrictions on bandwidth and kernel selection. For feasible implementation of AsympCSs, we establish strong consistency of HAC variance estimators under the FDM framework, accommodating a wider range of commonly used bandwidth choices than existing results under mixing conditions. As illustrations, we construct AsympCSs for algorithms with Markovian structure and time-series-based causal estimands.

\vspace{1em}
\noindent 
\textbf{Keywords:} anytime-valid inference, asymptotic confidence sequences, functional dependence measure, nonlinear time series, nonstationary almost sure invariance principle, strong consistency of HAC estimators
\end{abstract}

\section{Introduction} \label{Sec: Introduction}
\noindent Many modern time-series applications involve data collected continuously and often at high frequency, creating a natural need to conduct statistical inference as data arrive rather than only at a prespecified endpoint or at a limited number of interim analyses. For example, outcomes in time-series experiments (e.g., \cite{bojinov2019time}; \cite{hu2022switchback}; \cite{bojinov2023design}) are often observed at 30-minute or hourly intervals, making continuous monitoring particularly valuable for limiting exposure to poorly performing treatments and terminating experiments as soon as sufficient evidence has accumulated (\cite{ham2022design}). Similar needs arise in sequential change-point detection (e.g., \cite{chu1996monitoring}; \cite{leisch2000monitoring}; \cite{berkes2004sequential}; \cite{gosmann2021new}; \cite{horvath2025sequential}), where the aim is to promptly detect changes in the parameters of a time series to issue an alarm and mitigate the consequences of abnormal behavior. Iterative algorithms such as Markov chain Monte Carlo (MCMC) and stochastic gradient descent (SGD) also naturally give rise to time series (e.g., \cite{dieuleveut2020bridging}; \cite{li2024sgd}; \cite{merad2025convergence}; \cite{wei2025online}), where it is desirable to determine from the realized trajectory whether further iteration is statistically warranted, rather than committing to an ad-hoc iteration horizon. Such practical interest is already reflected in recent work establishing time-uniform guarantees for algorithmic procedures (e.g., \cite{xie2024asymptotic}; \cite{aolaritei2025stopping}; \cite{ritzwoller2026reproducible}; \cite{aolaritei2026beyond}).

These applications commonly involve temporal heterogeneity and time-varying effects, which can give rise to departures from stationarity. Accommodating potential nonstationarity is particularly important in our anytime-valid sequential settings, where one cannot generally rely on existing pretesting procedures to restrict attention to stationary processes (e.g., \cite{horvath2020sequential}; \cite{schmidt2021asymptotic}). Moreover, units in these settings are often generated through recursive decision-making and feedback dynamics, resulting in dependence patterns that can be highly nonlinear and complex. These practical considerations thus raise an important question of how to conduct valid inference under continuous monitoring for nonlinear time series with serial dependence and nonstationarity, while maintaining uniform Type-I error control over time. Such an inferential guarantee is precisely the objective of anytime-valid inference (AVI) (\cite{ramdas2023game}). While much of the AVI literature focuses on finite-sample guarantees, inference for time series relies on inherently asymptotic quantities such as the long-run variance. We therefore take an asymptotic approach in this paper. Our primary goal is to develop a time-uniform dependent central limit theorem (TUDCLT), aiming to bring the broad applicability of the classical fixed-time central limit theorem (CLT) to the anytime-valid setting. We next describe how we develop this approach and summarize our main contributions.

\vspace{1em}
\noindent \textbf{Our Approach and Contributions.} Our major technical contribution is to establish an almost sure invariance principle (ASIP) for nonstationary time series under the functional dependence measure (FDM) framework (\cite{wu2005nonlinear}), with an approximation rate that matches the optimal KMT rate (\cite{komlos1975approximation}; \cite{komlos1976approximation}) up to an additional logarithmic factor. The FDM framework provides a tractable approach to characterize nonlinear time series  where other dependence conditions such as mixing (\cite{doukhan1994mixing}) can be difficult to verify or fail to hold. An almost-sure notion is needed because anytime validity with an open-end guarantee is an eventual pathwise property in the asymptotic regime. We refer readers to Remark 1 of \cite{waudby2024time} for formal arguments and to Remarks 2.2 and 2.6(i) of \cite{gosmann2021new} for further discussion. Our ASIP result then reduces the otherwise intractable problem of AVI for time series to an analysis of the approximating Gaussian process, providing the foundation for our TUDCLT.  

Following \cite{waudby2024time}, we formulate the resulting TUDCLT in terms of asymptotic confidence sequences (AsympCSs), the time-uniform counterparts of CLT-based asymptotic confidence intervals (AsympCIs). The crux of the approach is to transfer nonasymptotic Gaussian boundaries obtained from martingale inequalities (\cite{ramdas2020admissible}) to the original partial-sum process via the ASIP, which we establish for nonstationary time series. This construction thus relaxes the finite-sample distributional requirements of nonasymptotic confidence sequences (e.g., \cite{darling1967confidence}; \cite{howard2021time}), and we adopt this formulation because it often offers greater flexibility than its dual counterpart, anytime-valid sequential testing, especially when the inferential objective is not clearly specified a priori or may evolve over time. Our formulation further refines the original definition of AsympCSs. As we show in Example \ref{IEg 2}, there exist AsympCSs that satisfy the original definition but fail to achieve asymptotic time-uniform coverage even in independent and identically distributed settings, and the result continues to hold under a broad class of data-generating processes. We address this issue by adding a simple eventual-inclusion condition. Our refinement is particularly useful in time-series settings, as it allows greater flexibility in kernel and bandwidth selection for long-run variance estimation than the approach taken in Section 2.5 of \cite{waudby2024time} to address a related issue in a specific setting.

For feasible implementation of AsympCSs, we establish strong consistency of heteroskedasticity and autocorrelation consistent (HAC) estimators under the FDM framework. To the best of our knowledge, there appears to be no strongly consistent estimator of the long-run variance for nonstationary time series under the FDM framework. Our result permits substantially larger bandwidth choices than those available under mixing conditions (\cite{de2000strong}), with a polynomial-factor improvement up to logarithmic terms. This thus recovers many standard bandwidth choices such as the conventional $b_t \asymp t^{1/3}$ rate associated with the Bartlett kernel under finite fourth moments. As illustrations, we construct AsympCSs for algorithms with Markovian structure and time-series-based causal estimands. We next survey related strands of the literature. 

\vspace{1em}
\noindent \textbf{Literature.} The paper most closely related to ours is \cite{waudby2024time}, which develops a time-uniform CLT for independent and identically distributed (i.i.d.) data and martingale difference sequences (MDSs). However, its framework does not accommodate more general serial dependence and excludes standard autoregressive processes. A subsequent paper briefly discusses stationary time series in Remark 4.2 (\cite{Gnettner2025}), but leaves nonstationarity unaddressed. This is partly due to the lack of a relevant ASIP, which we establish in this paper. Establishing such an ASIP is itself a nontrivial technical challenge. To the best of our knowledge, beyond MDSs, the only existing relevant nonstationary ASIP is \cite{hafouta2023almost} and other ASIP results are confined to stationary processes (e.g., \cite{dehling1982almost}; \cite{berkes2014kmt}).  However, the main theorem in \cite{hafouta2023almost}, established for uniformly bounded sequences with exponentially decaying $\alpha$-mixing coefficients and an additional $\phi$-mixing condition, only provides a Gaussian approximation whose variance approximately matches that of the original partial sums, and does not facilitate the construction of AsympCSs for nonstationary time series. We establish the approximation whose variance exactly matches that of the original partial sums under substantially weaker conditions in the FDM framework, with a comparison of the assumptions provided in Appendix \ref{Appendix: Assumptions in Existing Nonstationary ASIP Results}. In particular, our assumptions are comparable in spirit to those assumed in CLTs but we get the additional anytime validity.

More broadly, our work contributes to a growing literature that uses ASIPs to establish asymptotic time-uniform guarantees under weak distributional assumptions for anytime-valid inference. Following \cite{waudby2024time}, such results have been developed in cross-sectional i.i.d.\ settings (e.g., \cite{dalal2024anytime}; \cite{oprescu2025efficient}; \cite{imbens2026demonstration}), martingale difference sequences (e.g., \cite{bibaut2022near}; \cite{cho2025exploration}), and design-based finite-population frameworks (e.g., \cite{liang2023experimental}; \cite{molitor2025anytime}; \cite{wang2025design}; \cite{ham2022design}). These results, however, do not cover super-population time-series settings. Relatedly, \cite{shi2023dynamic} consider super-population time-series experiments but only construct a group sequential test with a fixed number of interim analyses. Specifically, they build on the $\alpha$-spending strategy (\cite{lan1983discrete}) and rely on computationally demanding bootstrap procedures. Our approach instead yields more analytically tractable time-uniform boundaries once the long-run variance is estimated, avoiding the recursive calibration of critical boundaries to control the overall Type-I error across interim analyses.  

\vspace{1em} 
\noindent \textbf{Paper organization.} The rest of the paper is organized as follows. In Section \ref{Sec: Almost Sure Invariance Principles for (Nonstationary) Time Series}, we introduce the time-series set-up under the functional dependence measure (FDM) framework and establish a near-optimal nonstationary ASIP result. We then instantiate our TUDCLT in terms of AsympCSs in Section \ref{Sec: Main Results}. Specifically, we revisit the existing definition of AsympCS and propose a refinement in Section \ref{Sec: Asymp CS}, followed by presentation of related main results in Sections \ref{Sec: Robbins Gaussian Mixture Asymp CSs for (Nonstationary) Time Series}-\ref{Sec: Iterated Logarithm Asymp CSs for (Nonstationary) Time Series}. For practical implementation of our AsympCSs, we establish the strong consistency of HAC estimators in Section \ref{Sec: Strongly Consistent Estimation Strategies}. In Section \ref{Sec: Illustrations: Applications to Time Series Data}, we demonstrate the use of AsympCSs in common time-series-oriented applications. We conduct simulation studies in Section \ref{Sec: Simulation Studies} and conclude in Section \ref{Sec: Conclusion}. Additional technical details, proofs, and discussions are provided in the supplementary material (Appendix \ref{Appendix: Proofs of Nonstationary ASIP} - Appendix \ref{Appendix: Additional Discussions}).

\vspace{1em}
\noindent \textbf{Notation.}  For a random variable $X$ and $p\ge1$, let $\norm{X}_p=(\E|X|^p)^{1/p}$ denote the $\Lp_p$-norm, and abbreviate $\norm{\cdot}=\norm{\cdot}_2$. We use $\E[X]$ and $\Var[X]$ to denote the expectation and variance of $X$, respectively.

\section{Time-Series Set-Up and Almost Sure Invariance Principle} \label{Sec: Almost Sure Invariance Principles for (Nonstationary) Time Series}
\noindent In this section, we introduce the time-series framework and establish a near-optimal ASIP for nonstationary time series, which serves as the main technical tool for deriving our AsympCSs as instantiations of the TUDCLT.

Following a common convention in time-series analysis (\cite{hamilton2020time}), we decompose the process $(Y_i)_{i\in\mathbb{Z}}$ as
\begin{equation} \label{eq:additive_decomposition}
Y_i = \mu_i + X_i,    
\end{equation}
where $\mu_i := \E[Y_i]$ is a deterministic mean sequence that may vary over time, and $(X_i)_{i\in\mathbb{Z}}$ is a centered stochastic process subject to dependence assumptions that we will specify in Assumptions \ref{Assump:dep_decay}-\ref{Assump:UI}.

Many theoretical developments in the time-series literature proceed under a mean-zero assumption (e.g., \cite{xiao2012covariance}; \cite{zhang2017gaussian}; \cite{bonnerjee2024gaussian}). We explicitly retain the deterministic mean sequence $(\mu_i)_{i\in\mathbb{Z}}$ in Equation \eqref{eq:additive_decomposition}, as its running averages are our inferential targets. We also refrain from imposing additional structure on the mean to keep our framework agnostic about its specific form. To prevent first-stage mean estimation errors from contaminating variance estimation, a phenomenon known as low-frequency contamination (\cite{casini2026theory}), we impose a mild Ces\`aro mean-stability condition similar to those in the relevant literature (\cite{johari2025estimation}; \cite{mathis2026confidence}). The formal statement is deferred to Condition \ref{Cond: Ces\`aro Mean-Stability Condition} in Section \ref{Sec: Strongly Consistent Estimation Strategies}. In short, we rule out persistent mean nonstationarity, such as deterministic trends or structural breaks occupying a non-vanishing fraction of the sample. The nonstationarity of our primary interest mainly arises from temporal heterogeneity in the stochastic component, as is often the case in the motivating applications discussed in Section \ref{Sec: Introduction}. Following the discussion therein, we adopt the functional dependence measure (FDM) framework of \cite{wu2005nonlinear} to model the stochastic component $(X_i)_{i\in\mathbb{Z}}$. For readers less familiar with FDM, we provide a brief overview and discussion in Appendix \ref{Appendix: Functional Dependence Measure Framework}. 

In this framework, we characterize $(X_i)_{i\in\mathbb{Z}}$ as the output of a physical system:
\begin{equation}   \label{eq:g_function}    
X_i = \varphi_i(\eps_i, \eps_{i-1}, \dots),
\end{equation}
where $(\eps_i)_{i\in\mathbb{Z}}$ are i.i.d.\ inputs, 
and $\varphi_i:\R^{\infty}\to \R$ are measurable functions.

We critically allow the physical mechanism $\varphi_i$ to vary over time to incorporate nonstationarity. Based on this representation, given a time lag $h\geq0$, we quantify the dependence of the system on the input $\eps_{i-h}$ by replacing it with an i.i.d.\ copy $\eps'_{i-h}$. Let $X_{i,\{i-h\}} = \varphi_i(\eps_i, \dots, \eps_{i-h+1},\eps_{i-h}',\eps_{i-h-1},\dots)$ denote the coupled version of $X_i$. For $p\geq1$, we define the uniform functional dependence measure as
\begin{equation}      
    \delta_p(h):= \sup_{i\in\mathbb{Z}} \norm{X_i - X_{i,\{i-h\}}}_p,
     \label{eq:f_measure} 
\end{equation}
where the term $\norm{X_i - X_{i,\{i-h\}}}_p$ captures the sensitivity of the output $X_i$ to the innovation $\eps_{i-h}$, and the supremum over all $i \in \mathbb{Z}$ is taken to enable a uniform measure of dependence over nonstationarity.

We then formulate the dependency condition using
\begin{equation}    
    \Theta_{i,p} = \sum_{h=i}^{\infty} \delta_p(h),\ i\geq0,
    \label{eq:dep_measure}   
\end{equation}
in which this quantity measures the cumulative dependence of the future trajectory on an input shock. To enable the construction of our TUDCLT, similar to the counterpart in the classic fixed-time CLT regime, we assume a short-range dependence condition throughout the rest of the paper: $\Theta_{0,p}<\infty$. If this fails, the process $(X_i)$ can be long-range dependent, and the partial sum process $\mathsf{S}_t := \sum_{i=1}^t X_i$ will no longer converge to a Brownian motion limit even in the stationary case (\cite{taqqu1975weak}).

We now introduce the primitive assumptions underlying the ASIP result. First, we impose a mild polynomial decay condition on the dependence structure of the process.
\begin{Ass}[Polynomial-Rate Functional Dependence Decay] \label{Assump:dep_decay}
    Consider Equation~\eqref{eq:g_function}. Suppose that $\Theta_{0,p}<\infty$ for $p \in (2,3]$. Assume there exist $A>1$ and a positive constant $C>0$ such that
    \begin{align}
        \sup_{i\geq0} (i+1)^A \Theta_{i,p} \leq C.
        \label{Eq:dep_decay}
    \end{align}
\end{Ass}
\noindent Assumption \ref{Assump:dep_decay} holds for a broad range of processes. We provide a collection of stationary and nonstationary examples from the literature that satisfy or violate this assumption in Appendix~\ref{Appendix: Examples Satisfying and Violating the Assumptions}. In particular, Assumption \ref{Assump:dep_decay} accommodates processes with arbitrarily slow polynomial decay above the summability threshold $A=1$. We formulate this assumption under minimal conditions, requiring only finite moments of order slightly greater than two, i.e., $p > 2$. Higher-order moment conditions will improve the approximation error rate, whereas the rate achieved under Assumption \ref{Assump:dep_decay} is already sufficient for constructing AsympCSs. We thus defer a full discussion for general $p$ to Appendix \ref{Appendix: The Complete Theorem}. 

In addition to the dependency decay, due to the potentially nonstationary nature of $(X_i)$, we assume the following stochastic domination condition to ensure a uniform control over the tail behavior. 

\begin{Ass}[Stochastic Dominance] \label{Assump:UI}
    For the same $p$ as in Assumption \ref{Assump:dep_decay}, there exist a non-negative random variable $Q \in\Lp_p$ and a constant $C > 0$, such that for all $x > 0$:
\begin{equation}
    \sup_{i\in\mathbb{Z}}\mathbb{P}\left[|X_i| > x\right] \leq C  \mathbb{P}\left[Q > x\right].
\end{equation}
\end{Ass}
\noindent Assumption \ref{Assump:UI} mainly rules out unit-root processes and we defer specific examples to Appendix~\ref{Appendix: Examples Satisfying and Violating the Assumptions}. Taken together, these two assumptions accommodate many familiar stationary and nonstationary time-series models already covered by classical fixed-time CLTs. This thus suggests that the practical scope of our TUDCLT is comparable in spirit to that of classical fixed-time CLTs but we obtain the additional anytime validity.

With these preliminary conditions in place, we are now ready to present our nonstationary ASIP result. For readability, with a slight abuse of notation, we index sequences of random variables starting from 1 rather than over $\mathbb{Z}$ for the rest of the paper, unless otherwise noted.  

\begin{Thm}[Nonstationary ASIP] \label{Thm: nonstationary ASIP}
For the process $(Y_i)_{i\geq1}$ satisfying Equation \eqref{eq:additive_decomposition}, suppose Assumptions \ref{Assump:dep_decay} and \ref{Assump:UI} hold. Let $\V_t := \Var\left[\sum_{i=1}^{t} Y_i\right]$ be the cumulative variance process. Then, on a possibly richer probability space, there exists a standard Brownian motion $\Bro(\cdot)$ such that
\begin{equation} \label{eq: KMT rate}
\sum_{i=1}^{t} (Y_i - \mu_i)  = \Bro(\V_t) + \oas{t^{1/p}\sqrt{\log t}}.    
\end{equation}
\end{Thm}
\noindent The complete proof is provided in Appendix \ref{Appendix: Proofs of Nonstationary ASIP}, along with an outline in Appendix \ref{sec:proof_ASIP} to highlight the main ideas since techniques are quite involved.

We briefly highlight some technical contributions. The approximation error rate in Equation \eqref{eq: KMT rate} is within only a logarithmic factor of the optimal KMT rate, i.e., $o_{\mathrm{a.s.}}(t^{1/p})$ (see, e.g., \cite{komlos1975approximation} \& \cite{komlos1976approximation} for the classical KMT results and \cite{berkes2014kmt} for an extension to stationary time series). Under an additional non-degeneracy condition on $\V_t$, this logarithmic factor can be removed, yielding the optimal KMT rate (see, e.g., the discussion of Condition 2.4 and Theorem 2.5 in \cite{bonnerjee2024gaussian}). We pursue this sharp construction because a finer control of the approximation errors typically entails a more careful analysis of the dependence structure, often leading to weaker dependence conditions, as is the case shown in Assumption \ref{Assump:dep_decay}. To this end, we develop a new multiscale argument that partitions intrinsic time into geometrically growing blocks and aggregates the resulting blockwise bounds. Specifically, we draw on the triadic grouping idea of \cite{berkes2014kmt}. Across the resulting blocks, we substantially reorganize and then strengthen the reduction steps of \cite{bonnerjee2024gaussian}. This allows us to obtain the near-optimal KMT-type approximation error rate under mild assumptions. Lastly, we provide a brief discussion of related results and further applications which may be of independent interest in Appendix \ref{Appendix: Non-Degenerate Condition on the Cumulative Variance Process}.

\section{Asymptotic Confidence Sequences} \label{Sec: Main Results}
\noindent Having developed the ASIP in Section \ref{Sec: Almost Sure Invariance Principles for (Nonstationary) Time Series}, we now use it to reduce the time-uniform inference problem to a finite-sample Gaussian one. Specifically, we instantiate our TUDCLT in terms of AsympCSs. In Section \ref{Sec: Asymp CS}, we first propose a refined notion of AsympCSs originally introduced in \cite{waudby2024time}.  Sections~\ref{Sec: Robbins Gaussian Mixture Asymp CSs for (Nonstationary) Time Series} and~\ref{Sec: Iterated Logarithm Asymp CSs for (Nonstationary) Time Series} then present Robbins' Gaussian mixture and iterated-logarithm AsympCSs for nonstationary time series, respectively. Finally, Section~\ref{Sec: Strongly Consistent Estimation Strategies} establishes strong consistency of the HAC estimator for long-run variance under FDM, enabling practical implementation.

\subsection{Refining Asymptotic Confidence Sequences} \label{Sec: Asymp CS}
\noindent We begin by reviewing the definition of a nonasymptotic confidence sequence (CS) before introducing its asymptotic counterpart, an asymptotic confidence sequence (AsympCS). Dating back to \cite{darling1967confidence}, a sequence of sets $\{C_t^\star\}_{t\ge1}$ is a
nonasymptotic $(1-\alpha)$ confidence sequence (CS) for the estimand $\theta_0$ if $\P \!\left[
\forall t\ge1:\theta_0\in C_t^\star
\right]\ge 1-\alpha.$ 

We next propose a refinement of the notion of AsympCSs originally introduced in \cite{waudby2024time}. Informally, an AsympCS is a sequence that asymptotically behaves like the underlying nonasymptotic CS, thereby remaining valid after a ``sufficiently large'' burn-in period. 

\begin{definition}[Asymptotic Confidence Sequences] \label{Def: Asym CS}
A sequence of random intervals $\left\{\widetilde C_t\right\}_{t \ge 1} :=
\left\{\left[\widehat\theta_t-\widetilde L_t,\,
\widehat\theta_t+\widetilde U_t\right]\right\}_{t \ge 1}$ centered at estimators $\left\{\widehat \theta_t\right\}_{t \ge 1}$ with non-zero bounds, i.e., $\widetilde L_t, \widetilde U_t > 0$, is a \emph{two-sided} $(1-\alpha)$-\emph{asymptotic confidence sequence (AsympCS)} for the estimand $\theta_0$ if there exists a (potentially unknown) \emph{two-sided nonasymptotic $(1-\alpha)$-CS} $\left\{C_t^\star\right\}_{t \ge 1} :=
\left\{\left[\widehat\theta_t-L_t^\star,\,
\widehat\theta_t+U_t^\star\right]\right\}_{t \ge 1}$ for
$\theta_0$ such that the following two conditions hold:
\begin{enumerate}
    \item \textbf{Almost-sure relative approximations:} $\widetilde{L}_t / L_t^\star \to 1 \ \text{and} \
\widetilde{U}_t / U_t^\star \to 1
\ \text{almost surely}.$
\item \textbf{Eventual inclusion:} $C_t^\star \subseteq \widetilde C_t$ eventually almost surely.
\end{enumerate}
\end{definition}
\noindent Definition \ref{Def: Asym CS} applies to arbitrary data-generating processes (DGPs), although this paper focuses exclusively on time series. Analogous definitions apply to one-sided lower or upper $(1-\alpha)$-AsympCSs on $ [\hat{\theta}_t-\widetilde{L}_t,\infty)
\ \text{or} \
(-\infty,\hat{\theta}_t+\widetilde{U}_t],$ respectively, so we omit them here.

As discussed in Section \ref{Sec: Introduction}, almost sure guarantees are essential here. For AsympCIs derived from CLTs, an analogue of Condition 1 holds in probability. Here, the almost-sure notion is used in place of the in probability notion to obtain time-uniformity in AsympCSs. Beyond Condition 1, we additionally impose Condition 2, which is not required in the original definition of \cite{waudby2024time}. In words, Condition 2 means that the boundary endpoints of AsympCSs need to be eventually no smaller than those of the underlying CSs, i.e., $\widetilde L_t \geq L_t^{\star}$ and $\widetilde U_t \geq U_t^{\star}$ eventually almost surely (see e.g., p. 23 of \cite{kallenberg1997foundations} for ``eventually/ultimately'' and p. 11 of \cite{shields1996ergodic} for ``eventually almost surely''). This strengthening is important in practice, as Condition~1 alone need not lead to a desired coverage guarantee. As the following Example \ref{IEg 2} demonstrates, iterated-logarithm-rate boundaries can exhibit severe undercoverage, with coverage probabilities approaching zero.

\begin{IEg}[Failure of Coverage under the Original AsympCS Defn]
\label{IEg 2}
Let $X_i=\theta_0+\varepsilon_i$, where
$\varepsilon_i\overset{\mathrm{iid}}{\sim}\mathcal N(0,1)$, and define
$\widehat\theta_t:=\bar X_t=t^{-1}\sum_{i=1}^t X_i$.
For $a>0$, let
\[
u_t(a):=
\sqrt{t\left(2\log\log t+a\log\log\log t\right)},
\qquad t\ge16.
\]
Fix $\alpha\in(0,1)$ and an arbitrary initialization point $T_0\ge16$.
Then, for sufficiently large constants $K_{\alpha/2}>0$ and
$R_{\alpha/2}>0$, define
$g_t^\star=R_{\alpha/2}$ for $t<T_0$ and
$g_t^\star=u_t(4)+K_{\alpha/2}$ for $t\ge T_0$.
The sequence $\{C_t^\star\}_{t\ge1}
:=
\left\{
\left[
\bar X_t-{g_t^\star}/{t},
\,
\bar X_t+ {g_t^\star}/{t}
\right]
\right\}_{t\ge1}$ is then a nonasymptotic $(1-\alpha)$-CS for $\theta_0$. Now define $\widetilde C_t
:=
\left[
\bar X_t-{\widetilde g_t}/{t},
\,
\bar X_t+{\widetilde g_t}/{t}
\right],
\ t\ge1,$ where $\widetilde g_t=g_t^\star/2$ for $t<T_0$ and $\widetilde g_t=u_t(2)$ for $t\ge T_0$. 
Then ${\widetilde g_t}/{g_t^\star}\to1,$ so Condition~1 of Definition \ref{Def: Asym CS} holds. However, it follows that 
\[
\lim_{m\to\infty}
\P \left[
\forall t\ge m:\theta_0\in\widetilde C_t
\right]
=0.
\]
\end{IEg}
\noindent The two-part construction ensures that the boundary is well-defined at all time points, with $u_t(a)$ used only for $t\ge 16$, where all iterated logarithmic terms are positive. We provide the proof in Appendix \ref{Appendix: Proofs of Illustrative Examples}. It exploits the fact that the time-uniform crossing probability is sensitive to the second-order term of the law of the iterated logarithm (LIL) boundary, resulting in fundamentally different crossing behavior when changing from $u_t(4)$ to $u_t(2)$, a violation of Condition 2. This is known as the Kolmogorov--Erd\H{o}s--Petrowski--Feller LIL (EFKP-LIL) behavior (Chapter 5.2 of \cite{revesz2013random}). Moreover, the phenomenon illustrated by this example is neither specific to exact normality for i.i.d.\ sequences nor a time-series phenomenon, and continues to hold under a broad class of DGPs (\cite{aistleitner2011upper}; \cite{sasai2019erdHos}). 

Example \ref{IEg 2} thus motivates us to incorporate the eventual-inclusion condition in the definition of AsympCSs, stated as Condition~2 in Definition \ref{Def: Asym CS}, which we show in Appendix \ref{Appendix: Proof of coverage by eventual inclusion} implies the asymptotic time-uniform coverage property. We next formally present this coverage guarantee. 
 
\begin{definition}[Asymptotic Time-Uniform Coverage]
\label{Defn：Asymptotic time-uniform coverage}
Let $\{\widetilde C_t\}_{t\ge1}$ be a sequence of sets and 
let $\alpha\in(0,1)$ be the nominal miscoverage level.
We say that $\{\widetilde C_t\}_{t\ge1}$ has
\emph{asymptotic time-uniform $(1-\alpha)$-coverage} for $\theta_0$ if
\[
\liminf_{m\to\infty}
\mathbb{P}\!\left[
\forall t\ge m:\theta_0\in\widetilde C_t
\right]
\ge 1-\alpha.
\]
The coverage is said to be \emph{sharp} if the above inequality holds
with equality.
\end{definition}
\noindent $m$ here represents a practically meaningful burn-in point to have asymptotic approximations to ``kick in''. As argued in \cite{waudby2024time}, we do not refer to sequences merely satisfying Definition \ref{Defn：Asymptotic time-uniform coverage} as AsympCSs. In practice, the events $\{\forall t\ge m:\theta_0\in\widetilde C_t\}$ may not always arise from a triangular-array formulation with different starting times $m$ by having the sequence of sets to be $\widetilde C_t(m)$, as in delayed-start CSs (Equation 20 of \cite{robbins1970statistical}). These events may therefore be non-decreasing in $m$, in which case this criterion may reduce to trivial eventual coverage with the coverage probability approaching one if one starts to peek late enough. That said, the notion alone does not ensure reasonable finite-sample calibration at the nominal level, as the following example illustrates.
 
\begin{IEg}[Coverage and Finite-Sample Calibration] \label{IEg 1}
Let $X_1,X_2,\ldots$ be i.i.d.\ random variables with $X_t\in[0,1]$ and a finite mean $\theta_0:= \mu$. Define $\widehat{\theta}_t :=\bar X_t = t^{-1} \sum_{i=1}^{t} X_i.$ For any fixed $\alpha' \in (0,1)$, $\{C_t^\star\}_{t \ge 1}
= \left \{\left[
\bar X_t
\pm
\sqrt{
\log\!\left(\pi^2 t^2 /(3\alpha')\right)/2t
}
\right] \right\}_{t \ge 1}$ is an exact nonasymptotic $(1-\alpha')$-CS for $\theta_0$. However, asymptotically, it holds that 
\[
\lim_{m \to \infty}
\mathbb{P}
\!\left[
\forall t \ge m,\,
\theta_0 \in C_t^\star
\right]
= 1.
\]
\end{IEg}
\noindent The proof is provided in Appendix \ref{Appendix: Proofs of Illustrative Examples}. The key here is that $\alpha'$ is not required to be identical to the asymptotic significance level $\alpha$ stated in Definition \ref{Defn：Asymptotic time-uniform coverage}. In fact, choosing $\alpha'=0.5$ already suffices for this confidence sequence to satisfy the tail-validity criterion at arbitrarily stringent levels, even including $\alpha = 0$, while its underlying finite-sample time-uniform coverage remains only 50\%. Condition 1 in Definition \ref{Def: Asym CS} addresses this issue by requiring the underlying CS and the AsympCS to use the same $\alpha$ so that the approximation exhibits reasonable finite-sample behavior.

In light of these discussions, we conclude the section with a schematic figure summarizing the relationships among the original AsympCS, our refined AsympCS, and sequences merely satisfying asymptotic time-uniform coverage in Definition \ref{Defn：Asymptotic time-uniform coverage}.
\begin{figure}[H]
     \centering
     \includegraphics[width=0.4\textwidth]{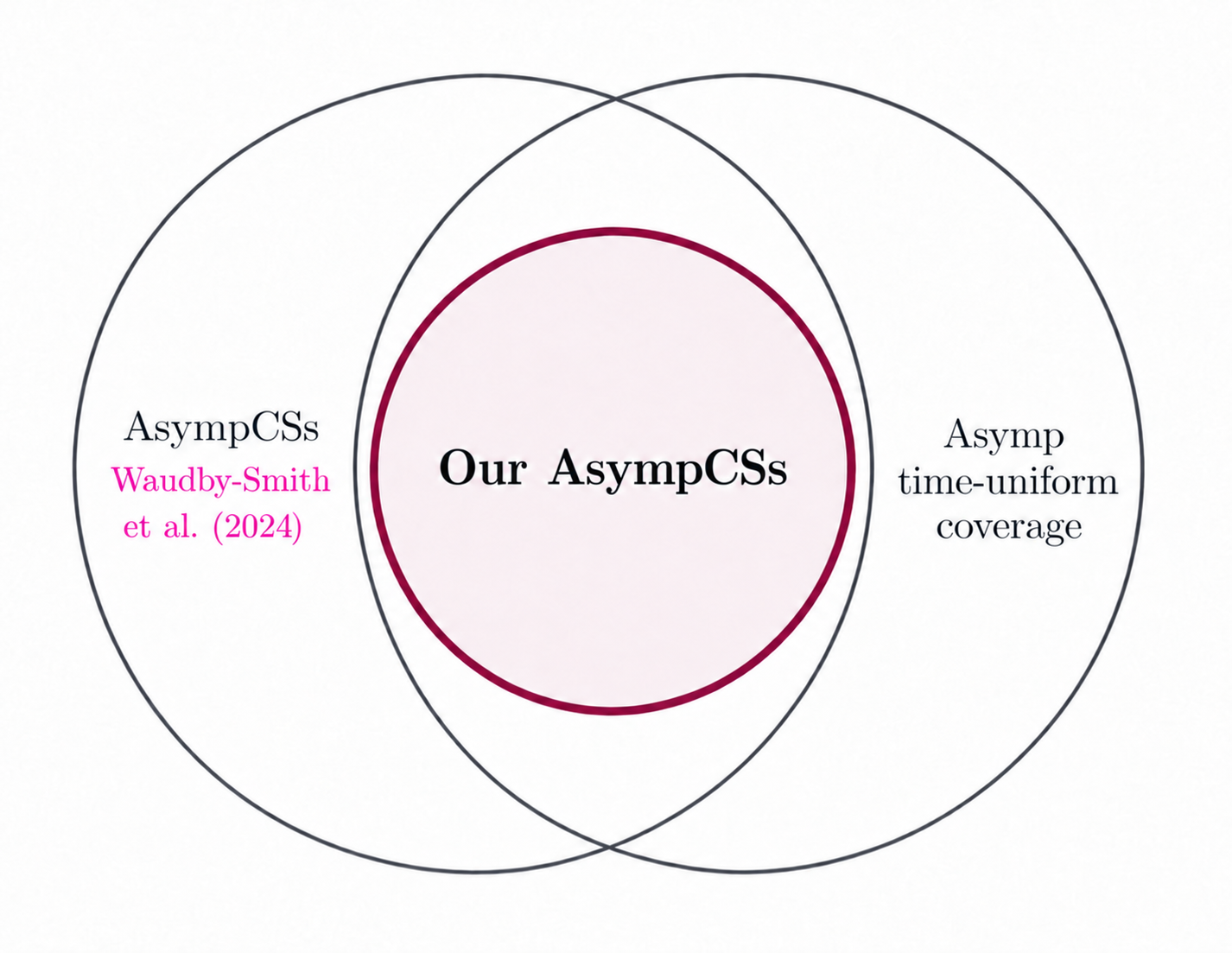}
     \caption{A schematic illustration of the relationships among three notions.}
     \label{fig:relationship}
 \end{figure}
\noindent Following Examples \ref{IEg 2}-\ref{IEg 1}, our Definition \ref{Def: Asym CS} lies at the intersection of the other two notions in Figure ~\ref{fig:relationship}. In this sense, we think that our refinement may more closely capture the original purpose of \cite{waudby2024time}. We thus hope that our refined notion will help applied researchers exercise appropriate care when applying AsympCSs in complex settings.

\subsection{Robbins' Gaussian Mixture AsympCSs for Nonstationary Time Series}  \label{Sec: Robbins Gaussian Mixture Asymp CSs for (Nonstationary) Time Series} 
\noindent In this section, we derive Robbins' Gaussian mixture asymptotic confidence sequences (AsympCSs), one of the most widely used constructions in the literature, for potentially nonstationary time series.  

For the mixture boundary to be nontrivial and meaningful, we assume the cumulative variance process $\V_t$ to be non-degenerate, which is also imposed in \cite{waudby2024time}'s Lindeberg- and Lyapunov-type AsympCSs, as well as in its related subsequent work. For simplicity\footnote{\label{fn:Vt order} For instance, one may relax the order condition to $\V_t/t^{2/p} \rightarrow \infty$ as $t \rightarrow \infty$ where the condition becomes weaker with higher order moments $p$. This unnecessarily obfuscates the proofs, so we impose a cleaner condition to give more interpretable results.}, we further require it to grow linearly in \(t\), i.e., it is bounded above and below by constant multiples of \(t\), where the upper bound requirement\footnote{$\V_t = O(t)$ is also an established and well-known result in alternative dependence frameworks, such as ones with (decaying) mixing conditions (e.g., Corollary 1.1 in Section 1.4 of \cite{rio2017asymptotic}).} is already implicitly satisfied given our dependency decay condition in Assumption \ref{Assump:dep_decay}. This is not very restrictive as it still covers a rather broad class of processes with regular accumulation of variance and it mainly rules out degenerate or explosive cases, in which many of the counterexamples arise from long-range dependent processes that we have already ruled out in Assumption \ref{Assump:dep_decay}. Also, a necessary but not sufficient condition for a non-zero long-run variance (LRV) to exist is $\V_t \asymp t$. Under these conditions, we have 

\begin{Thm}[Robbins' Gaussian Mixture AsympCSs] \label{Thm: Robbins Gaussian Mixture Asymp CS}
For the process $(Y_i)_{i\geq1}$, under the same notation and assumptions as in Theorem \ref{Thm: nonstationary ASIP}, suppose further that $\V_t \rightarrow \infty$ as $t \rightarrow \infty$ with $\V_t \asymp t$. Define the running average of the unconditional means as
$\widetilde{\mu}_t := t^{-1}\sum_{i=1}^t \mu_i$, with its strongly consistent sample analogue $\widehat{\mu}_t := t^{-1}\sum_{i=1}^t Y_i$. Assume that there exists a strongly consistent cumulative variance estimator $(\widehat{\V}_t)$ such that $\widehat{\V}_t/\V_t = 1 + \oas{1/\log t}.$ Then, for any prespecified constant $\rho >0$,
\[
\widetilde{C}^{\mathcal{G}}_t = \left(
\widehat{\mu}_t \pm
\sqrt{
\frac{\widehat{\V}_t \rho^2 + 1}{t^2\rho^2}
\log\!\left(\frac{\widehat{\V}_t \rho^2 + 1}{\alpha^2}\right)
}
\right)
\]
forms a $(1-\alpha)$-AsympCS for the unconditional running average estimand $\widetilde{\mu}_t$.
\end{Thm}
\noindent The mixture boundary offers flexibility through the choice of tuning parameter $\rho$. The optimal choice of $\rho$ is provided in Appendix \ref{Appendix: optimal rho}. The proof of Theorem \ref{Thm: Robbins Gaussian Mixture Asymp CS} is provided in Appendix \ref{Appendix: Proofs of Gaussian Mixture Asymp CS}. Our proof differs from standard tricks in that for nonstationary time series, Ville's inequality cannot be applied directly to $\Bro(\V_t)$ because $\V_t$ in Theorem~\ref{Thm: nonstationary ASIP} need not be non-decreasing, in which case the associated time-changed Brownian motion is generally not adapted to the canonical Brownian filtration in a manner compatible with the required martingale structure. Therefore, we proceed with upper bounding $\V_t$ by its running maximum, a deterministic non-decreasing sequence, so that we could apply Ville's inequality on the upper bound while preserving the asymptotic order of the boundary. Moreover, under the logarithmic rate assumption on variance estimation, which is weaker than the polynomial rate condition imposed by \cite{waudby2024time}, Condition 2 of Definition \ref{Def: Asym CS} holds. This relaxation can substantially ease restrictions on bandwidth and kernel selection. In particular, as will be shown in Section \ref{Sec: Strongly Consistent Estimation Strategies}, our condition imposes little additional restriction on these choices beyond what is needed for strong consistency alone.

Since assuming the existence of LRV is both appealing and common in practice, we provide an asymptotically equivalent $(1-\alpha)$-AsympCS. 

\begin{corollary}[Equivalent Robbins' Gaussian Mixture AsympCSs] \label{Corollary: Equivalent Asymp CS Boundary}
For the process $(Y_i)_{i\geq1}$, under the same notation and assumptions as in Theorem \ref{Thm: Robbins Gaussian Mixture Asymp CS}, further assume the LRV exists, i.e., $\sigma^2 := \lim_{t \rightarrow \infty} \sigma^2_t >0$, where $\sigma^2_t :=\V_t/t$. Suppose that $\widehat{\sigma}^2_t/\sigma_t^2 = 1 + \oas{1/\log t}$. Then, for any prespecified constant $\rho >0$,
\[
\widetilde{C}^{\mathcal{G}}_t = \left(
\widehat{\mu}_t \pm \widehat{\sigma}_t
\sqrt{
\frac{t  \rho^2 + 1}{t^2\rho^2}
\log\!\left(\frac{t\rho^2 + 1}{\alpha^2}\right)
}
\right)
\]
also forms a $(1-\alpha)$-AsympCS for the unconditional running average estimand $\widetilde{\mu}_t$. 
\end{corollary}
\noindent The proof is provided in Appendix \ref{Appendix: Proof of Corollary  Equivalent Asymp CS Boundary}. As in the i.i.d.\ setting of \cite{waudby2024time}, this formulation directly parallels the fixed-time CLT. Under regularity conditions, the classical time-series AsympCI takes the form $\left(
\widehat{\mu}_t \pm t^{-1/2}\widehat{\sigma}_t \cdot 
\Phi^{-1}(1-\alpha/2)
\right)$, while its time-uniform counterpart (AsympCS) in Corollary \ref{Corollary: Equivalent Asymp CS Boundary} incurs an additional logarithmic factor in $t$, reflecting the well-known $\log(t)$ penalty in sample sizes required to achieve time-uniformity. This formulation also facilitates application of the delta method for AsympCSs stated in Proposition 3.5 of \cite{waudby2024time}, as factoring out the square root of the estimated variance term enables the transformation of variance to be applied directly, separately from the boundary.

\subsection{Iterated Logarithm AsympCSs for Nonstationary Time Series} \label{Sec: Iterated Logarithm Asymp CSs for (Nonstationary) Time Series}
\noindent Unlike classical fixed-time CLT boundaries, which are typically obtained from normal quantiles, in this anytime-valid regime, virtually any boundary for Gaussian observations can be used to derive an AsympCS as long as the appropriate ASIP can be applied under the given assumptions. Among them, we present the law-of-the-iterated-logarithm (LIL) boundary for completeness because it attains the asymptotically optimal width $O\left(\sqrt{\log\log t/t}\right)$, whereas the current asymptotic width in Robbins' Gaussian mixture boundary in Corollary \ref{Corollary: Equivalent Asymp CS Boundary} is $O\!\left(\sqrt{\log t/t}\right)$.

To the best of our knowledge, there is no established LIL result in the time-series literature for nonstationary processes. In particular, the implied LIL rate for the strong law of large numbers (SLLN) under the $t^{-1}$ normalization does not appear to have been established yet. As a direct consequence of the nonstationary ASIP in Theorem \ref{Thm: nonstationary ASIP}, we state related results as a corollary below.

\begin{corollary}[SLLN with an LIL rate]  \label{Corollary: SLLN with an LIL rate}For the process $(Y_i)_{i\geq1}$, under the same notation and assumptions as in Corollary \ref{Corollary: Equivalent Asymp CS Boundary}, let $\mathsf{S}_t:=\sum_{i=1}^t (Y_i-\mu_i).$
Then 
\[
\limsup_{t\to\infty}
\frac{\mathsf{S}_t /t}{\sqrt{2\V_t\log\log \V_t} /t }
=1,
\qquad
\liminf_{t\to\infty}
\frac{\mathsf{S}_t /t}{\sqrt{2\V_t\log\log \V_t} /t}
=-1
\qquad \mathrm{a.s.}
\]
Equivalently,
\[
\limsup_{t\to\infty}
\frac{\mathsf{S}_t /t}{ \sigma \sqrt{2t\log\log t} /t}
=1,
\qquad
\liminf_{t\to\infty}
\frac{\mathsf{S}_t /t}{\sigma \sqrt{2t\log\log t} /t}
=-1
\qquad \mathrm{a.s.}
\]
In particular,
\[
\limsup_{t\to\infty}
\frac{|\mathsf{S}_t| /t}{ \sqrt{2\V_t\log\log \V_t} /t}
=1
\qquad \mathrm{a.s.}
\]
\end{corollary}
\noindent We remark that $\V_t \to \infty$ alone in Theorem \ref{Thm: nonstationary ASIP} does not guarantee that we observe the Brownian motion densely enough in variance time. If the sequence $\V_t$ grows too rapidly, we may completely miss the Brownian LIL peaks. Therefore, $\V_t \asymp t $ in Theorem \ref{Thm: Robbins Gaussian Mixture Asymp CS} is a convenient sufficient condition for characterizing the LIL-rate behavior of the nonstationary process. 

\begin{Thm}[Iterated Logarithm AsympCSs] \label{Thm: Iterated Logarithm Asymp CSs}
Under the same conditions as Corollary \ref{Corollary: Equivalent Asymp CS Boundary},
\[
\widetilde{C}^{\mathcal{L}} _t 
:= \left(
\widehat{\mu}_t \pm \widehat{\sigma}_t \cdot 1.7 \sqrt{
\frac{\log \log (2t) + 0.72 \log(10.4/\alpha)}{t}
}
\right)
\]
forms a $(1 - \alpha)$-AsympCS for $\widetilde{\mu}_t$.     
\end{Thm}
\noindent We omit the proof as it proceeds in a similar fashion to that of Theorem \ref{Thm: Robbins Gaussian Mixture Asymp CS}. By construction from Theorem 1 of \cite{howard2021time}, the implemented boundary in Theorem \ref{Thm: Iterated Logarithm Asymp CSs} is already a slight asymptotic inflation of the exact stitched boundary, providing sufficient slack for the Condition 2 of Definition \ref{Def: Asym CS}.

As discussed in Section 3.5 of \cite{howard2021time} and Section 2.2.4 of \cite{waudby2024time}, mixture boundaries such as the one in Theorem \ref{Thm: Robbins Gaussian Mixture Asymp CS} may be preferable in practice, because constructing an anytime-valid boundary is essentially an $\alpha$-budget allocation problem in the sense that any bound that is tighter ``later on'' asymptotically must be looser  ``early on'' at practical sample sizes to account for the use of a smaller portion of the $\alpha$ budget for potential miscoverage. That said, although we present the iterated logarithm AsympCSs, we recommend Theorem \ref{Thm: Robbins Gaussian Mixture Asymp CS} or Corollary \ref{Corollary: Equivalent Asymp CS Boundary} over Theorem \ref{Thm: Iterated Logarithm Asymp CSs} in practice.

\subsection{Practical Implementation of AsympCSs for Nonstationary Time Series} \label{Sec: Strongly Consistent Estimation Strategies}
\noindent Thus far, we have assumed the existence of strongly consistent estimators. We now provide explicit constructions for the target parameters in Section \ref{Sec: Robbins Gaussian Mixture Asymp CSs for (Nonstationary) Time Series} and Section \ref{Sec: Iterated Logarithm Asymp CSs for (Nonstationary) Time Series}.
 
For the process $(Y_i)_{i\geq1}$, under the same notation and process assumptions as in Theorem \ref{Thm: Robbins Gaussian Mixture Asymp CS}, the sample analogue already provides a strongly consistent estimator of the unconditional running mean, i.e., 
\begin{equation} \label{Eq: unconditioning running mean estimator}
\widehat{\mu}_t := \frac{1}{t}\sum_{i=1}^{t} Y_i =\widetilde{\mu}_t + O_{\mathrm{a.s.}}\left(\sqrt{\frac{\log\log t}{t}}\right),   
\end{equation}
where the stated rate follows directly from the SLLN result in Corollary \ref{Corollary: SLLN with an LIL rate}. 

However, unlike in i.i.d.\ or martingale settings, variance estimation for time series involves long-run variance (LRV) and often admits various choices. One of the most commonly adopted estimators is the heteroskedasticity and autocorrelation consistent (HAC) estimator. 
While strong consistency of HAC estimators has been established under mixing conditions (\cite{de2000strong}), to the best of our knowledge, no such result is available under the functional dependence framework\footnote{For readers familiar with this literature, Chapter 8 of the Handbook of Statistics (\cite{wu2012covariance}) provides related results on HAC estimation, but does not establish the strong consistency. We also note a few minor typographical issues in their stated results, e.g., the unnormalized HAC estimator in Equation (38) of Theorem 6 should not scale only with bandwidth.}. In fact, there appears to be no strongly consistent estimator of LRV for nonstationary time series. Moreover, the HAC estimator studied in \cite{de2000strong} is constructed directly on a centered nonstationary process. In contrast, our setting involves a running mean structure, which induces additional technical difficulties. Therefore, under suitable regularity conditions, we first prove strong consistency of the oracle HAC estimator, namely on a centered nonstationary process under the functional dependence framework, and then we prove strong consistency of the plug-in mean HAC estimator on the uncentered process. We also verify the specific rate needed to obtain the asymptotic time-uniform coverage.

\begin{condition}[Conditions for $\oas{1}$-consistent Oracle HAC Estimation]\label{Cond: primitive-bandwidth-rate}
Let the kernel $K$ be an even, bounded function supported on $[-1,1]$, satisfying $K(0)=1, K(u)=0$ whenever $|u| >1$, and continuous at zero, i.e., $\lim_{u\to 0}K(u)=K(0)$. Furthermore, let the bandwidth $b_t$ be a deterministic non-decreasing bandwidth satisfying $b_t\to\infty$ and $b_t=O\left(t^{\,1-2/(p \wedge 4)}(\log t)^{-\lambda}\right)$, for some $\lambda>2/p$.
\end{condition}
\noindent Condition \ref{Cond: primitive-bandwidth-rate} is used to construct the strong consistency result for the oracle HAC estimator so that it helps us compare it with the counterpart used under mixing conditions. Our kernel condition is broadly in line with those commonly adopted in the literature (\cite{de2000strong}; \cite{wu2012covariance}), and permits many popular choices such as the Bartlett, Bohman, Epanechnikov, Parzen, quartic, truncated, and Tukey--Hanning kernels. The bandwidth conditions require $b_t\to\infty$, so that increasingly many autocovariance lags are incorporated, while the upper restrictions prevent $b_t$ from growing too rapidly, as sample autocovariances at larger lags are based on fewer observations.  The additional $\log$ term of $b_t$ is to get rid of the boundary case when ensuring the summability condition to apply the Borel-Cantelli lemma. Our bandwidth conditions are substantially weaker, by a polynomial factor up to logarithmic terms, than their counterparts under mixing conditions in \cite{de2000strong}. 

\begin{Rmk}[A subset relation for admissible bandwidth classes] \label{Rmk: A subset relation for admissible bandwidth classes}
Let
\[
\mathcal{B}^{\mathrm{FDM}}_1(\lambda)
=
\left\{
b_t:
b_t \to \infty,\;
b_t = O\!\left(t^{1-2/(p \wedge 4)}(\log t)^{-\lambda}\right)
\right\},
\qquad \lambda>\frac{2}{p},
\]
and
\[
\mathcal{B}^{\mathrm{Mixing}}_2(\gamma)
=
\left\{
b_t:
b_t \to \infty,\;
b_t = O\!\left(t^{1/2-1/(p \wedge 4)}
(\log t)^{-1/(p \wedge 4)-\gamma}\right)
\right\},
\qquad \gamma >0.
\]
Then, for any fixed finite $\lambda>2/p$ and any fixed finite $\gamma >0$, $\mathcal{B}^{\mathrm{Mixing}}_2(\gamma) \subset \mathcal{B}^{\mathrm{FDM}}_1(\lambda).$ 
\end{Rmk}
\noindent For comparison, $\mathcal{B}^{\mathrm{Mixing}}_2(\gamma)$ is adapted from Assumption 1 of \cite{de2000strong} by setting $r:=p\wedge 4$, which yields an upper envelope of the admissible bandwidth classes under the original condition. The proof is omitted because the result follows by just comparing the leading polynomial orders, in which the inclusion is strict as one could easily find a bandwidth $b_t \in \mathcal{B}^{\mathrm{FDM}}_1(\lambda)$ but $b_t \notin \mathcal{B}^{\mathrm{Mixing}}_2(\gamma)$. Technically, higher-order $p$ may also improve the polynomial rate of the admissible bandwidth and the truncation of the effective moment order at four may be relaxed. However, Condition \ref{Cond: primitive-bandwidth-rate} already accommodates most practical choices, so we do not pursue this direction further. In particular, it recovers standard bandwidth choices such as the conventional $b_t \asymp t^{1/3}$ rate associated with the Bartlett kernel under finite fourth moments, which were previously unavailable under \cite{de2000strong}. We provide further details and some comparisons in related contexts for interested readers in Appendix \ref{Appendix: Comparison of Admissible Bandwidth Classes}. 
 
We now state conditions under which the HAC estimator is strongly consistent at the specified rate used to establish the asymptotic time-uniform coverage in Theorem \ref{Thm: Robbins Gaussian Mixture Asymp CS}.
\begin{condition}[Conditions for $\oas{1/\log t}$-consistent Oracle HAC Estimation] \label{Cond: o(1/logt) oracle HAC bandwidth}
In addition to Condition \ref{Cond: primitive-bandwidth-rate}, further suppose that there exists some $\kappa \in (0,A)$ with $A$ defined in Assumption \ref{Assump:dep_decay} such that the kernel satisfies $|1-K(u)| = O(|u|^\kappa)
\ \text{as }u\to 0,$ 
and the bandwidth satisfies ${b_t^\kappa}/\log t\to\infty$ along with {$b_t = O\left(t^{\,1-2/(p \wedge 4)} (\log t)^{-\lambda} \right)$, for some $\lambda> 1 + 2/p$.}
\end{condition}
\noindent Condition \ref{Cond: o(1/logt) oracle HAC bandwidth} is slightly more restrictive than Condition \ref{Cond: primitive-bandwidth-rate}, but we do not view these additional assumptions as practically restrictive. The local smoothness condition admits all commonly used HAC kernels discussed above, and the additional bandwidth restrictions are essentially innocuous, as practitioners typically adopt polynomial-rate bandwidth sequences. However, this flexibility in kernel and bandwidth selection is generally unavailable under the stronger polynomial-rate convergence requirement imposed by \cite{waudby2024time}.

With Condition \ref{Cond: primitive-bandwidth-rate} and Condition \ref{Cond: o(1/logt) oracle HAC bandwidth}, we are now ready to present the strong consistency result for the oracle HAC estimator constructed from the centered  nonstationary process. Formally, we have 
\begin{Prop}[Strong Consistency of Oracle HAC Estimator] \label{Prop: Strong Consistency of Oracle HAC Estimator}
Let  $(X_i)_{i \ge 1}$ satisfy Equation \eqref{eq:g_function} with $\E[X_i] = 0$. Under the same notation and process assumptions as in Corollary \ref{Corollary: Equivalent Asymp CS Boundary}, suppose that the kernel $K$ and the bandwidth $b_t$ satisfy Condition \ref{Cond: primitive-bandwidth-rate}, and define the oracle HAC estimator as 
\[
\widetilde{\sigma}_t^2
:=
\frac{1}{t}
\sum_{i=1}^t
\sum_{j=1}^t
K\!\left(\frac{i-j}{b_t}\right)
X_i X_j. 
\]
Then it follows that $\widetilde{\sigma}_t^2 \xrightarrow{\mathrm{a.s.}} \sigma^2$. If, in addition, 
Condition \ref{Cond: o(1/logt) oracle HAC bandwidth} holds, then $\widetilde{\sigma}^2_t/\sigma_t^2 = 1 + \oas{1/\log t}$ so that the rate condition assumed in Corollary \ref{Corollary: Equivalent Asymp CS Boundary} is satisfied.
\end{Prop}
\noindent The proof is provided in Appendix \ref{Proof of Prop: A.S. Oracle HAC Estimator}.

To establish the strong consistency of the plug-in residual-based HAC estimator, we impose a Ces\`aro mean-stability condition as briefly discussed in the beginning of Section \ref{Sec: Almost Sure Invariance Principles for (Nonstationary) Time Series}. The following condition ensures that the unconditional running mean estimator $\widehat{\mu}_t$ defined in Equation \eqref{Eq: unconditioning running mean estimator} can be directly employed without affecting the asymptotic behavior of the HAC estimator.  

\begin{condition}[Ces\`aro Mean-Stability Condition] \label{Cond: Ces\`aro Mean-Stability Condition}
Under Equation \eqref{eq:additive_decomposition}, for a deterministic sequence $c_t \to \infty$, suppose that $t^{-1} \sum_{i=1}^{t}(\mu_i-\widetilde{\mu}_t)^2  = o(c_t^{-1}),$ 
where $\widetilde{\mu}_t$ is defined in Theorem \ref{Thm: Robbins Gaussian Mixture Asymp CS}. 
\end{condition}
\noindent A similar Ces\`aro mean-stability condition appears in related settings, for example in Assumption 3 of \cite{johari2025estimation} and Equation 13 of \cite{mathis2026confidence}. In words, Condition \ref{Cond: Ces\`aro Mean-Stability Condition} allows the individual means $\mu_i$ to vary with $i$, but requires their average squared fluctuation from the time-varying mean to be asymptotically negligible relative to $c_t^{-1}$, a factor of the bandwidth scale specified in Theorem \ref{Thm: Strong Consistency of Plug-in HAC Estimator} below. It mainly rules out persistent mean nonstationarity, such as deterministic trends or structural breaks occupying a non-vanishing fraction of the sample. 

By leveraging Proposition \ref{Prop: Strong Consistency of Oracle HAC Estimator}, we present the strong consistency result for the ultimate plug-in residual-based HAC estimator as below. 
\begin{Thm}[Strong Consistency of Plug-in HAC Estimator]  \label{Thm: Strong Consistency of Plug-in HAC Estimator}
For the process $(Y_i)_{i\geq1}$ satisfying Equation \eqref{eq:additive_decomposition}, under the same notation and process assumptions as in Corollary \ref{Corollary: Equivalent Asymp CS Boundary}, suppose that the conditions of Proposition \ref{Prop: Strong Consistency of Oracle HAC Estimator} hold. Further suppose Condition \ref{Cond: Ces\`aro Mean-Stability Condition} holds with $c_t = b_t$, and assume that the kernel $K$ is positive semi-definite (PSD). Define $$\widehat{\sigma}_t^2
:=
t^{-1}
\sum_{i=1}^t
\sum_{j=1}^t
K\!\left(\frac{i-j}{b_t}\right)
\widehat X_{i,t} \widehat X_{j,t}, \quad \widehat{X}_{i,t} : = Y_i - \widehat{\mu}_t,$$ where $\widehat{\mu}_t$ is defined in Equation \eqref{Eq: unconditioning running mean estimator}. Then $\widehat{\sigma}_t^2 - \widetilde \sigma_t^2 \xrightarrow{\mathrm{a.s.}} 0,$ and by Proposition \ref{Prop: Strong Consistency of Oracle HAC Estimator}, it follows that 
$\widehat{\sigma}_t^2 \xrightarrow{\mathrm{a.s.}}  \sigma^2.$ In addition, if Condition \ref{Cond: Ces\`aro Mean-Stability Condition} holds with $c_t = b_t (\log t)^2$, the estimator satisfies the rate condition assumed in Corollary \ref{Corollary: Equivalent Asymp CS Boundary}.
\end{Thm} 
\noindent The proof is provided in Appendix \ref{Proof of Thm: A.S. Plug-in HAC Estimator}. The additional PSD requirement on the kernel is only for technical convenience and could be easily relaxed. In particular, it still accommodates many standard kernel choices such as the Bartlett and Parzen kernels, whose relatively smooth weighting properties tend to improve finite-sample stability.

\section{Illustrations: Applications to Time-Series Data} \label{Sec: Illustrations: Applications to Time Series Data}
\noindent This section illustrates the use of the time-uniform dependent CLT (AsympCSs) developed in Section~\ref{Sec: Main Results} for two applications. Section~\ref{Sec: Sampling and Optimization Algorithms} applies the result to Markovian processes and Section~\ref{Sec: Illustration: Time-series Based Causal Estimands} considers time-series-based causal estimands.

\subsection{Illustration: Algorithms with Markovian Structure} \label{Sec: Sampling and Optimization Algorithms}
\noindent Under suitable conditions, sampling and optimization algorithms such as MCMC and online SGD admit contractive Markovian recursions satisfying Assumptions~\ref{Assump:dep_decay} and ~\ref{Assump:UI}. We use MCMC as the leading example and construct AsympCSs for ergodic averages in Proposition~\ref{Prop: Asymp CS for MCMC averages}. Constant-step-size SGD also falls within the same Markovian framework (\cite{dieuleveut2020bridging}; \cite{li2024sgd}; \cite{merad2025convergence}; \cite{wei2025online}), which we defer to Appendix~\ref{Appendix: Stochastic Gradient Descent} given the close parallel.

In MCMC, a central objective is to estimate the ergodic average for a target distribution $\pi$ and terminate sampling once the estimate appears sufficiently precise. The classical fixed-width stopping rules for MCMC have been developed in the fixed-time literature (\cite{jones2006fixed}; \cite{flegal2015relative}), and we further establish their time-uniform guarantees in the following Proposition \ref{Prop: Asymp CS for MCMC averages}. 

For generality, we consider a multivariate time-homogeneous Markov chain on $\mathbb{R}^d$ with invariant distribution $\pi$, denoted as $(\bm{X}_i)_{i \geq 0}$. Our focus, however, remains on consstructing univariate AsympCSs for a scalar. Specifically, let $f:\R^d \to \mathbb R$ be a Lipschitz function. The estimand of interest is $\mu := \E_\pi \left[f(\bm{X})\right]$, which may be interpreted as a posterior mean in the Bayesian context. 

\begin{Prop}[AsympCS for MCMC averages] \label{Prop: Asymp CS for MCMC averages}
  Suppose the chain evolves as $\bm{X}_i = G(\bm{X}_{i-1}, \bm{\varepsilon}_i)$, where
$(\bm{\varepsilon}_i)_{i \in \mathbb{Z}}$ are i.i.d.\ innovations and $G$ is a 
measurable transition map. Assume that for some $p > 2 $, we have
\begin{align} 
    L_p :=\sup_{\bm{x} \neq \bm{y}}
  \frac{\|G(\bm{x},\bm{\varepsilon}) - G(\bm{y},\bm{\varepsilon})\|_p}{\norm{\bm{x} - \bm{y}}} < 1 \label{eq:contraction_map},
\end{align}
and for some $\bm{x}_0$, 
\begin{align}
\norm{G(\bm{x}_0, \bm{\varepsilon})}_p < \infty \label{eq: initial moment condition}. 
\end{align}
Let $f:\R^d \to \mathbb R$ be Lipschitz, and define the estimand of interest and the corresponding estimator as $\mu := \E_\pi \left[f(\bm{X})\right]$ and $
\widehat\mu_t=t^{-1}\sum_{i=1}^t f(\bm{X}_i)$ respectively. Then, under the same LRV conditions as in Corollary \ref{Corollary: Equivalent Asymp CS Boundary}, for any prespecified constant $\rho >0$, 
\[
\widetilde{C}^{\mathcal{G}}_t = \left(
\widehat{\mu}_t \pm \widehat{\sigma}_t
\sqrt{
\frac{t  \rho^2 + 1}{t^2\rho^2}
\log\!\left(\frac{t\rho^2 + 1}{\alpha^2}\right)
}
\right)
\]
forms a $(1-\alpha)$-AsympCS for the estimand of interest $\mu$, where $\widehat{\sigma}_t$ is the square root of the HAC estimator on the process of $Y_i := f(\bm{X}_i)$ provided in Theorem \ref{Thm: Strong Consistency of Plug-in HAC Estimator}. 
\end{Prop}
\noindent The proof is provided in Appendix \ref{Proof of Prop: Asymp CS for MCMC averages}, in which the key idea is that the centered process $Y_i - \mu$ admits the representation as in Equation \eqref{eq:g_function}, and Assumptions \ref{Assump:dep_decay}-\ref{Assump:UI} could be shown to be satisfied via the iterated random map theory.


Proposition~\ref{Prop: Asymp CS for MCMC averages} is stated for a general time-homogeneous Markov chain with invariant distribution $\pi$, but could be easily extended to time-inhomogeneous chains, allowing for nonstationarity induced by initialization provided that the associated initialization error is asymptotically negligible. For specific MCMC algorithms, the required dependency condition can often be established from regularity conditions on the target distribution. We illustrate this with the unadjusted Langevin algorithm (ULA), a discrete approximation of the Langevin diffusion process (\cite{durmus2017nonasymptotic, dalalyan2017theoretical}). 
\begin{Eg}[Unadjusted Langevin Algorithm (ULA)] \label{Example: ULA}
\begin{align*}
    \bm{X}_{i+1} = \bm{X}_i - \eta\nabla U(\bm{X}_i) + \sqrt{2\eta}\,\bm{\xi}_i,
  \quad \bm{\xi}_i \overset{\mathrm{i.i.d.}}{\sim} \mathcal{N}(0, \bm{I}_d),
\end{align*}
where $U = -\log\pi$ and $\eta > 0$; further suppose the target distribution $\pi$ is $m$-strongly log-concave and $L$-smooth, so that $m\bm{I} \preceq \nabla^2 U \preceq L\bm{I}$. 
\end{Eg} 
\noindent For this Example \ref{Example: ULA}, since
$G(\bm{x},\bm{\varepsilon}) - G(\bm{y},\bm{\varepsilon}) = (\bm{x} - \bm{y}) - \eta(\nabla U(\bm{x}) - \nabla U(\bm{y}))$ is
independent of~$\bm{\varepsilon}$, the contraction is deterministic. For
$\eta \in (0,\, 2m/L^2)$, standard convexity-smoothness inequalities give
\begin{align*}
   |G(\bm{x},\bm{\varepsilon}) - G(\bm{y},\bm{\varepsilon})| \leq\kappa\norm{\bm{x} - \bm{y}}, \quad
\kappa :=\sqrt{1 - 2\eta m + \eta^2 L^2} < 1. 
\end{align*}
Consequently, $\delta_p(h) \leq C\kappa^h$ and
$\Theta_{i,p} \leq C'\kappa^i$, so
$\sup_{i \geq 0}(i+1)^A\Theta_{i,p} < \infty$ for every $A > 1$, and Assumption \ref{Assump:dep_decay} is satisfied.  ULA is an approximate sampler, but similar geometric stability can also hold for exact samplers such as the Metropolis-Langevin algorithm under comparable conditions.

\subsection{Illustration: Time-series-Based Causal Estimands} \label{Sec: Illustration: Time-series Based Causal Estimands}
\noindent We study time-series experiments under a super-population framework, where one aggregate unit is sequentially randomized into treatment or control at each time. These designs are particularly relevant when only a small number of experimental units are available or when cross-sectional interference is too complex to model explicitly, such as in online marketplaces.

Let $Y_i \in \mathbb{R}$ and $D_i \in \{0,1\}$ denote the observed outcome and treatment assignment, respectively, for the aggregate unit in period $i$. Let $Y_i(1)$ and $Y_i(0)$ denote the corresponding potential outcomes under treatment and control, respectively. For simplicity of exposition, we assume the stable unit treatment value assumption holds so that the observed outcome satisfies $Y_i = D_i Y_i(1) + (1-D_i)Y_i(0)$ as usual (\cite{rubin1980randomization}). We thus focus on the 
canonical average treatment effect (ATE): 
\begin{equation} \label{Eq: ATE in RCT}
\E\!\left[Y_i(1)-Y_i(0)\right] = \tau,   
\end{equation}
where the underlying expectation is with respect to the stochastic evolution of the aggregate unit. 

In the presence of carryover effects, one may instead consider the contemporaneous ATE with potential outcomes $Y_i(D_{1:i-1},1)$ and $Y_i(D_{1:i-1},0)$ evaluated on the observed treatment path $D_{1:i-1}$ (\cite{ham2022design}), or the lag-$k$ ATEs that compare treatment histories over the most recent $k+1$ periods (\cite{bojinov2019time}; \cite{lin2025unifying}). Alternatively, one may study ATEs with the always-treated and always-control regimes, represented by $Y_i(1,\ldots,1)$ and $Y_i(0,\ldots,0)$, under Markovian dynamics (\cite{glynn2020adaptive}; \cite{farias2022markovian}; \cite{hu2022switchback}; \cite{johari2025estimation}). Our results can be extended to both settings, with the latter following the approach in Section \ref{Sec: Sampling and Optimization Algorithms}. We also acknowledge that the stability of the ATE in Equation \eqref{Eq: ATE in RCT} may be restrictive in settings with dynamic behavioral responses such as novelty effects or changing environments. Nevertheless, it does not require stationarity of the observed outcomes, and time-varying ATEs can be accommodated with additional modeling, in which case the resulting low-frequency contamination issue can be addressed under Condition \ref{Cond: Ces\`aro Mean-Stability Condition}. In particular, such a stable effect is not rare in experiments (\cite{kohavi2012trustworthy}; \cite{sadeghi2022novelty}), and industry evidence from Microsoft Research found that treatment effects still appeared to be considerably more stable than the underlying outcome levels even during the COVID-19 pandemic (\cite{townsend2020ab}). Moreover, in time-series experiments with aggregate outcomes, interventions inducing drastic changes to the platform, particularly those that generate pronounced novelty or learning effects at the aggregate level, may be of limited practical relevance.

To make progress with the super-population analysis, we impose conditions on the temporal dynamics of the potential outcomes, allowing for both nonstationarity and serial dependence. 

\begin{Ass}[Potential-Outcome Mean Decomposition] \label{Ass: RCT's nonstationary potential outcomes}
For each $i \geq 1$ and $d \in \{0,1\}$, suppose
\[
    Y_i(d) = m_i + \tau d + X_i(d),
    \qquad
    \E[X_i(d)] = 0,
\]
where $\{m_i\}_{i\geq 1}$ is a deterministic sequence and $X_i(d)$ denotes the centered stochastic component.
\end{Ass}
\noindent Assumption \ref{Ass: RCT's nonstationary potential outcomes} is meant to be a potential-outcome analogue of the standard reparametrization of time series in Equation \eqref{eq:additive_decomposition}, with the additional decomposition $\mu_i(d) := m_i + \tau d$. Here, $m_i$ absorbs temporal mean variation unrelated to treatment, including potentially nonstationary common shocks, while $\tau$ represents the treatment-induced effects. $X_i(d)$ may exhibit serial dependence and nonstationarity, while allowing for both treatment-specific stochastic variation and heterogeneity in treatment responses. Indeed, $Y_i(1)-Y_i(0) = \tau + X_i(1)-X_i(0),$ so that the constant $\tau$ does not imply homogeneous realized treatment effects.

We then specify the treatment assignment mechanism.
\begin{Ass}[Treatment Assignment] \label{Ass: RCT's Treatment assignment}
Let $D_i \in \{0,1\}$ denote the treatment assignment. Assume sequential randomization, i.e., $D_i \perp \left(Y_i(1),Y_i(0)\right) \mid
\mathcal{F}_{i-1},$ where $\mathcal{F}_{i-1}$ represents the canonical filtration of the observed history prior to the treatment assignment in the $i$-th period. Also, assume a strong overlap condition that $\P[D_i=1\mid\mathcal{F}_{i-1}]=p_i,\ \underline p \leq p_i \leq 1-\underline p$, 
for some $\underline p>0$.
\end{Ass}
\noindent This condition accommodates general sequential randomization schemes, allowing the treatment probability to vary over time and adapt to the observed history. Given the treatment assignment, we consider the inverse-propensity score weighting (IPW) estimator:
\begin{align} \label{Eq: RCT's IPW}
\widehat{\tau}_{\mathrm{IPW},t}
:=
\frac{1}{t}\sum_{i=1}^t \phi_i
=
\frac{1}{t}\sum_{i=1}^t
\left(
\frac{D_iY_i}{p_i}
-
\frac{(1-D_i)Y_i}{1-p_i}
\right),    
\end{align}
with $\phi_i$ as the observed IPW pseudo-outcome, i.e., the noncentered IPW estimating score, and $p_i$ as the treatment assignment probability at time $i$ defined in Assumption \ref{Ass: RCT's Treatment assignment}.

We now construct the corresponding AsympCS for the ATE estimand as follows.
\begin{Prop}[AsympCS for the ATE in Time-series-Based Experiments] \label{Prop: Asymp CS for RCT IPW}
Suppose Assumptions \ref{Ass: RCT's nonstationary potential outcomes}-\ref{Ass: RCT's Treatment assignment} hold. Suppose further that the noncentered IPW estimating score $\phi_i$ in Equation \eqref{Eq: RCT's IPW} satisfies the assumptions required by Corollary \ref{Corollary: Equivalent Asymp CS Boundary}. Then, for any prespecified constant $\rho >0$, 
\[
\widetilde{C}^{\mathcal{G}}_t = \left(
\widehat{\tau}_{\mathrm{IPW},t} \pm \widehat{\sigma}^{\mathrm{HAC}}_{\mathrm{IPW},t}
\sqrt{
\frac{t  \rho^2 + 1}{t^2\rho^2}
\log\!\left(\frac{t\rho^2 + 1}{\alpha^2}\right)
}
\right)
\]
forms a $(1-\alpha)$-AsympCS for $\tau$, where $\widehat{\sigma}^{\mathrm{HAC}}_{\mathrm{IPW},t}$ is the square root of the HAC estimator, provided in Theorem \ref{Thm: Strong Consistency of Plug-in HAC Estimator},  applied to the IPW estimating score process $\{\phi_i\}_{i\geq 1}$. 
\end{Prop}
\noindent The proof of Proposition \ref{Prop: Asymp CS for RCT IPW} is omitted as it follows from Corollary \ref{Corollary: Equivalent Asymp CS Boundary}. We remark that even in this simplified setting, the IPW estimating score process remains nonstationary because $D_iY_i = D_iY_i(1)$ by consistency and the potential outcome $Y_i(1)$ exhibits potential nonstationarity as in Assumption \ref{Ass: RCT's nonstationary potential outcomes}. That said, it generally cannot be reduced to a stationary process, nor does it satisfy a martingale difference structure, underscoring the importance and broad applicability of our nonstationary framework. We leave other complex models, such as semiparametric causal models for time series (\cite{angrist2011causal}; \cite{angrist2018semiparametric}), for future work.

\section{Simulation Studies} \label{Sec: Simulation Studies}
\noindent In this section, we evaluate the finite-sample performance of the AsympCSs through simulations. Since, to the best of our knowledge, no nonasymptotic CSs are available for comparison with AsympCSs under our general time-series assumptions, we use the CLT-based AsympCIs as a baseline. We then conduct simulations using the equivalent AsympCS boundary established in Corollary~\ref{Corollary: Equivalent Asymp CS Boundary} for a more direct comparison. Simulations using the boundary in Theorem~\ref{Thm: Robbins Gaussian Mixture Asymp CS} yield similar results and are therefore not reported separately.

Figure \ref{fig:sim_kernels} considers four data-generating processes (DGPs). The first two DGPs are adapted from \cite{bonnerjee2024gaussian}. Specifically, the first is a linear process $Y_t = 0.25Y_{t-1} + \sqrt{2/3}U_t$, where $U_t$ are independent $t_6$ innovations, with $t_6$ denoting the Student's $t$-distribution with six degrees of freedom. The second is a nonlinear process with periodic dependence, $Y_t=\sin(Z_t)$, where $Z_t = a_t Z_{t-1}+\eta_t$, $a_t = 0.2 + 0.1\sin(2\pi t/100)$, and $\eta_t$ are independent standard normal innovations. To examine robustness to strong dependence, the third and fourth processes that we consider are stationary Gaussian AR(1) processes with coefficients $0.9$ and $-0.9$, respectively. All four processes have mean zero. For each of the DGPs, we report the average width of the intervals and the cumulative miscoverage rate, defined as the fraction of simulated paths among Monte Carlo replicates on which the target is excluded at least once after a burn-in period $m$ to allow the asymptotic approximation as in Definition \ref{Defn：Asymptotic time-uniform coverage}. We optimize the AsympCS boundaries by tuning $\rho$ as described in Appendix \ref{Appendix: optimal rho}. 

\begin{figure}[H]
\centering  
\includegraphics[width=0.55\textwidth]{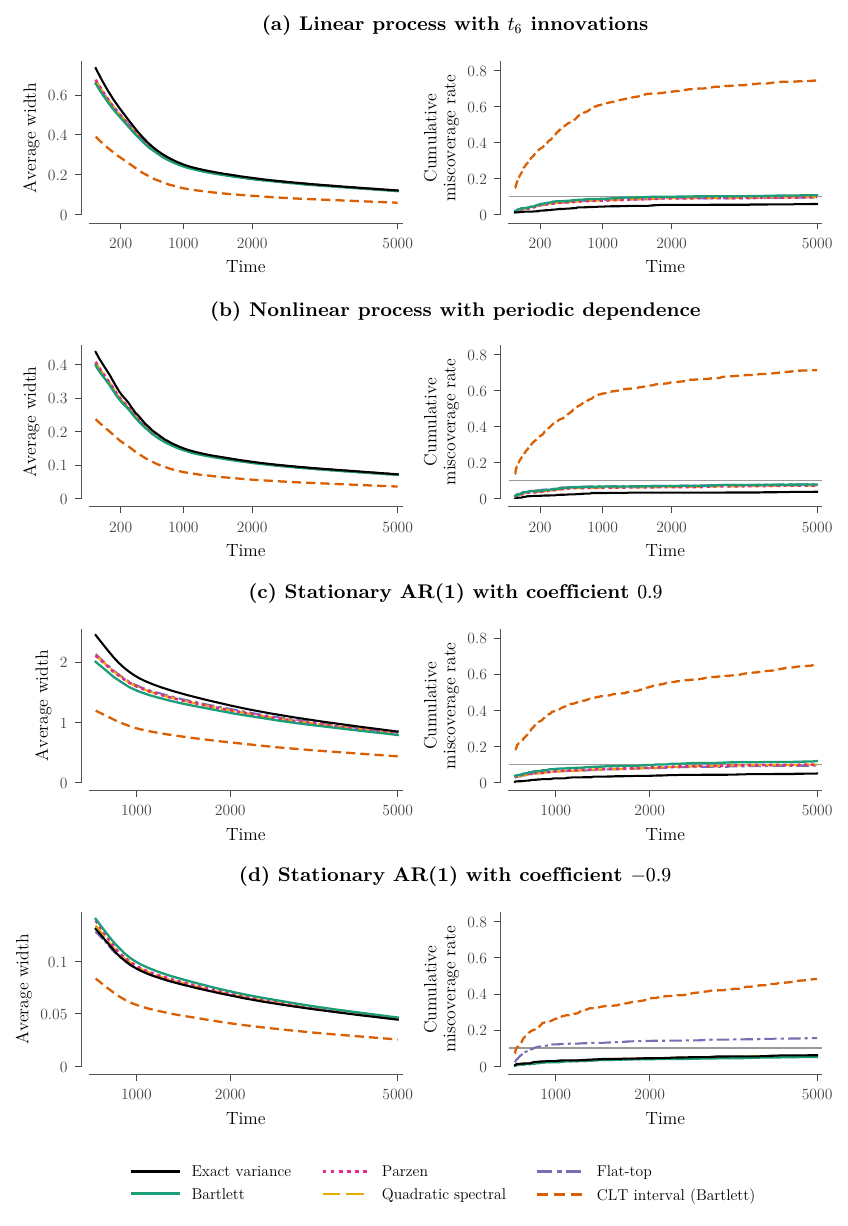}
\caption{Left: average width. Right: cumulative miscoverage rate of nominal $90\%$ AsympCSs beginning at time 100 for (a) and (b) and time 500 for (c) and (d), where the horizontal line marks $\alpha=0.1$. The simulation horizon is 5000, with 1000 replicates for each design.} 
\label{fig:sim_kernels}
\end{figure}
\noindent  Exact variance $\sigma_t^2 = \mathsf{V}_t/t$ and the CLT-interval under the Bartlett kernel serve as baselines in Figure \ref{fig:sim_kernels}. The former removes the estimation error introduced by plug-in HAC estimators, isolating the Gaussian approximation, while the latter illustrates the well-known phenomenon that fixed-time boundaries are invalid under continuous monitoring, leading to inflated miscoverage rates. In particular, the results under the exact variance suggest that the Gaussian approximation in Theorem \ref{Thm: nonstationary ASIP} works well for constructing AsympCSs. For AsympCSs with plug-in HAC estimators, we consider four commonly used kernels: Bartlett, Parzen, quadratic spectral (Equation 2.7 of \cite{andrews1991heteroskedasticity}) and flat-top (\cite{politis1995bias}), with conventional bandwidths $b_t=\lfloor t^{1/3}\rfloor$ for Bartlett and flat-top and $b_t=\lfloor t^{1/5}\rfloor$ for Parzen and quadratic spectral. We do not observe any substantial differences in average width or cumulative miscoverage rate. This suggests that guidance on kernel and bandwidth selection from the fixed-time literature (e.g., \cite{andrews1991heteroskedasticity}; \cite{newey1994automatic}; \cite{politis1995bias}) and more recent developments (\cite{vats2022lugsail}; \cite{chan2024asymptotically}) may remain applicable in the anytime-valid sequential setting. A theoretical justification is left to future work.  

Next, since practitioners may be interested in early monitoring, we examine cumulative miscoverage for smaller burn-in times $m$ in Figure \ref{fig:sim_first_look}. We focus on the first two nonstationary DGPs with moderate dependence, as stronger dependence slows information accumulation and may limit the practical value of early monitoring.
 
\begin{figure}[H]
\centering
\includegraphics[width=0.55\textwidth]{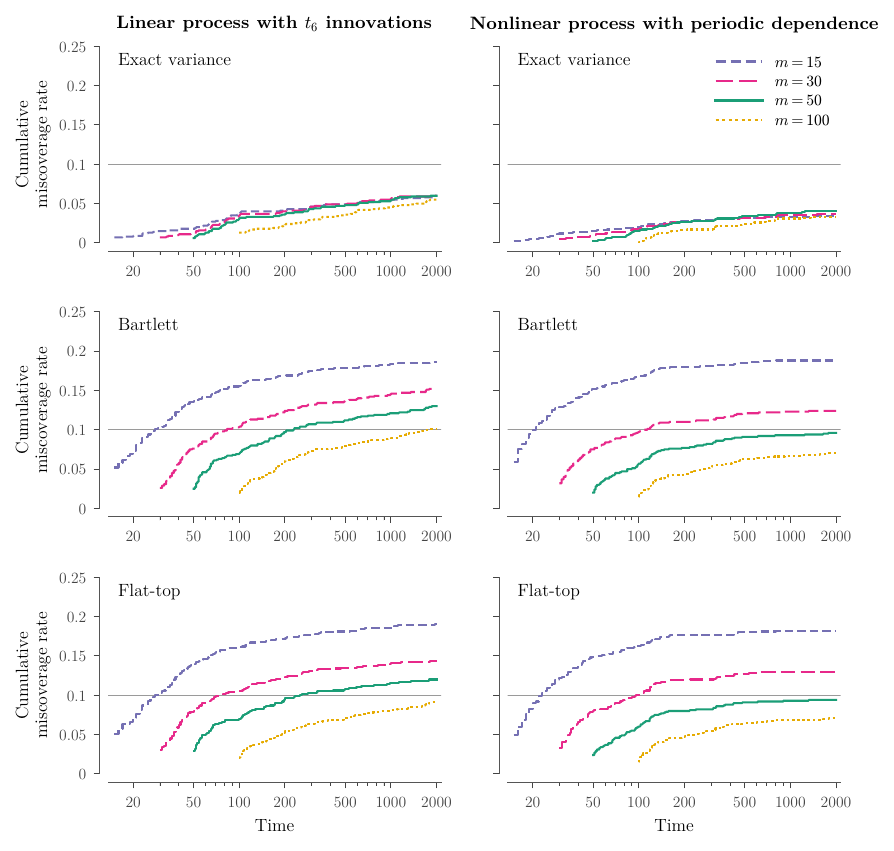}
\caption{Left: cumulative miscoverage rate of nominal $90\%$ AsympCSs beginning at time $m\in\{15,30,50,100\}$ in the linear process. Right: the counterpart in the  nonlinear process with periodic dependence. The horizontal line marks $\alpha=0.1$. The simulation horizon is 2000, with 1000 replicates for each design.}
\label{fig:sim_first_look}
\end{figure}
\noindent As Figure \ref{fig:sim_kernels} shows no substantial differences across kernels, we consider only the Bartlett and flat-top kernels in Figure \ref{fig:sim_first_look}. Consistent with findings of existing simulation studies under the super-population framework in the AsympCS literature (e.g., \cite{waudby2024time}; \cite{dalal2024anytime}), the asymptotic approximation starts to ``kick in'' after a modest burn-in period, here around 50 observations, slightly longer than that in the i.i.d.\ settings owing to dependence.

\section{Discussion} \label{Sec: Conclusion}
\noindent In this paper, we establish a time-uniform dependent central limit theorem under the functional dependence framework, extending asymptotic confidence sequences beyond cross-sectional and martingale settings to accommodate potentially nonstationary time series. We illustrate our framework through applications in algorithms with Markovian structure and time-series-based causal estimands.

Several directions for future research remain. First, while we focus on the univariate case, our proof techniques, together with \cite{wu2011gaussian}, suggest a natural route toward establishing a multivariate nonstationary ASIP. Combined with the multivariate nonasymptotic confidence sequence construction developed in \cite{chugg2023time}, this could further enable vector-valued asymptotic confidence sequences for nonstationary time series. Second, given the broad class of distributions for time series, one may leverage the recent distribution-uniform anytime-valid inference framework of \cite{waudby2023distribution} to construct distribution-uniform extensions of our results. Third, extending our framework to other distributional functionals, such as quantiles which are central to tail-risk applications in financial time-series data, constitutes a promising direction for future research. Such an extension may be achieved by combining the tail functional dependence framework introduced by \cite{zhang2021high} with further advances in quantile or empirical process ASIPs in time series to construct AsympCSs. None of these extensions is trivial, yet each is both technically interesting and of substantial practical relevance. Taken together, we hope that this work opens a fruitful direction for future research on anytime-valid inference for time-series data beyond martingale settings.

\bibliography{ref}

\appendix

\clearpage

\section*{Appendix}
\addcontentsline{toc}{section}{Appendix}

\noindent The appendices are organized as follows. Appendix~\ref{Appendix: Proofs of Nonstationary ASIP}
proves the nonstationary ASIP and collects the technical inequalities and auxiliary lemmas that are needed for the proof. Throughout this Appendix, for ease of exposition, we use the terms ``ASIP'' and ``strong Gaussian approximation'' interchangeably. Some other literature uses the second term more broadly for uniform Gaussian couplings in probability (e.g., \cite{chernozhukov2013intersection}; \cite{belloni2015some}; \cite{li2020uniform}), but we use it to refer exclusively to approximations in the almost-sure sense. Appendix~\ref{Appendix: Proof of TUCLT} provides supporting proofs of results in Section \ref{Sec: Asymp CS} and Section \ref{Sec: Robbins Gaussian Mixture Asymp CSs for (Nonstationary) Time Series}. Appendix \ref{Appendix: Sampling and Optimization Algorithms} provides further discussion of SGD, followed by the proof of the AsympCS result for MCMC ergodic averages. Finally, Appendix~\ref{Appendix: Additional Discussions} collects additional discussion.

\section{Proofs of Nonstationary ASIP - Theorem \ref{Thm: nonstationary ASIP}}  \label{Appendix: Proofs of Nonstationary ASIP}

\noindent The proof of Theorem~\ref{Thm: nonstationary ASIP} builds on \cite{bonnerjee2024gaussian}. It is lengthy and largely self-contained. To facilitate readers' verification and to better follow conventions in the time-series literature, we use notation that differs slightly from that in the main text. Unless otherwise specified, the notation introduced in this appendix is local and does not carry over to other sections. In the main text, Assumption~\ref{Assump:dep_decay} of Theorem~\ref{Thm: nonstationary ASIP} is stated only for $p\in(2,3]$. In this appendix, we present the complete Gaussian approximation theorem for general $p>2$ and restate the assumptions at this level of generality.

This part is organized as follows. Section~\ref{sec_asip: set up} introduces the notation and presents the complete theorem. Section~\ref{sec:proof_ASIP} gives the main proof.
Section~\ref{sec_asip: FDM lemmas} proves auxiliary lemmas stated and used in Section~\ref{sec:proof_ASIP}, including technical inequalities needed under the functional dependence
framework.

\subsection{Notations and the Complete Theorem} \label{sec_asip: set up}
\noindent \textbf{Notations.} For a random variable $X$ and $p\ge1$, let $\norm{X}_p=(\E|X|^p)^{1/p}$ denote the $\Lp_p$-norm, and write $\norm{\cdot}=\norm{\cdot}_2$. $C$ denotes a generic positive constant whose value may change from line to line. For positive quantities $a$ and $b$, write $a\lesssim b$ if $a\le Cb$ for some constant $C>0$. For positive sequences $a_n$ and $b_n$, write $a_n=o(b_n)$ if $a_n/b_n\to0$. For random variables $(X_n)_{n\ge1}$ and positive numbers $(y_n)_{n\ge1}$, write $X_n=\oas{y_n}$ if $X_n/y_n\to0$ almost surely, and $X_n=O_{\mathrm{a.s.}}(y_n)$ if $|X_n|/y_n$ is almost surely bounded. We write $X\overset{\mathcal D}{=}Y$ for equality in distribution, use $\1\{\cdot\}$ for the indicator function, denote by $|\mathcal D|$ the cardinality of a set $\mathcal D$, and let $\lfloor x\rfloor$ be the largest integer not exceeding $x$. Unless otherwise specified, we assume throughout that $p>2$ and $\E[X_i]=0$ for all $i\in\mathbb Z$.

\vspace{12pt}
\label{Appendix: The Complete Theorem}
\noindent \textbf{Assumptions \& the Complete Theorem.} For simplicity, we work exclusively with the centered process $(X_i)_{i\in\mathbb{Z}}$. Since the mean component is deterministic, the ASIP result remains valid for the original process $(Y_i)_{i\in\mathbb{Z}}$. We now state the complete assumptions for our result under general $p >2$. 

\begin{Ass} \label{asm:dep_decay}
Consider Equation \eqref{eq:g_function}. Suppose that $\Theta_{0,p}<\infty$ for some $p>2$. Assume there exist $A>1$ and a positive constant $C>0$ such that
    \begin{align}
        \sup_{i\geq0} (i+1)^A \Theta_{i,p} \leq C,
        \label{eq:dep_decay}
    \end{align}
    where $A$ further satisfies
    \begin{align}\label{eq:decay_rate}
    A > \frac{p^2 - p - 2 + (p-2)\sqrt{p^2+10p+1}}{4p}.
    \end{align}
\end{Ass}

\noindent When $p \in (2,3]$, Equation (\ref{eq:decay_rate}) implies $A > (2 + \sqrt{10})/6$, so the additional constraint is always satisfied given $ A > 1$. In other words, Assumption \ref{Assump:dep_decay} is a special case of Assumption \ref{asm:dep_decay}. As stated in the main text, a uniform tail behavior control is also needed due to the nonstationarity behavior.

\begin{Ass} \label{asm:UI}
    For the same $p$ as in Assumption \ref{asm:dep_decay}, there exists a non-negative random variable $Q\in\Lp_p$ and a constant $C > 0$, such that for all $x > 0$:
\begin{equation}
    \sup_{i\in\mathbb{Z}}\mathbb{P}[|X_i| > x] \leq C  \mathbb{P}[Q > x].
\end{equation}
\end{Ass}

\noindent With these preliminary conditions in place, we have the following complete ASIP result. 
\begin{Thm} \label{thm:main}
    For the process $(X_i)_{i\geq1}$, assume Assumptions \ref{asm:dep_decay} and \ref{asm:UI} hold. Then, on a possibly richer probability space, there exists a standard Brownian motion $\Bro(\cdot)$ such that
    \begin{align}
        \mathsf{S}_n - \Bro(\E[\mathsf{S}_n^2])  = \oas{n^{1/p}\sqrt{\log n}},
    \end{align}
where $\mathsf{S}_n := \sum_{i=1}^n X_i$ denotes the partial sum process.
\end{Thm}

\subsection{Proofs of Theorem \ref{thm:main} (the Complete Theorem)} \label{sec:proof_ASIP}

\noindent The proof of Theorem \ref{thm:main} adopts the construction framework established in \cite{bonnerjee2024gaussian}, where a convergence rate of $\op{n^{1/p}\sqrt{\log n}}$ was obtained. To strengthen this result from convergence in probability to almost sure convergence, we require significantly sharper estimates and a more delicate analysis of the tail probabilities. We begin by outlining the main steps of the proof, followed by the detailed verification of each step.

\vspace{6pt}
\noindent \textbf{Outline of the proof of part (a).} To strengthen convergence in probability to almost sure convergence, we partition the time horizon into geometrically increasing groups, denoted by $\mathcal{D}_k = (3^{k-1}, 3^k]$. Let $h_n$ be the integer satisfying $3^{h_n-1} < n \leq 3^{h_n}$, so that the sample of size $n$ is decomposed into $h_n$ groups. Within each group $\mathcal{D}_k$ ($1 \le k \le h_n$), we perform the following three reduction steps:

\begin{itemize}
    \item \textbf{Truncation}: In Proposition \ref{prop:truncation_error}, we truncate the process at the level $3^{k/p}$. This step allows us to exploit the stochastic domination condition to control the tail probabilities.
    
    \item \textbf{$m$-dependent approximation}: In the second step, we use an $m_k$-dependent approximation within the group $\mathcal{D}_k$, where $m_k$ increases with $k$. This limits the arbitrary non-stationary dependency structure to a lag of $m_k$. We provide an optimal choice of $m_k$ to ensure that the error rate of $\oas{n^{1/p}}$ is achieved.
    
    \item \textbf{Blocking}: We employ a blocking strategy similar to \cite{berkes2014kmt}: a two-step construction with an inner layer of blocks of size $m_k$ being combined into an outer layer of blocks of size $3m_k$. This structure enables us to perform the necessary mathematical manipulations to obtain an explicit form of the variance in terms of $m_k$-dependent processes.
\end{itemize}

\noindent This three-step reduction transforms the complex original process into a sequence of block sums suitable for the strong approximation result of \cite{sakhanenko2006estimates}, while ensuring that the cumulative approximation error at each stage remains within the required $\oas{n^{1/p}}$ bound. 

\vspace{3pt}
In each of the reduction steps outlined above, our strategy relies on constructing a more tractable surrogate to approximate the original group sum $\sum_{i\in\mathcal{D}_k} X_i$. To verify that the cumulative approximation errors remain within the desired bound, we formulate the following general lemma based on the Borel-Cantelli lemma.

\begin{Lem} \label{lem:approx_principle}
Let $\{\Delta_k\}_{k\geq 1}$ be a sequence of random variables representing the approximation errors within the group $\mathcal{D}_k$. If for any $\delta > 0$, the tail probabilities satisfy the summability condition:
\begin{align*} 
\sum_{k=1}^\infty \mathbb{P}\left[ \abs{\Delta_k} > \delta 3^{k/p} \right] < \infty,
\end{align*}
then the cumulative error satisfies the almost sure bound:
\begin{align*}
\max_{1\le m \le h_n} \abs{\sum_{k=1}^m \Delta_k} = \oas{n^{1/p}}.
\end{align*}
\end{Lem}

\noindent The proof is a direct consequence of the Borel-Cantelli lemma and is therefore omitted. We will frequently invoke the Borel-Cantelli lemma to establish similar results throughout the paper.

\vspace{6pt}
\noindent \textbf{Outline of the proof of part (b).} This constitutes the most intricate component of our proof. While the strong Gaussian approximation for non-stationary but independent processes is well-established in \cite{sakhanenko2006estimates}, applying it to our setting requires decoupling the dependence structure. Thus, the main objective here is to construct independent increments that allow us to leverage Sakhanenko's result from \cite{sakhanenko2006estimates}. Building upon the regularized block structure established above, we achieve this by implementing a delicate two-step Gaussian approximation strategy within each group $\mathcal{D}_k$: \begin{itemize}
    \item \textbf{Conditional Gaussian approximation}: By conditioning on the shared $\eps$'s between outer blocks, we ensure block-wise independence, allowing for the direct application of Theorem 1 from \cite{sakhanenko2006estimates} to establish a conditional strong approximation. Plugging back the randomness of the shared $\eps$ yields an intermediate unconditional object, which behaves as a Brownian motion \text{only conditional on} the boundary terms.

    \item \textbf{Unconditional Gaussian approximation}: We transform this intermediate object into a surrogate process composed of independent increments. This surrogate serves as a valid proxy for the original partial sum $\mathsf{S}_n$, enabling a second application of Theorem 1 in \cite{sakhanenko2006estimates} to derive an unconditional Gaussian approximation.
\end{itemize}

This unconditional approximation hinges on the explicit variance structure of the $m_k$-dependent process derived earlier. We construct independent increments by performing a strategic algebraic reformulation of this variance and exploiting the scaling properties of Brownian motion.

\vspace{6pt}
\noindent \textbf{Outline of the proof of part (c).} In the first two steps, we use a sequence of approximations to obtain a tractable auxiliary process and its strong Gaussian approximation. However, the resulting Gaussian approximation has a variance that differs from the target $(\mathbb E[\mathsf S_i^2])_{i\ge1}$. Another issue is that making all preceding approximation errors negligible requires a careful choice of $(m_k)_{k\geq1}$ to achieve the optimal approximation. We address these two issues and therefore complete the proof in two steps: \begin{itemize}
    \item \textbf{Regularization of variance}: We correct the approximating Gaussian process step by step to account for the effect of the preceding approximations, so that the final Gaussian approximation has the target variance process $(\mathbb{E}[\mathsf{S}_i^2])_{i\ge1}$.
    
    \item \textbf{Decay rate calibration}: We first collect the set of conditions on the block sequence $(m_k)_{k\ge1}$ required for the optimal approximation to hold. These constraints on $m_k$ can be equivalently re-expressed as restrictions on the dependence-decay exponent $A$. We then characterize the strongest admissible dependency for which all conditions are simultaneously satisfied.
\end{itemize}

This part is standard but somewhat involved. Note that the optimal rate $n^{1/p}$ is already achieved once the Gaussian proxy is matched to the truncated variance process $(\E[(\mathsf S_i^{\oplus})^2])_{i\ge1}$. The subsequent step that replaces $(\E[(\mathsf S_i^{\oplus})^2])_{i\ge1}$ by the target $(\E[\mathsf S_i^2])_{i\ge1}$ typically incurs an additional $\sqrt{\log n}$ factor; in general this term is non-negligible without strong nondegeneracy conditions, and we will quantify it later in the proof.

\subsubsection{(a) - Truncation}

\noindent We begin by defining the standard truncation operator $\T{a}{\cdot}$ at level $a > 0$:
\begin{align*}
    \T{a}{w} = \max\{\min\{w, a\}, -a\}, \quad w \in \R.
\end{align*}
Note that $\T{a}{\cdot}$ is a Lipschitz continuous function with a Lipschitz constant of 1.

\vspace{3pt}
To handle the potential heavy tails, we apply a group-dependent truncation. Specifically, for any index $i$, let $k(i)$ denote the unique group index such that $i \in \mathcal{D}_{k}$. We then define the truncated variable $X_i^{\oplus}$ by $X_i^{\oplus} := \T{3^{k/p}}{X_i}.$ In other words, if $i$ falls into the $k$-th group $\mathcal{D}_k$, we truncate $X_i$ at the level $3^{k/p}$. Let $\mathcal{W}_{k,l}$ denote the partial sum of these centered truncated variables starting from the beginning of the $k$-th group:
\begin{align*}
    \mathcal{W}_{k,l} = \sum_{i = 3^{k-1}+1}^{3^{k-1}+l} \left( \T{3^{k/p}}{X_i} - \E[\T{3^{k/p}}{X_i}] \right).
\end{align*}
The truncated partial sum process, denoted by $\mathsf{S}_n^{\oplus}$, is constructed by accumulating these contributions. For any $n$ falling into the $h_n$-th group (i.e., $3^{h_n-1} < n \le 3^{h_n}$), we have the decomposition:
\begin{align*}
    \mathsf{S}_n^{\oplus} = \sum_{i=1}^n \left(X_i^{\oplus} - \E[X_i^{\oplus}]\right)=\sum_{k=1}^{h_n-1} \mathcal{W}_{k, 3^{k} - 3^{k-1}} + \sum_{i=3^{h_n-1}+1}^n \left( \T{3^{h_n/p}}{X_i} - \E[\T{3^{h_n/p}}{X_i}] \right).
\end{align*}
The following proposition establishes that $\mathsf{S}_n^{\oplus}$ serves as an optimal approximation for the original partial sum process $\mathsf{S}_n$.

\begin{Prop} \label{prop:truncation_error}
    Under Assumption \ref{asm:UI}, the approximation error satisfies
    \begin{align*}
        \max_{1\le i \le n} \abs{\mathsf{S}_i - \mathsf{S}_i^{\oplus}} = \oas{n^{1/p}}.
    \end{align*}
\end{Prop}

\begin{proof}
   The proof consists of two parts. First, we show that $\max_{1\leq r\leq n}| \mathsf{S}_r - \sum_{i=1}^r X_i^{\oplus}| = \oas{n^{1/p}}$. Second, we show that $\sum_{i=1}^n |\E[X_i^{\oplus}]| = o(n^{1/p})$.

\vspace{3pt}
\noindent To prove the first claim, by the Borel-Cantelli lemma, it suffices to verify that $\sum_{i\geq1} \P[X_i \neq X_i^{\oplus}]<\infty$. Noting that $\P[X_i \neq X_i^{\oplus}] = \P[\abs{X_i} > 3^{k(i)/p}]$ and invoking Assumption \ref{asm:UI}, we have:
\begin{align*}
    &\sum_{i=1}^\infty \P[X_i \neq X_i^{\oplus}]  =\sum_{k=1}^\infty \sum_{i\in\mathcal{D}_k} \P[\abs{X_i}>3^{k/p}] \\ 
     \lesssim &\sum_{k=1}^\infty 3^{k} \P[Q > 3^{k/p}] \lesssim \E[Q^p] <\infty.
\end{align*}

\noindent To prove the second claim $\sum_{i=1}^n \E[X_i^{\oplus}] = o(n^{1/p})$, recall that $\abs{\E[\T{a}{X_i}]} \leq \E[\abs{X_i}\1\{\abs{X_i} > a\}]$. By applying Assumption \ref{asm:UI} again, we have $\E[\abs{X_i}\1\{X_i > a\}] \leq C\E[Q\1\{ Q>a\}]$. Thus, we deduce:
\begin{align*}
    \sum_{i=1}^n \abs{\E[X_i^{\oplus}]} & \leq \sum_{k=1}^{h_n-1} \sum_{i\in\mathcal{D}_k} \E[\abs{X_i}\1\{\abs{X_i}>3^{k/p}\}] + \sum_{i=3^{h_n-1}+1}^n \E[\abs{X_i}\1\{\abs{X_i}>3^{h_n/p}\}] \\ 
    & \lesssim \sum_{k=1}^{h_n} 3^k \E[Q \1\{Q > 3^{k/p}\}] \\
    & \leq  \sum_{k=1}^{h_n} 3^{k/p} \E[Q^{p} \1\{Q>3^{k/p}\}] = o(n^{1/p}).
\end{align*}

\noindent The final inequality follows from the bound $\E[\abs{X}\1\{\abs{X}>a\}] \leq a^{1-p} \E[\abs{X}^p\1\{\abs{X}> a\}]$.
\end{proof}

\subsubsection{(a) - \texorpdfstring{$m$}{m}-dependence Approximation}

\noindent For $i\in \mathcal{D}_k$, we construct the $m_k$-dependent approximation $\tilde{X}_{k,i} = \E[X_i^{\oplus}|\eps_i,\dots,\eps_{i-m_k}] - \E[X_i^{\oplus}]$. By definition, this sequence forms an $m_k$-dependent process. Consequently, the associated partial sum process is defined as
\begin{align*}
    \tilde{\mathsf{S}}_n = \sum_{k=1}^{h_n-1} \tilde{\mathcal{W}}_{k,3^k - 3^{k-1}} + \tilde{\mathcal{W}}_{h_n,n - 3^{h_n-1}},\quad \mathrm{where}\ \tilde{\mathcal{W}}_{k,l} = \sum_{i=3^{k-1}+1}^{3^{k-1}+l} \tilde{X}_{k,i}.
\end{align*}

\noindent To establish the bound $\max_{1\leq i \leq n}|\tilde{\mathsf{S}}_i - \mathsf{S}_i^{\oplus}| = \oas{n^{1/p}}$, it is necessary that the dependence lag $m_k$ increases with $k$ to improve approximation precision. The following lemma quantifies the approximation error between the $m_k$-dependent process and the truncated process $\mathsf{S}_n^t$:

\begin{Lem}\label{lem:inequality_1}
    Under the assumptions of Theorem \ref{thm:main}, we have \begin{align}\label{mk_approx_inequality}
        \norm{\max_{1\leq l \leq 2\times 3^{k-1}}|\tilde{\mathcal{W}}_{k,l} - \mathcal{W}_{k,l} |}_p\leq C\left( 3^{k} - 3^{k-1}\right)^{1/2} \Theta_{1+m_k,p}.
    \end{align}
\end{Lem}

\noindent Equipped with this lemma, we can now specify the growth condition on $m_k$ sufficient to ensure $\max_{1\leq i \leq n}|\tilde{\mathsf{S}}_i - \mathsf{S}_i^{\oplus}| = \oas{n^{1/p}}$.

\begin{Prop}
    Under the assumptions of Theorem \ref{thm:main}, if there exists a positive integer sequence $(m_k)_{k=1}^{\infty}$ such that \begin{align}\label{condition_1}
        \sum_{k=1}^{\infty} \frac{3^{kp/2}\Theta^p_{m_k,p}}{3^{k}} <\infty,
    \end{align}

    \noindent then we have $\max_{1\leq i \leq n}|\tilde{\mathsf{S}}_i - \mathsf{S}_i^{\oplus}| = \oas{n^{1/p}}$.
\end{Prop}

\begin{proof}
    To establish the convergence $\max_{1\leq i \leq n}|\tilde{\mathsf{S}}_i - \mathsf{S}_i^{\oplus}| = \oas{n^{1/p}}$, we first observe the following decomposition:
    \[
        \max_{1\leq i\leq n} |\tilde{\mathsf{S}}_i - \mathsf{S}_i^{\oplus}| \leq \sum_{k=1}^{h_n-1} \max_{1\leq l \leq 2\times 3^{k-1}} |\tilde{\mathcal{W}}_{k,l} - \mathcal{W}_{k,l}| + \max_{1\leq l \leq n-3^{h_n-1}}|\tilde{\mathcal{W}}_{h_n,l} - \mathcal{W}_{h_n,l}|.
    \]
    Consequently, by the Borel-Cantelli lemma, it suffices to show that for any $\delta>0$,
    \begin{align*}
        \sum_{k=1}^{\infty} \P\left[ \max_{1\leq l\leq 2\times 3^{k-1}} |\tilde{\mathcal{W}}_{k,l} - \mathcal{W}_{k,l}| > \delta 3^{k/p}\right] <\infty.
    \end{align*}
    Invoking Lemma \ref{lem:inequality_1} together with Markov's inequality and Condition \eqref{condition_1}, we obtain
    \begin{align*}
        \sum_{k=1}^{\infty} \P\left[ \max_{1\leq l\leq 2\times 3^{k-1}} |\tilde{\mathcal{W}}_{k,l} - \mathcal{W}_{k,l}| > \delta 3^{k/p}\right] \lesssim \frac{1}{\delta^p} \sum_{k=1}^{\infty} \frac{(3^k - 3^{k-1})^{p/2}\Theta_{1+m_k,p}^p}{3^k}<\infty,
    \end{align*}
    which completes the proof.
\end{proof}

To satisfy condition \eqref{condition_1}, $m_k$ must grow sufficiently fast with $k$. Conversely, the proof in the next step implies that the growth rate of $m_k$ must be restricted to ensure the generation of sufficiently many independent blocks. 

\subsubsection{(a) - Blocking}

\noindent In this step, we build blocks of sums using the $m$-dependent process derived above. Let $q_k = \lfloor 2\times 3^{k-2}/m_k \rfloor - 2$. We choose $K_0\in\mathbb{N}^{\times}$ large enough so that $q_k\geq 2$ for all $k\geq K_0$, and set $N_0 = 3^{K_0}$. For $k\geq K_0$, we define the outer blocks as
\begin{align}\label{def_block}
    \mathcal{B}_{k,j} = \sum_{3^{k-1}+1+3jm_k}^{3^{k-1} + 3(j+1)m_k} \tilde{X}_{k,i},
\end{align}

\noindent where $\mathcal{B}_{k,j}$ is the sum of $3m_k$ consecutive random variables, and $q_k$ is the maximum number of blocks we can keep within group $\mathcal{D}_k$. The corresponding partial sum process is defined as
\begin{align}\label{def_S_block}
    \mathsf{S}_n^{\Diamond} := \sum_{k = K_0}^{h_n-1} \sum_{j=1}^{q_k} \block_{k,j}+ \sum_{j=1}^{\tau_n} \block_{h_n,j},
\end{align}

\noindent where $\tau_n = \lfloor (n-3^{h_n-1})/3m_{h_n}\rfloor - 2$.

\vspace{3pt}
This construction introduces a specific dependence structure. Since we remove the first $3m_k$ terms in each group $\mathcal{D}_k$, blocks from different groups are independent. However, adjacent blocks $\block_{k,j}$ and $\block_{k,j+1}$ within the same group remain dependent because there is no gap between them. Crucially, the resulting process $\mathsf{S}_n^{\Diamond}$ has a tractable variance structure, which facilitates the subsequent Gaussian approximation in later steps. 

\vspace{3pt}
Now, our goal is to show that $\max_{1\leq i \leq n} |\mathsf{S}_n^{\Diamond} - \tilde{\mathsf{S}}_n | = \oas{n^{1/p}}$. The difference between these two processes comes from the blocks removed during construction. Intuitively, a larger block size $m_k$ leads to a larger approximation error. Therefore, we need to limit the growth rate of $m_k$ to keep the error small enough. To control the moments of the maximal partial sums of the removed blocks, we introduce the following key lemma.

\begin{Lem}\label{lem:rosenthal}
    Under the assumptions of Theorem \ref{thm:main}, suppose there exists $\alpha > p$ such that
    \begin{align}\label{condition_2}
        \Xi_{p}(\alpha):= \sum_{h= 0}^{\infty} \abs{h}^{1/2-1/\alpha} \delta_{p}(h)^{p/\alpha} <\infty,\quad \Lambda_p(\alpha):= \sum_{k=1}^{\infty} 3^{k-k\alpha/p} m_k^{\alpha/2-1} < \infty.
    \end{align}
    Then, there exists a constant $C$ depending only on $\alpha$ and $p$ such that
    \begin{align}\label{rosenthal_ineq}
        \sum_{k=K_0}^{\infty} \max_{1\leq j \leq q_k}\frac{3^k\E\left[\max_{3jm_k < l\leq 3(j+1)m_k} |\tilde{\mathcal{S}}_{3jm_k,l}^{(k)}|^{\alpha}\right]}{m_k3^{k\alpha/p}} \leq C\left( \Lambda_p(\alpha) \Theta_{0,2}^{\alpha} + \Xi_p(\alpha)^{\alpha} + \sup_{i\in\mathbb{Z}} \norm{X_i}_p^p\right),
    \end{align}
    where $\tilde{\mathcal{S}}_{s,l}^{(k)} = \sum_{i=s+1}^l \tilde{X}_{k,i}$ with $\tilde{\mathcal{S}}_{l,l}^{(k)}=0$.
\end{Lem}

Condition \eqref{condition_2} imposes technical constraints on the decay rate of the dependence measure $\delta_p(h)$ and the growth rate of the block size $m_k$. These requirements are introduced specifically to bound the summation in \eqref{rosenthal_ineq} via a Rosenthal-type inequality for non-stationary processes. With the help of Lemma \ref{lem:rosenthal}, we are now positioned to show $\max_{1\leq i \leq n} |\mathsf{S}_n^{\Diamond} - \tilde{\mathsf{S}}_n | = \oas{n^{1/p}}$.

\begin{Prop}
    Under the assumptions of Theorem \ref{thm:main}, if there exists $\alpha > p$ such that Condition~\eqref{condition_2} holds, then $\max_{1\leq i\leq n}|\mathsf{S}_i^{\Diamond} - \tilde{\mathsf{S}}_i| = \oas{n^{1/p}}$.
\end{Prop}

\begin{proof}
    Assume $k\geq K_0$ and $n\geq N_0$. For any $\delta>0$, the maximal deviation between $\mathsf{S}_n^{\Diamond}$ and $\tilde{\mathsf{S}}_n$ within group $\mathcal{D}_k$ satisfies  
    \begin{align}
        &\P\left[ \max_{1\leq l \leq 2\times 3^{k-1}} \abs{\tilde{\mathcal{W}}_{k,l} - \sum_{j=1}^{\lfloor l/(3m_k)\rfloor} \block_{k,j}}  \geq \delta 3^{k/p}\right] \nonumber \\ \leq & \frac{2\times 3^{k-2}}{m_k} \max_{1\leq j \leq q_k}\P\left[ \max_{3jm_k < l\leq 3(j+1)m_k} |\tilde{\mathcal{S}}_{3jm_k,l}^{(k)}| > \delta 3^{k/p} \right] \\ \leq & \max_{1\leq j \leq q_k}\frac{3^k  \E\left[\max_{3jm_k < l\leq 3(j+1)m_k} |\tilde{\mathcal{S}}_{3jm_k,l}^{(k)}|^{\alpha}\right] }{m_k \delta^{\alpha}3^{k\alpha/p}}, \nonumber
    \end{align}
    
    \noindent where the first inequality uses the union bound over the sub-blocks covering the interval $\mathcal{D}_k$. By Lemma \ref{lem:rosenthal}, we know that this tail probability is summable over the group index $k$. Thus, by the Borel-Cantelli lemma, we can conclude that $\max_{1\leq i\leq n}|\mathsf{S}_i^{\Diamond} - \tilde{\mathsf{S}}_i| = \oas{n^{1/p}}$.
\end{proof}

In the proof of the above results, we assume Conditions \eqref{condition_1} and \eqref{condition_2} hold without explicitly establishing that such a positive sequence $(m_k)_k$ exists under the assumptions of Theorem \ref{thm:main}. The existence and optimal choice of such a sequence $(m_k)_{k\ge 1}$ will be discussed at the end of the proof.

\subsubsection{(b) - Conditional Gaussian Approximation}

\noindent As noted previously, while the blocks $\mathcal{B}_{k,j}$ in group $\mathcal{D}_k$ are independent of blocks in other groups, they remain dependent on their adjacent blocks $\mathcal{B}_{k,j-1}$ and $\mathcal{B}_{k,j+1}$. This local dependence is induced by the shared innovations $\varepsilon_i$ along the block boundaries. In this step, we condition on these boundary innovations to render the blocks conditionally independent, allowing us to apply Theorem 1 of \cite{sakhanenko2006estimates} to establish a conditional Gaussian approximation.

\vspace{3pt}
To formalize the conditioning argument, we introduce a generalized notation based on $m_k$-block partitions. First, for any integer $l \ge 1$, we define the index set $\mathcal{I}_{k,l} = \{ 3^{k-1} + (l-1)m_k + i : i = 1,\dots, m_k\}$ and its corresponding innovation vector $\bm{\eta}_{k,l} = (\varepsilon_i)_{i\in\mathcal{I}_{k,l}}$. Under this notation, the innovations constituting the $j$-th outer block $\mathcal{B}_{k,j}$ correspond to the indices $l \in \{3j+1, 3j+2, 3j+3\}$. Specifically, we partition these into two categories: the \text{internal innovations}, represented by $\bm{\eta}_{k,3j+1}$ and $\bm{\eta}_{k,3j+2}$, and the \text{boundary innovations}. We define the collection of boundary innovations as $\bm{\eta}^{b} = (\bm{\eta}_{k,3j})_{k\geq K_0, 1\leq j \leq (q_k+1)}$. These boundary innovations serve as the dependency bridges between adjacent blocks. We also define a corresponding collection of real vectors $\bm{a}^{b} = (\bm{a}_{k,3j})_{k \ge K_0, 1 \le j \le (q_k+1)}$, where each component $\bm{a}_{k,3j} \in \mathbb{R}^{m_k}$ represents a specific realization of the boundary innovations. Furthermore, for any $i \in \mathcal{D}_k$, let $\tilde{\varphi}_i$ be a measurable function such that the process admits the representation $\tilde{X}_{k,i} = \tilde{\varphi}_i(\varepsilon_i, \dots, \varepsilon_{i-m_k})$.

\vspace{3pt}

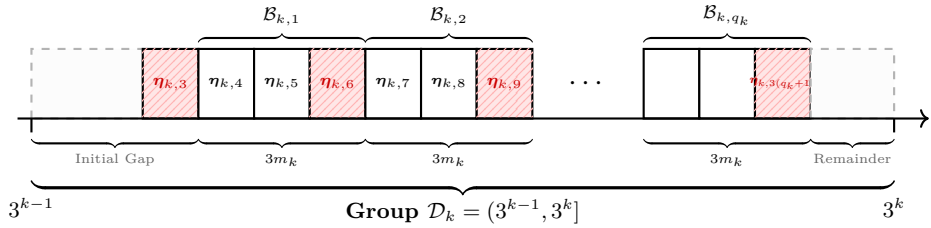
\begin{figure}[H]
    \centering
    \begin{tikzpicture}[
        scale=0.92, transform shape, 
        block/.style={draw, thick, minimum height=1.0cm, outer sep=0pt}, 
        gap/.style={draw, dashed, thick, color=gray!60, minimum height=1.0cm, fill=gray!2},
        brace/.style={decorate, decoration={calligraphic brace, amplitude=3pt, mirror}, thick},
        topbrace/.style={decorate, decoration={calligraphic brace, amplitude=3pt}, thick}, 
        bigbrace/.style={decorate, decoration={calligraphic brace, amplitude=5pt, mirror}, very thick},
        tick/.style={draw, thick}
    ]

        \def\mwidth{0.8}  
        \def\bheight{1.0} 
        \def\tailgap{1.2} 

        
        \draw[gap] (0,0) rectangle (2*\mwidth, \bheight);
        
        \draw[block, fill=red!10] (2*\mwidth,0) rectangle (3*\mwidth, \bheight);
        \fill[pattern=north east lines, pattern color=red!30] (2*\mwidth,0) rectangle (3*\mwidth, \bheight);
        \node[font=\scriptsize, text=red!80!black] at (2.5*\mwidth, \bheight/2) {$\bm{\eta}_{k,3}$};
        
        
        \begin{scope}[shift={(3*\mwidth, 0)}] 
            \draw[block, fill=white] (0,0) rectangle (1*\mwidth, \bheight);
            \node[font=\tiny] at (0.5*\mwidth, \bheight/2) {$\bm{\eta}_{k,4}$};

            \draw[block, fill=white] (1*\mwidth,0) rectangle (2*\mwidth, \bheight);
            \node[font=\tiny] at (1.5*\mwidth, \bheight/2) {$\bm{\eta}_{k,5}$};
            
            \draw[block, fill=red!10] (2*\mwidth,0) rectangle (3*\mwidth, \bheight);
            \fill[pattern=north east lines, pattern color=red!30] (2*\mwidth,0) rectangle (3*\mwidth, \bheight);
            \node[font=\scriptsize, text=red!80!black] at (2.5*\mwidth, \bheight/2) {$\bm{\eta}_{k,6}$};

            \draw[topbrace] (0, \bheight+0.1) -- (3*\mwidth, \bheight+0.1) 
                node[midway, above=3pt, font=\footnotesize\bfseries] {$\mathcal{B}_{k,1}$};
        \end{scope}

        \begin{scope}[shift={(6*\mwidth, 0)}] 
            \draw[block, fill=white] (0,0) rectangle (1*\mwidth, \bheight);
            \node[font=\tiny] at (0.5*\mwidth, \bheight/2) {$\bm{\eta}_{k,7}$};

            \draw[block, fill=white] (1*\mwidth,0) rectangle (2*\mwidth, \bheight);
            \node[font=\tiny] at (1.5*\mwidth, \bheight/2) {$\bm{\eta}_{k,8}$};
            
            \draw[block, fill=red!10] (2*\mwidth,0) rectangle (3*\mwidth, \bheight);
            \fill[pattern=north east lines, pattern color=red!30] (2*\mwidth,0) rectangle (3*\mwidth, \bheight);
            \node[font=\scriptsize, text=red!80!black] at (2.5*\mwidth, \bheight/2) {$\bm{\eta}_{k,9}$};

            \draw[topbrace] (0, \bheight+0.1) -- (3*\mwidth, \bheight+0.1) 
                node[midway, above=3pt, font=\footnotesize\bfseries] {$\mathcal{B}_{k,2}$};
        \end{scope}

        \node[font=\large] at (10*\mwidth, \bheight/2) {$\cdots$};

        \begin{scope}[shift={(11*\mwidth, 0)}] 
            \draw[block, fill=white] (0,0) rectangle (1*\mwidth, \bheight);

            \draw[block, fill=white] (1*\mwidth,0) rectangle (2*\mwidth, \bheight);
            
            \draw[block, fill=red!10] (2*\mwidth,0) rectangle (3*\mwidth, \bheight);
            \fill[pattern=north east lines, pattern color=red!30] (2*\mwidth,0) rectangle (3*\mwidth, \bheight);
            
            \node[font=\scriptsize, text=red!80!black, inner sep=0pt] at (2.5*\mwidth, \bheight/2) 
                {\scalebox{0.75}{$\bm{\eta}_{k,3(q_k+1)}$}};

            \draw[topbrace] (0, \bheight+0.1) -- (3*\mwidth, \bheight+0.1) 
                node[midway, above=3pt, font=\footnotesize\bfseries] {$\mathcal{B}_{k,q_k}$};
        \end{scope}

        \def\finalstart{14*\mwidth}
        \def\finalpos{14*\mwidth + \tailgap}
        
        \draw[gap] (\finalstart, 0) rectangle (\finalpos, \bheight);

        \draw[->, thick] (-0.2, 0) -- (\finalpos + 0.5, 0) ;
        
        \draw[tick] (0, 0) -- (0, -0.2);
        
        \draw[tick] (\finalpos, 0) -- (\finalpos, -0.2);

        
        \draw[brace] (0, -0.3) -- (3*\mwidth, -0.3) 
            node[midway, below=2pt, font=\tiny, color=gray!80!black] {Initial Gap};

        \draw[brace] (3*\mwidth, -0.3) -- (6*\mwidth, -0.3) 
            node[midway, below=2pt, font=\tiny] {$3m_k$};

        \draw[brace] (6*\mwidth, -0.3) -- (9*\mwidth, -0.3) 
            node[midway, below=2pt, font=\tiny] {$3m_k$};

        \draw[brace] (11*\mwidth, -0.3) -- (14*\mwidth, -0.3) 
            node[midway, below=2pt, font=\tiny] {$3m_k$};
            
        \draw[brace, decoration={amplitude=3pt}] (14*\mwidth, -0.3) -- (\finalpos, -0.3)
            node[midway, below=2pt, font=\tiny, color=gray!80!black] {Remainder};

        
        \node[below=3pt, font=\small] at (0, -0.9) {$3^{k-1}$};
        \node[below=3pt, font=\small] at (\finalpos, -0.9) {$3^k$};

        \draw[bigbrace] (0, -0.9) -- (\finalpos, -0.9) 
            node[midway, below=3pt, font=\small\bfseries] {Group $\mathcal{D}_k = (3^{k-1}, 3^k]$};

    \end{tikzpicture}

    \caption{Illustration of the block decomposition within group $\mathcal{D}_k$. The shaded regions (red) denote the boundary innovations $\bm{\eta}$ that introduce dependence between adjacent blocks $\mathcal{B}_{k,j}$ and $\mathcal{B}_{k,j+1}$, while the white regions represent internal innovations.}
    \label{fig:block_structure}
\end{figure}

To explicitly characterize the dependency structure induced by the boundary innovations $\bm{\eta}^b$, recall that each outer block $\block_{k,j}$ consists of $3m_k$ elements. Due to the $m$-dependence of the process, the boundary innovation $\bm{\eta}_{k,3j}$ simultaneously influences the last (rightmost) $m_k$ elements of the preceding block $\block_{k,j-1}$ and the first (leftmost) $m_k$ elements of the current block $\block_{k,j}$. Crucially, the central $m_k$ segment of $\block_{k,j}$ relies solely on internal innovations, rendering it independent of any other block $\block_{k,s}$ for $s \neq j$. To formalize this, we decompose $\block_{k,j}$ into three consecutive inner sub-blocks via the following random functions:
\begin{align*}
    \mathbf{B}_{k,3j}(\bm{a}_{k,3j}) & = \sum_{i=1+(3j-1)m_k}^{3jm_k} \tilde{\varphi}_i\big(a_{i+3^{k-1}}, \dots, a_{3jm_k+3^{k-1}}, \varepsilon_{3jm_k+1+3^{k-1}}, \dots, \varepsilon_{i+m_k+3^{k-1}} \big); \\
    \mathbf{B}_{k,3j+1} & = \sum_{i=1+3jm_k}^{(3j+1)m_k} \tilde{\varphi}_i\big(\varepsilon_{i+3^{k-1}}, \dots, \varepsilon_{(3j+1)m_k+3^{k-1}}, \varepsilon_{(3j+1)m_k+1+3^{k-1}}, \dots, \varepsilon_{i+m_k+3^{k-1}} \big); \\
    \mathbf{B}_{k,3j+2}(\bm{a}_{k,3j+3}) & = \sum_{i=1+(3j+1)m_k}^{(3j+2)m_k} \tilde{\varphi}_i\big(\varepsilon_{i+3^{k-1}}, \dots, \varepsilon_{(3j+2)m_k+3^{k-1}}, a_{(3j+2)m_k+1+3^{k-1}}, \dots, a_{i+m_k+3^{k-1}} \big).
\end{align*}
Consequently, the outer block $\block_{k,j}$ defined in \eqref{def_block} can be expressed as the sum of these components, explicitly highlighting the functional dependence on the boundaries:
\begin{equation}\label{eq:def_B}
    \block_{k,j} = \mathbf{B}_{k,3j}(\bm{\eta}_{k,3j}) + \mathbf{B}_{k,3j+1} + \mathbf{B}_{k,3j+2}(\bm{\eta}_{k,3j+3}).
\end{equation}

\noindent Note that $\E[\mathbf{B}_{k,3j+1}] = 0$ by definition. To facilitate the subsequent approximation, we define the conditional mean functions for the boundary-dependent components:
\begin{align*}
    \mathbb{M}_{k,0}^{(j)} (\bm{a}_{k,3j})= \E[\mathbf{B}_{k,3j}(\bm{a}_{k,3j})],\quad \mathbb{M}_{k,2}^{(j)}(\bm{a}_{k,3j+3}) = \E[\mathbf{B}_{k,3j+2}(\bm{a}_{k,3j+3})].
\end{align*}
With these means defined, we introduce the centered process:
\begin{align*}
    \mathcal{Y}_{k,j}(\bm{a}_{k,3j},\bm{a}_{k,3j+3}) = [\mathbf{B}_{k,3j}(\bm{a}_{k,3j}) - \mathbb{M}_{k,0}^{(j)} (\bm{a}_{k,3j})] + \mathbf{B}_{k,3j+1} + [\mathbf{B}_{k,3j+2}(\bm{a}_{k,3j+3}) - \mathbb{M}_{k,2}^{(j)}(\bm{a}_{k,3j+3})].
\end{align*}

By construction, $\mathcal{Y}_{k,j}$ represents the block $\block_{k,j}$ conditioned on the boundary configuration $\bm{\eta}^b = \bm{a}^b$ and centered to have zero mean. Crucially, for any fixed $\bm{a}^b$, the sequence $\{\mathcal{Y}_{k,j}(\bm{a}_{k,3j},\bm{a}_{k,3j+3})\}_{j=1}^{q_k}$ consists of mutually independent random variables. Analogous to the definition of $\mathsf{S}_n^{\Diamond}$ in \eqref{def_S_block}, we define the conditional partial sum process
\begin{align*}
    \mathsf{C}_n(\bm{a}^b) = \sum_{k=K_0}^{h_n-1} \sum_{j=1}^{q_k} \mathcal{Y}_{k,j}(\bm{a}_{k,3j},\bm{a}_{k,3j+3}) + \sum_{j=1}^{\tau_n} \mathcal{Y}_{h_n,j}(\bm{a}_{h_n,3j},\bm{a}_{h_n,3j+3}).
\end{align*}

\noindent Furthermore, we define the associated cumulative mean and variance functions as
\begin{align}
    \mathsf{M}_n(\bm{a}^b) & = \sum_{k=K_0}^{h_n-1} \sum_{j=1}^{q_k} \left[\mathbb{M}_{k,0}^{(j)}(\bm{a}_{k,3j}) + \mathbb{M}_{k,2}^{(j)}(\bm{a}_{k,3j+3})\right] + \sum_{j=1}^{\tau_n} \left[\mathbb{M}_{h_n,0}^{(j)}(\bm{a}_{h_n,3j}) + \mathbb{M}_{h_n,2}^{(j)}(\bm{a}_{h_n,3j+3})\right], \nonumber \\
    \mathsf{Q}_n(\bm{a}^b) & = \sum_{k=K_0}^{h_n-1} \sum_{j=1}^{q_k} \Var_k^{(j)}(\bm{a}_{k,3j},\bm{a}_{k,3j+3}) + \sum_{j=1}^{\tau_n} \Var_{h_n}^{(j)}(\bm{a}_{h_n,3j},\bm{a}_{h_n,3j+3}), 
    \label{variance_process_1}
\end{align}
where $\Var_k^{(j)}(\bm{a}_{k,3j},\bm{a}_{k,3j+3}) = \norm{\mathcal{Y}_{k,j}(\bm{a}_{k,3j},\bm{a}_{k,3j+3})}^2$.

\vspace{3pt}
Since the process $\mathsf{C}_n(\bm{a}^b)$ is a partial sum process consisting of independent variables, we are now able to state our first Gaussian approximation result for $\mathsf{C}_n(\bm{a}^b)$. By Theorem 1 in \cite{sakhanenko2006estimates}, there exists a probability space $(\Omega_{\bm{a}^b}, \mathcal{A}_{\bm{a}^b}, \P_{\bm{a}^b})$ on which we can define random variables $(\mathcal{R}_{k,j}(\bm{a}^b))_{k\geq K_0, 1\leq j \leq q_k} \overset{\mathcal{D}}{=} (\mathcal{Y}_{k,j}(\bm{a}_{k,3j}, \bm{a}_{k,3j+3}))_{k\geq K_0, 1\leq j \leq q_k}$, and a standard Brownian motion $\Bro_{\bm{a}^b}$, such that for all $x>0$,
\begin{align}\label{conditional_gauss_approx}
    \P_{\bm{a}^b}\left[ \max_{N_0\leq i \leq 3^h} \abs{\mathsf{H}_i(\bm{a}^b) - \Bro_{\bm{a}^b}(\mathsf{Q}_i(\bm{a}^b))} \geq c\alpha x \right] \leq \mathsf{U}_{h}(\bm{a}^b,x,\alpha),
\end{align}
where $c$ is an absolute constant,
\begin{align*}
    \mathsf{H}_n(\bm{a}^b) &:= \sum_{k=K_0}^{h_n-1} \sum_{j=1}^{q_k} \mathcal{R}_{k,j}(\bm{a}^b) + \sum_{j=1}^{\tau_n} \mathcal{R}_{h_n,j}(\bm{a}^b), \\
    \mathsf{U}_h(\bm{a}^b,x,\alpha) &:= \sum_{k=K_0}^h \sum_{j=1}^{q_k} \E\left[\abs{\mathcal{Y}_{k,j}(\bm{a}_{k,3j},\bm{a}_{k,3j+3})/x}^{\alpha}\right].
\end{align*}

The result established above represents a conditional Gaussian approximation, valid given the realization $\bm{\eta}^b = \bm{a}^b$. However, our ultimate goal is to obtain an unconditional Gaussian approximation, which requires us to incorporate the randomness inherent in $\bm{\eta}^b$ back into the conditional estimate \eqref{conditional_gauss_approx}. To this end, we replace the fixed parameter $\bm{a}^b$ with the random variable $\bm{\eta}^b$, denoting the resulting unconditional partial sum process by $\mathsf{H}_n(\bm{\eta}^b)$ and the corresponding approximating process by $\Bro_{\bm{\eta}^b}(\mathsf{Q}_n(\bm{\eta}^b))$. Crucially, due to the random time change and dependence on $\bm{\eta}^b$, this approximating process is not a standard Brownian motion. We formalize the resulting unconditional approximation in the following proposition.

\begin{Prop}\label{prop:prob_space} 
    Under the assumptions of Lemma \ref{lem:rosenthal}, there exists a probability space $(\Omega_*,\mathcal{A}_*,\P_*)$ on which we can define $\mathsf{H}_n(\bm{\eta}^b)$ and $\Bro_{\bm{\eta}^b}(\mathsf{Q}_n(\bm{\eta}^b))$ such that \begin{align*}\max_{1\leq i \leq n} \abs{\mathsf{H}_i(\bm{\eta}^b) - \Bro_{\bm{\eta}^b}(\mathsf{Q}_{i}(\bm{\eta}^b))} = \oas{n^{1/p}}.\end{align*}
\end{Prop}

\begin{proof}
    We first construct the probability space on which we can discuss the strong approximation. Define
\begin{align*}
    (\Omega_*,\mathcal{A}_*,\P_*) = (\Omega,\mathcal{A},\P) \times \prod_{\omega\in\Omega} (\Omega_{\bm{\eta}^b(\omega)},\mathcal{A}_{\bm{\eta}^b(\omega)},\P_{\bm{\eta}^b(\omega)}),
\end{align*}
where $(\Omega,\mathcal{A},\P)$ is the probability space on which the random variables $(\eps_i)_{i\in\mathbb{Z}}$ are defined. For a set $A\subset \Omega_*$ with $A\in\mathcal{A}_*$, the probability measure $\P_*$ is defined as
\begin{align*}
    \P_*[A] = \int_{\Omega}\P_{\bm{\eta}^b(\omega)}[A_{\omega}] \P[\d\omega],
\end{align*}
where $A_{\omega}$ is the $\omega$-section of $A$. Here we recall that, for each fixed $\bm{a}^b$, $(\Omega_{\bm{a}^b},\mathcal{A}_{\bm{a}^b},\P_{\bm{a}^b})$ is the probability space carrying $\Bro_{\bm{a}^b}$ and $\mathcal{R}_{k,j}(\bm{a}^b)$ given $\bm{\eta}^b = \bm{a}^b$. On the probability space $(\Omega_*,\mathcal{A}_*,\P_*)$, the random variable $\mathcal{R}_{k,j}(\bm{\eta}^b)$ is defined as
\begin{align*}
    \mathcal{R}_{k,j}(\bm{\eta}^b)(\omega,\theta(\cdot)) = \mathcal{R}_{k,j}(\bm{\eta}^b(\omega))(\theta(\omega)),
\end{align*}
where $(\omega,\theta(\cdot))\in\Omega_*$, and $\theta(\cdot)$ is an element in the product space $\prod_{\tau \in \Omega} \Omega_{\bm{\eta}^b(\tau)}$ such that $\theta(\tau) \in \Omega_{\bm{\eta}^b(\tau)}$. The other random processes $\mathsf{M}_i(\bm{\eta}^b),\ \mathsf{Q}_i(\bm{\eta}^b)$ and $\Bro_{\bm{\eta}^b}(\mathsf{Q}_i(\bm{\eta}^b))$ are defined similarly.

\vspace{3pt}
With the probability space constructed, we proceed to establish the almost sure convergence. By the Borel-Cantelli lemma, it suffices to show $\sum_{h=K_0}^{\infty} \E[\mathsf{U}_h(\bm{\eta}^b,\delta3^{h/p},\alpha)] < \infty,\ \forall \delta>0$. First, by Jensen's inequality, there exists a constant $c_{\alpha}$ such that for $j = 1,\dots,q_k$,
\begin{align}\label{ineq:bound_variance_1}
    \E[\abs{\mathcal{Y}_{k,j}(\bm{\eta}_{k,3j},\bm{\eta}_{k,3j+3})}^{\alpha}] \leq c_{\alpha} \E\left[\abs{\tilde{\mathcal{S}}_{3jm_k,3(j+1)m_k}^{(k)}}^{\alpha}\right] \leq c_{\alpha} \E\left[\max_{3jm_k < l  \leq 3(j+1)m_k} \abs{\tilde{\mathcal{S}}_{3jm_k,l}^{(k)}}^{\alpha} \right].  
\end{align}
Thus, by exchanging the order of summation, we have
\begin{align}
    \sum_{h=K_0}^{\infty} \mathbb{E}[\mathsf{U}_h(\bm{\eta}^b, \delta 3^{h/p}, \alpha)] 
    &\lesssim \sum_{h=K_0}^{\infty} \sum_{k=K_0}^{h} \frac{ \delta^{-\alpha} q_k}{3^{\alpha h/p}} \left( \max_{1\leq j \leq q_k} \E\left[\max_{3jm_k < l  \leq 3(j+1)m_k} \abs{\tilde{\mathcal{S}}_{3jm_k,l}^{(k)}}^{\alpha} \right] \right) \label{summation} \\
    &= \sum_{k=K_0}^{\infty} \left(\sum_{h=k}^{\infty} \frac{ \delta^{-\alpha} q_k}{3^{\alpha h/p}} \right)\cdot\left( \max_{1\leq j \leq q_k} \E\left[\max_{3jm_k < l  \leq 3(j+1)m_k} \abs{\tilde{\mathcal{S}}_{3jm_k,l}^{(k)}}^{\alpha} \right] \right)  \nonumber \\
    &\lesssim \sum_{k=K_0}^{\infty} \frac{3^k \delta^{-\alpha}}{m_k 3^{\alpha k /p}} \left( \max_{1\leq j \leq q_k} \E\left[\max_{3jm_k < l  \leq 3(j+1)m_k} \abs{\tilde{\mathcal{S}}_{3jm_k,l}^{(k)}}^{\alpha} \right] \right) < \infty. \nonumber
\end{align}
By Lemma \ref{lem:rosenthal}, the last series converges. 
\end{proof}

Proceeding from the result in Proposition \ref{prop:prob_space}, our next goal is to establish a standard Gaussian approximation for the composite process $\mathsf{H}_n(\bm{\eta}^b)+ \mathsf{M}_n(\bm{\eta}^b)$, which reconstructs the original target $\mathsf{S}^{\Diamond}$. The motivation for introducing the intermediate object $\Bro_{\bm{\eta}^b}(\mathbf{Q}_n(\bm{\eta}^b))$ is pivotal: by exploiting the independent increment property of Brownian motion, we can translate the decomposition of the variance process $\mathsf{Q}_n(\bm{a}^b)$ into a sum of independent Gaussian variables. This yields a \text{surrogate} for $\Bro_{\bm{\eta}^b}(\mathbf{Q}_n(\bm{\eta}^b))$ composed of independent increments, for which a strong Gaussian approximation is readily derived.

\vspace{3pt}
In the subsequent proofs, we may need to switch the underlying probability space multiple times to establish strong Gaussian approximation results. To simplify the presentation, we adopt the convention of sometimes using the original notation for all processes without explicitly distinguishing the probability spaces. It is implicitly understood that the processes are defined on the space that supports the construction of the independent copies.

\subsubsection{(b) - Unconditional Gaussian Approximation}

\noindent 
We begin by regularizing the variance process $\mathsf{Q}_n(\bm{a}^b)$ defined in \eqref{variance_process_1}. Observe that in each increment $\Var_k^{(j)}(\bm{a}_{k,3j},\bm{a}_{k,3j+3})$, the boundary terms $\bm{a}_{k,3j}$ and $\bm{a}_{k,3j+3}$ are coupled. Our objective is to decouple these dependencies. Let $\mathbf{C}_{k,3j+1}(\bm{\eta}_{k,3j+1}) = \mathbb{E}[\mathbf{B}_{k,3j+1} | \bm{\eta}_{k,3j+1}]$. By conditioning on $\bm{\eta}_{k,3j+1}$, we decompose the conditional variance $\mathbb{V}_k^{(j)}(\bm{a}_{k,3j}, \bm{a}_{k,3j+3})$ as follows:
\begin{align*}
    &\mathbb{V}_k^{(j)}(\bm{a}_{k,3j}, \bm{a}_{k,3j+3}) = \mathbb{E}[\mathcal{Y}_{k,j}(\bm{a}_{k,3j}, \bm{a}_{k,3j+3})^2] \\
     = &\mathbb{E}\bigg[ \Big( \mathcal{Y}_{k,j}(\bm{a}_{k,3j}, \bm{a}_{k,3j+3})  - \mathbb{E}[\mathcal{Y}_{k,j}(\bm{a}_{k,3j}, \bm{a}_{k,3j+3}) | \bm{\eta}_{k,3j+1}] \Big)^2 \bigg] +  \mathbb{E}\bigg[ \Big( \mathbb{E}[\mathcal{Y}_{k,j}(\bm{a}_{k,3j}, \bm{a}_{k,3j+3}) | \bm{\eta}_{k,3j+1}] \Big)^2 \bigg] \\
    = & \mathbb{E}\bigg[ \Big( \mathbf{B}_{k,3j+1} - \mathbf{C}_{k,3j+1}(\bm{\eta}_{k,3j+1}) + \mathbf{B}_{k,3j+2}(\bm{a}_{k,3j+3}) - \mathbb{M}_{k,2}^{(j)}(\bm{a}_{k,3j+3}) \Big)^2 \bigg] \\
    & \quad + \mathbb{E}\bigg[ \Big( \mathbf{B}_{k,3j}(\bm{a}_{k,3j}) - \mathbb{M}_{k,0}^{(j)}(\bm{a}_{k,3j}) + \mathbf{C}_{k,3j+1}(\bm{\eta}_{k,3j+1}) \Big)^2 \bigg] \\
    := & \Var_{k,2j+2}(\bm{a}_{k,3j+3}) + \Var_{k,2j+1}(\bm{a}_{k,3j}). 
\end{align*}

\noindent Let $\Var^{\circ}_{k, j+1}(\bm{a}_{k,3j+3}) = \Var_{k,2j+2}(\bm{a}_{k,3j+3}) + \Var_{k,2j+3}(\bm{a}_{k,3j+3})$. Then, for all $t\in\mathbb{N}^{\times}$, we have the summation identity
\begin{align*}
    \sum_{j=1}^t \Var_{k}^{(j)} (\bm{a}_{k,3j}, \bm{a}_{k,3j+3}) = \Var_{k,3}(\bm{a}_{k,3}) + \sum_{j=1}^{t-1} \Var^{\circ}_{k,j+1}(\bm{a}_{k,3j+3}) + \Var_{k,2t+2}(\bm{a}_{k,3t+3}).
\end{align*}

\noindent Based on this decomposition, we define a surrogate $\mathsf{Q}_n^{\circ}(\bm{a}^b)$ for $\mathsf{Q}_n(\bm{a}^b)$ as follows:
\begin{align*}
    \mathsf{Q}_n^{\circ}(\bm{a}^b) = \sum_{k = K_0}^{h_n-1} \sum_{j=1}^{q_k-1} \Var_{k,j+1}^{\circ}(\bm{a}_{k,3j+3}) + \sum_{j=1}^{\tau_n - 1} \Var^{\circ}_{h_n,j+1}(\bm{a}_{h_n,3j+3}).
\end{align*}
It is readily apparent that in $\mathsf{Q}_n^{\circ}(\bm{a}^b)$, the contributions of $\bm{a}_{k,3j}$ across different $j$ are decoupled, while the process itself preserves the dominant component of $\mathsf{Q}_n(\bm{a}^b)$.

\vspace{3pt}
Guided by the structure of $\mathsf{Q}_n^{\circ}(\bm{a}^b)$, we proceed to decompose the Brownian motion into a sum of independent increments, thereby constructing a tractable surrogate for $\Bro_{\bm{\eta}^b}(\mathsf{Q}_n(\bm{\eta}^b))$. Specifically, invoking the scaling property of Brownian motion, conditional on $\bm{\eta}^b = \bm{a}^b$, there exist i.i.d. standard normal random variables $(Z_{k,j}^{\bm{a}^b})_{k,j}$ such that
\begin{align*}
    \Bro_{\bm{a}^b}(\mathsf{Q}_n(\bm{a}^b)) = \mathsf{R}_n(\bm{a}^b) + \mathsf{E}_n(\bm{a}^b),
\end{align*}
where
\begin{align*}
    \mathsf{R}_n(\bm{a}^b) &= \sum_{k=K_0}^{h_n-1}\sum_{j=1}^{q_k-1} \sqrt{\Var_{k,j+1}^{\circ}(\bm{a}_{k,3j+3})} Z_{k,j}^{\bm{a}^b} + \sum_{j=1}^{\tau_n-1} \sqrt{\Var_{h_n,j+1}^{\circ}(\bm{a}_{h_n,3j+3})} Z_{h_n,j}^{\bm{a}^b},\\
    \mathsf{E}_n(\bm{a}^b) & = \sum_{k = K_0}^{h_n-1} \left( \sqrt{\Var_{k,3}(\bm{a}_{k,3})} Z_{k,0}^{\bm{a}^b} + \sqrt{\Var_{k,2q_k+2}(\bm{a}_{k,3q_k+3})} Z_{k,q_k}^{\bm{a}^b}\right) \\
    &\quad + \left( \sqrt{\Var_{h_n,3}(\bm{a}_{h_n,3})} Z_{h_n,0}^{\bm{a}^b} + \sqrt{\Var_{h_n,2\tau_n+2}(\bm{a}_{h_n,3\tau_n+3})} Z_{h_n,\tau_n}^{\bm{a}^b}\right).
\end{align*}
Here, the term $\mathsf{R}_n(\bm{a}^b)$ corresponds to the surrogate variance process $\mathsf{Q}_n^{\circ}(\bm{a}^b)$. Now, let $(Z_{k,j})_{k,j}$ be a sequence of i.i.d. standard normal random variables that is independent of the innovation sequence $(\eps_j)_{j\in\mathbb{Z}}$, and define
\begin{align*}
    \overline{\mathsf{H}}_n = \sum_{k=K_0}^{h_n-1}\sum_{j=1}^{q_k-1} \sqrt{\Var_{k,j+1}^{\circ}(\bm{\eta}_{k,3j+3})} Z_{k,j} + \sum_{j=1}^{\tau_n-1} \sqrt{\Var_{h_n,j+1}^{\circ}(\bm{\eta}_{h_n,3j+3})} Z_{h_n,j}.
\end{align*}

\noindent Observe that $(\overline{\mathsf{H}}_n | \bm{\eta}^b = \bm{a}^b)_{i\geq1} \overset{\mathcal{D}}{=} (\mathsf{R}_n(\bm{a}^b))_{i\geq1}$. Consequently, we have the distributional equivalence $\overline{\mathsf{H}}_n \overset{\mathcal{D}}{=} \mathsf{R}_n(\bm{\eta}^b)$. Since $\overline{\mathsf{H}}_n$ consists of independent increments, establishing a strong Gaussian approximation for it is straightforward. It remains to verify that the surrogate $\overline{\mathsf{H}}_n$ sufficiently approximates the intermediate object $\Bro_{\bm{\eta}^b}(\mathsf{Q}_n(\bm{\eta}^b))$.

\begin{Prop}\label{prop:main_prop}
    Under the assumptions of Theorem \ref{thm:main}, if there exists $\alpha > p$ such that Condition~\eqref{condition_2} holds, then there exists a probability space on which we can couple the processes $\overline{\mathsf{H}}_n$ and $\Bro_{\bm{\eta}^b}(\mathsf{Q}_n(\bm{\eta}^b))$ such that
    \begin{align*}
        \max_{1\leq i \leq n} \abs{\overline{\mathsf{H}}_n - \Bro_{\bm{\eta}^b}(\mathsf{Q}_n(\bm{\eta}^b))}  = \oas{n^{1/p}}.
    \end{align*}
\end{Prop}

\begin{proof}
    Working on a suitable sample space where the decomposition $\Bro_{\bm{\eta}^b}(\mathsf{Q}_n(\bm{\eta}^b)) = \overline{\mathsf{H}}_n + \mathsf{E}_n(\bm{\eta}^b)$ holds, it suffices to show that $\max_{1\leq i\leq n }\abs{\mathsf{E}_n(\bm{\eta}^b)} = \oas{n^{1/p}}.$ By the Borel-Cantelli lemma, this reduces to verifying that for any $\delta > 0$,
\begin{align}\label{condition_3}
    &\sum_{k=K_0}^\infty \P\left[ \max_{1\leq j \leq q_k} \abs{\sqrt{\Var_{k,2j+1}(\bm{\eta}_{k,3j})}Z_{k,j}^*} > \delta 3^{k/p}\right] < \infty, \\
    \text{and}\quad &\sum_{k=K_0}^\infty \P\left[ \max_{1\leq j \leq q_k} \abs{\sqrt{\Var_{k,2j+2}(\bm{\eta}_{k,3j+3})}Z_{k,j}^*} > \delta 3^{k/p}\right] < \infty, \nonumber
\end{align}
where $(Z_{k,j}^*)_{k,j}$ denotes another sequence of i.i.d. standard normal random variables independent of $(\eps_i)_i$. Without loss of generality, we focus on establishing \eqref{condition_3}, as the proof for the second series follows an identical argument. Observe that $\Var_{k,2j+1}(\bm{a}_{k,3j}) \leq \E[\mathcal{Y}_{k,j}(\bm{a}_{k,3j},\bm{a}_{k,3j+3})^2]$. Consequently, by applying Jensen's inequality, there exists a constant $c_{\alpha}$ such that
\begin{align}\label{ineq:bound_variance_2}
    \E\left[ \abs{\sqrt{\Var_{k,2j+1}(\bm{\eta}_{k,3j})}Z_{k,j}^*}^{\alpha} \right] \lesssim \E\left[ \abs{\sqrt{\Var_{k,2j+1}(\bm{\eta}_{k,3j})}}^{\alpha} \right]\leq c_{\alpha} \E\left[ \abs{\mathcal{Y}_{k,j}(\bm{\eta}_{k,3j},\bm{\eta}_{k,3j+3})}^{\alpha} \right].
\end{align}
Therefore, by applying the union bound followed by Markov's inequality, we obtain
\begin{align*}
    &\sum_{k=K_0}^{\infty}\P\left[ \max_{1\leq j \leq q_k}\abs{\sqrt{\Var_{k,2j+1}(\bm{\eta}_{k,3j})}Z_{k,j}^*} > \delta 3^{k/p}\right] \\
    \leq & \sum_{k=K_0}^{\infty}  q_k \left( \max_{1\leq j \leq q_k} \P\left[ \abs{\sqrt{\Var_{k,2j+1}(\bm{\eta}_{k,3j})}Z_{k,j}^*} > \delta 3^{k/p}\right]\right) \\
    \lesssim  & \sum_{k=K_0}^{\infty} q_k \left( \max_{1\leq j \leq q_k} \frac{\E[\abs{\mathcal{Y}_{k,j}(\bm{\eta}_{k,3j},\bm{\eta}_{k,3j+3})}^{\alpha}]}{\delta^{\alpha} 3^{\alpha k/p}} \right).
\end{align*}
By invoking Lemma \ref{lem:rosenthal} and following an argument similar to that in \eqref{summation}, we conclude that the series above converges.
\end{proof}

Recall that $\mathsf{H}_n(\bm{\eta}^b) + \mathsf{M}_n(\bm{\eta}^b) \overset{\mathcal{D}}{=} \mathsf{S}_n^{\Diamond}$. Combining Proposition \ref{prop:prob_space} and Proposition \ref{prop:main_prop}, we obtain the bound $\max_{1\leq i \leq n}|\mathsf{H}_i(\bm{\eta}^b) - \overline{\mathsf{H}}_i| = \oas{n^{1/p}}$. Consequently, there exists a sample space on which the approximation $\max_{1\leq i\leq n}|\overline{\mathsf{H}}_i+ \mathsf{M}_i(\bm{\eta}^b) - \mathsf{S}_i^{\Diamond}| = \oas{n^{1/p}}$ holds. It therefore suffices to establish a strong Gaussian approximation for the composite process $\overline{\mathsf{H}}_n+ \mathsf{M}_n(\bm{\eta}^b)$. To achieve this, we introduce a modified partial sum process $\mathsf{S}_n^{\natural}$ such that
\begin{align*}
    \mathsf{S}_n^{\natural}:= \sum_{k=K_0}^{h_n-1} \sum_{j=1}^{q_k-1}\mathcal{A}_{k,j} + \sum_{j=1}^{\tau_n-1}\mathcal{A}_{h_n,j},
\end{align*}
where \begin{align}\label{eq:def_A}
    \mathcal{A}_{k,j} = \sqrt{\Var_{k,j+1}^{\circ}(\bm{\eta}_{k,3j+3})} Z_{k,j} + \mathbb{M}_{k,2}^{(j)}(\bm{\eta}_{k,3j+3}) + \mathbb{M}_{k,0}^{(j+1)}(\bm{\eta}_{k,3j+3}).
\end{align}
Observe that the residual error $\mathsf{E}^{\natural}_n := (\overline{\mathsf{H}}_n + \mathsf{M}_n(\bm{\eta}^b)) - \mathsf{S}_n^{\natural}$ consists solely of boundary terms:
\begin{align*}
    \mathsf{E}^{\natural}_n = \sum_{k=K_0}^{h_n-1}\left[\mathbb{M}_{k,0}^{(1)}(\bm{\eta}_{k,3}) + \mathbb{M}_{k,2}^{(q_k)}(\bm{\eta}_{k,3q_k+3})\right] + \left[\mathbb{M}_{h_n,0}^{(1)}(\bm{\eta}_{h_n,3}) + \mathbb{M}_{h_n,2}^{(\tau_n)}(\bm{\eta}_{h_n,3\tau_n+3}) \right].
\end{align*}
It is straightforward to verify that $\E[|\mathbb{M}_{k,0}^{(j)}(\bm{\eta}_{k,3j})|^{\alpha}] \lesssim \E[\max_{3jm_k < l \leq 3(j+1)m_k} |\tilde{\mathcal{S}}^{(k)}_{3jm_k,l}|^{\alpha}]$. Consequently, by employing arguments analogous to the proof of Proposition \ref{prop:main_prop}, we have $\max_{1\leq i \leq n}|\mathsf{E}_i^{\natural}| = \oas{n^{1/p}}$.

\vspace{3pt}
In light of this estimate, the task of approximating the original process reduces to establishing a strong Gaussian approximation for $\mathsf{S}_n^{\natural}$. Crucially, the terms $\mathcal{A}_{k,j}$ constitute a sequence of independent random variables with zero mean. This independence structure allows us to readily apply standard strong Gaussian approximation techniques.

\begin{Prop}
    Under the assumptions of Theorem \ref{thm:main} and Lemma \ref{lem:rosenthal}, we can construct a probability space such that
    \begin{align*}
        \max_{1\leq i\leq n} | \mathsf{S}_i^{\natural} - \Bro(\overline{\mathsf{Q}}_i) |  = \oas{n^{1/p}},
    \end{align*}
    where $\overline{\mathsf{Q}}_n := \sum_{k=K_0}^{h_n-1} \sum_{j=1}^{q_k-1}\norm{\mathcal{A}_{k,j}}^2 + \sum_{j=1}^{\tau_n-1}\norm{\mathcal{A}_{h_n,j}}^2$.
\end{Prop}

\begin{proof}
    By invoking Theorem 1 of Sakhanenko \cite{sakhanenko2006estimates} and utilizing arguments analogous to those in \eqref{conditional_gauss_approx} and Proposition \ref{prop:prob_space}, the proof is nearly complete; it remains only to bound the moment $\E[|\mathcal{A}_{k,j}|^{\alpha}]$. Observe that
    \begin{align*}
        \E[|\mathcal{A}_{k,j}|^{\alpha}] &\lesssim \E\left[ \abs{\sqrt{\Var^{\circ}_{k,j+1}(\bm{\eta}_{k,3j+3})}}^{\alpha}\right] + \E\left[\abs{\mathbb{M}_{k,2}^{(j)}(\bm{\eta}_{k,3j+3})}^{\alpha}\right] + \E\left[\abs{\mathbb{M}_{k,0}^{(j+1)}(\bm{\eta}_{k,3j+3})}^{\alpha}\right]  \\
        &\lesssim \E\left[ \abs{\sqrt{\Var_{k,2j+2}(\bm{\eta}_{k,3j+3})}}^{\alpha}\right] + \E\left[ \abs{\sqrt{\Var_{k,2j+3}(\bm{\eta}_{k,3j+3})}}^{\alpha}\right] \\ &\quad+ \E\left[\abs{\mathbb{M}_{k,2}^{(j)}(\bm{\eta}_{k,3j+3})}^{\alpha}\right] + \E\left[\abs{\mathbb{M}_{k,0}^{(j+1)}(\bm{\eta}_{k,3j+3})}^{\alpha}\right] \\
        & \lesssim \E\left[\max_{3jm_k < l \leq 3(j+1)m_k} \abs{\tilde{\mathcal{S}}^{(k)}_{3jm_k,l}}^{\alpha}\right] + \E\left[\max_{3(j+1)m_k < l \leq 3(j+2)m_k} \abs{\tilde{\mathcal{S}}^{(k)}_{3(j+1)m_k,l}}^{\alpha}\right],
    \end{align*}
    where the variance terms are controlled by \eqref{ineq:bound_variance_1} and \eqref{ineq:bound_variance_2}. Substituting this upper bound into the derivation steps outlined in Proposition \ref{prop:prob_space} concludes the proof.
\end{proof}

Under the assumptions of Theorem \ref{thm:main}, and provided that Conditions \eqref{condition_1} and \eqref{condition_2} hold, by combining the established results, we can construct a sample space such that
\begin{align}
    \max_{1\leq i\leq n}| \mathsf{S}_i - \Bro(\overline{\mathsf{Q}}_i)| = \oas{n^{1/p}}.
\end{align}
This confirms that we have successfully established the strong Gaussian approximation for non-stationary processes with the optimal rate, marking a major milestone in our proof.

\subsubsection{(c) - Regularization of Variance}

\noindent Although we have established a strong Gaussian approximation with an optimal rate, the current result is still essentially existential, since the variance process $\overline{\mathsf{Q}}_i$ has not yet been explicitly identified. Recall that our ultimate goal is to establish
\[
\max_{1\leq i \leq n} | \mathsf{S}_i - \Bro(\E[\mathsf{S}_i^2]) |
= \oas{n^{1/p}\sqrt{\log n}} .
\]
In this section, we regularize the variance by showing that $\overline{\mathsf{Q}}_i$ can be replaced by $\E[\mathsf{S}_i^2]$ up to a negligible error. We begin with the following lemma, which characterizes the variance of $\mathcal{A}_{k,j}$.


\begin{Lem}\label{lem:variance_A}
    The variance of $\mathcal{A}_{k,j}$ defined in \eqref{eq:def_A} is given by
    \begin{align*}
        \norm{\mathcal{A}_{k,j}}^2
        &= \norm{\mathbf{B}_{k,3j+1}}^2 + \norm{\mathbf{B}_{k,3j+2}}^2 + \norm{\mathbf{B}_{k,3j+3}}^2 \\
        &\quad + 2\E[\mathbf{B}_{k,3j+1}\mathbf{B}_{k,3j+2}]
        + 2\E[\mathbf{B}_{k,3j+2}\mathbf{B}_{k,3j+3}]
        + 2\E[\mathbf{B}_{k,3j+3}\mathbf{B}_{k,3j+4}] \\
        &\quad - \norm{\mathbf{C}_{k,3j+1}(\bm{\eta}_{k,3j+1})}^2
        + \norm{\mathbf{C}_{k,3j+4}(\bm{\eta}_{k,3j+4})}^2 .
    \end{align*}
    Here, for notational simplicity, we suppress the dependence on the boundary variables and write
    $\mathbf{B}_{k,3j+1}(\bm{\eta}_{k,3j})$ as $\mathbf{B}_{k,3j+1}$ and
    $\mathbf{B}_{k,3j+2}(\bm{\eta}_{k,3j+3})$ as $\mathbf{B}_{k,3j+2}$.
\end{Lem}

According to the explicit form of $\norm{\mathcal{A}_{k,j}}^2$, it is clear that this quantity is closely related to the variance of the outer blocks $\norm{\mathcal{B}_{k,j}}^2$. As a first step toward regularizing the variance, we compare $\overline{\mathsf{Q}}_n$ with $\E[(\mathsf{S}_n^{\Diamond})^2]$ by examining their variance components within each group $\mathcal{D}_k$. Note that the variance of $\sum_{j=1}^{q_k} \mathcal{B}_{k,j}$ can be written as
\begin{align*}
\left(\norm{\mathbf{B}_{k,3}}^2 + \norm{\mathbf{B}_{k,4}}^2 + \cdots + \norm{\mathbf{B}_{k,3q_k+2}}^2 \right)
+ 2\left( \E[\mathbf{B}_{k,3}\mathbf{B}_{k,4}] + \E[\mathbf{B}_{k,4}\mathbf{B}_{k,5}] + \cdots + \E[\mathbf{B}_{k,3q_k+1} \mathbf{B}_{k,3q_k+2}] \right).
\end{align*}
On the other hand, by Lemma~\ref{lem:variance_A}, the contribution of group $\mathcal{D}_k$ to $\overline{\mathsf{Q}}_n$ is given by
\begin{align*}
\sum_{j=1}^{q_k-1} \norm{\mathcal{A}_{k,j}}^2
&= \left(\norm{\mathbf{B}_{k,4}}^2 + \norm{\mathbf{B}_{k,5}}^2 + \cdots + \norm{\mathbf{B}_{k,3q_k}}^2 \right)
- \norm{\mathbf{C}_{k,4}(\bm{\eta}_{k,4})}^2
+ \norm{\mathbf{C}_{k,3q_k+1}(\bm{\eta}_{k,3q_k+1})}^2 \\
&\quad + 2\left( \E[\mathbf{B}_{k,4}\mathbf{B}_{k,5}] + \E[\mathbf{B}_{k,5}\mathbf{B}_{k,6}] + \cdots + \E[\mathbf{B}_{k,3q_k}\mathbf{B}_{k,3q_k+1}] \right).
\end{align*}
Therefore, the difference between the two variance expressions equals
\begin{align*}
&\norm{\mathbf{B}_{k,3}}^2 + \norm{\mathbf{B}_{k,3q_k+1}}^2 + \norm{\mathbf{B}_{k,3q_k+2}}^2
- \norm{\mathbf{C}_{k,4}(\bm{\eta}_{k,4})}^2
+ \norm{\mathbf{C}_{k,3q_k+1}(\bm{\eta}_{k,3q_k+1})}^2 \\
&\quad + 2\E[\mathbf{B}_{k,3}\mathbf{B}_{k,4}]
+ 2\E[\mathbf{B}_{k,3q_k+1}\mathbf{B}_{k,3q_k+2}].
\end{align*}
Using Jensen's inequality to control $\mathbf{C}_{k,l}(\bm{\eta}_{k,l})$ and Cauchy's inequality to control the cross terms, the above difference is bounded by
\begin{align*}
9 \max_{3\leq l \leq 3q_k+2} \norm{\mathbf{B}_{k,l}}^2
\;\lesssim\;
\sup_{s}\,
\norm{\max_{1\leq l \leq m_k}
\bigl|\tilde{X}_{k,s+1}+\cdots+\tilde{X}_{k,s+l}\bigr|}^2.
\end{align*}

\vspace{1pt}
\begin{Lem}\label{lem:maximal_ineq}
    Under Assumption~\ref{asm:dep_decay}, for any $r\geq2$ and
    $m\geq 1$, it holds that
    \begin{align*}
        \sup_{s}
        \norm{
        \max_{1\leq l\leq m}
        \bigl|X_{s+1}+\cdots+X_{s+l}\bigr|
        }_r
        \leq
        \frac{r}{\sqrt{r-1}}\,
        \Theta_{0,r}\sqrt{m},
    \end{align*}
    where the same bound continues to hold if $X_i$ is replaced by
    $X_i^{\oplus}-\E[X_i^{\oplus}]$ or by $\tilde{X}_{k,i}$.
\end{Lem}

By applying Lemma~\ref{lem:maximal_ineq} with $r=2$, we can bound the variance difference between 
$\sum_{j=1}^{q_k-1}\mathcal{A}_{k,j}$ and $\sum_{j=1}^{q_k}\mathcal{B}_{k,j}$ by 
$C\, m_k \Theta_{0,p}^2$, where the constant $C$ is independent of $k$. 
Define $\alpha_k = 3^{k-k\alpha/p} m_k^{\alpha/2-1}$. 
Under Condition~\eqref{condition_2}, we have $\alpha_k \to 0$ as $k \to \infty$. 
Consequently,
\begin{align}
\max_{1\le i\le n}\bigl|\overline{\mathsf{Q}}_i-\E[(\mathsf{S}_i^{\Diamond})^2]\bigr|
&\lesssim \Theta_{0,p}^2 \sum_{k=1}^{h_n} m_k  \label{eq:variance_group_bound}\\
&\lesssim \sum_{k=1}^{h_n} \alpha_k^{1/(\alpha/2-1)} 
3^{k(\alpha/p-1)/(\alpha/2-1)}
= o\!\left(n^{(\alpha/p-1)/(\alpha/2-1)}\right). \nonumber
\end{align}
Recall that $\mathsf{S}_i^{\Diamond}$ differs from $\tilde{\mathsf{S}}_i$ only by removing a subset of the
$\tilde{X}_{k,i}$ terms. By the same argument, we therefore obtain
\begin{align}
\max_{1\le i \le n}\bigl|\overline{\mathsf{Q}}_i-\E[\tilde{\mathsf{S}}_i^2]\bigr|
= o\!\left(n^{(\alpha/p-1)/(\alpha/2-1)}\right). \label{eq:var_regular_1}
\end{align}

We now have a uniform control on the deviation of the variance processes.
To transfer this variance deviation into a deviation of the associated Brownian motions,
we invoke the following lemma.

\begin{Lem}\label{lem:increment_property}
    Suppose that $\{\mathsf{V}_i\}_{i\ge1}$ and $\{\overline{\mathsf{V}}_i\}_{i\ge1}$ are
    nonnegative processes. If there exists $\beta\in(0,1)$ such that
    \[
        \max_{1\leq i \leq n}|\mathsf{V}_i - \overline{\mathsf{V}}_i|
        = o(n^{\beta}),
    \]
    then it holds that
    \begin{align*}
        \max_{1\leq i \leq n}
        \bigl|\Bro(\mathsf{V}_i) - \Bro(\overline{\mathsf{V}}_i)\bigr|
        = \oas{n^{\beta/2}\sqrt{\log n}}.
    \end{align*}
\end{Lem}

\noindent With the help of Lemma \ref{lem:increment_property}, it is easy to check that \begin{align*}
    \max_{1\le i \le n}\abs{\Bro(\overline{\mathsf{Q}}_i) - \Bro(\E[\tilde{\mathsf{S}}_i^2])}
= \oas{n^{(\alpha/p-1)/(\alpha-2)}\sqrt{\log n}} = \oas{n^{1/p}}.
\end{align*}


Our next step toward regularizing the variance is to obtain a uniform bound on
\(
\max_{1\le i\le n}\bigl|\E[\tilde{\mathsf{S}}_i^2]-\E[(\mathsf{S}_i^{\oplus})^2]\bigr|.
\)
Compared with the previous variance regularization step, the key new issue here is that
$\mathsf{S}_n^{\oplus}$ is no longer $m$-dependent, and hence cross-term contributions
must be handled explicitly. In particular,
\begin{align*}
    \E[\tilde{\mathsf{S}}_n^2]
    &= \sum_{k=1}^{h_n-1} \E\!\left[\tilde{\mathcal{W}}_k^2\right]
       + \E\!\left[\tilde{\mathcal{W}}_{h_n,n-3^{h_n-1}}^2 \right]
       + o\!\left(n^{(\alpha/p-1)/(\alpha/2-1)}\right), \\ 
    \E[(\mathsf{S}^{\oplus}_n)^2]
    &= \sum_{k=1}^{h_n-1} \E\!\left[\mathcal{W}_k^2\right]
       + \E\!\left[\mathcal{W}_{h_n,n-3^{h_n-1}}^2 \right]
       + 2\sum_{1\le k_1< k_2 \le h_n} \E\!\left[\mathcal{W}_{k_1}\mathcal{W}_{k_2}\right],
\end{align*}
where we write $\mathcal{W}_k := \mathcal{W}_{k,3^k-3^{k-1}}$ for brevity. We begin by stating the covariance bound used to control dependence across different groups.

\begin{Lem}\label{lem:covariance}
Under Assumption~\ref{asm:dep_decay}, it holds that
\(
\abs{\E[X_s X_{s+l}]} \leq \sum_{h=0}^{\infty}\delta_p(h)\delta_p(h+l)
\)
for all $s,l\ge1$. Moreover, the same bound remains valid if
$X_i$ is replaced by $X_i^{\oplus}-\E[X_i^{\oplus}]$.
\end{Lem}

\begin{Prop}\label{prop:variance_regularization_1}
Under the assumptions of Theorem \ref{thm:main}, it holds that
\begin{align}
    \max_{1\le i\le n}\bigl|\E[\tilde{\mathsf{S}}_i^2]-\E[(\mathsf{S}_i^{\oplus})^2]\bigr|
    \lesssim \log n + \sum_{k=1}^{h_n} 3^k\,\Theta_{m_k,p} + o\!\left(n^{(\alpha/p-1)/(\alpha/2-1)}\right).
\end{align}
\end{Prop}

\begin{proof}
Using the group-wise decompositions above, the variance difference can be written as the sum
of the difference between the group-wise squared terms and the contribution of the cross terms.
By Lemmas~\ref{lem:inequality_1} and~\ref{lem:maximal_ineq}, the group-wise squared terms satisfy
\begin{align}\label{eq:block_var_diff}
    \max_{1\leq h \leq h_n}\abs{\sum_{k=1}^{h}
    \Bigl(\E[\tilde{\mathcal{W}}_k^2]-\E[\mathcal{W}_k^2]\Bigr)}
    \le \sum_{k=1}^{h_n}
    \norm{\tilde{\mathcal{W}}_k-\mathcal{W}_k}\,
    \norm{\tilde{\mathcal{W}}_k+\mathcal{W}_k}
    \lesssim \sum_{k=1}^{h_n} 3^k\,\Theta_{m_k,p}.
\end{align}
It therefore remains to bound the cross terms
\(
\E[\mathcal{W}_{k_1}\mathcal{W}_{k_2}].
\)
For notational convenience, let
$c_p(l)=\sum_{h=0}^{\infty}\delta_p(h)\delta_p(h+l)$.
By Lemma~\ref{lem:covariance}, for any $k\le h_n$,
\begin{align*}
    \sum_{k_2>k}\abs{\E[\mathcal{W}_{k}\mathcal{W}_{k_2}]}
    &\le \sum_{i_1\in\mathcal{D}_k}\sum_{i_2>3^k}c_p(i_2-i_1)
    \le \sum_{l=1}^{\infty}l\,c_p(l),
\end{align*}
where the last inequality follows since, for each fixed lag $l$,
there are at most $l$ pairs $(i_1,i_2)$ satisfying
$i_1\in\mathcal{D}_k$, $i_2>3^k$, and $i_2-i_1=l$. Moreover,
\begin{align*}
    \sum_{l=1}^{\infty}l\,c_p(l)
    &=\sum_{r=1}^{\infty}\sum_{l=r}^{\infty}c_p(l)
    =\sum_{r=1}^{\infty}\sum_{h=0}^{\infty}\delta_p(h)\Theta_{h+r,p}
    \le \Theta_{0,p}\sum_{r=1}^{\infty}\Theta_{r,p}
    \lesssim 1,
\end{align*}
where the last inequality follows from Assumption~\ref{asm:dep_decay} and $A>1$.
Consequently,
\[
2\sum_{1\le k_1<k_2\le h_n}
\abs{\E[\mathcal{W}_{k_1}\mathcal{W}_{k_2}]}
\lesssim h_n\lesssim \log n.
\]
The result follows by combining this estimate with \eqref{eq:block_var_diff}.
\end{proof}

By Lemma~\ref{lem:increment_property}, in order to ensure that
\(\max_{1\le i\le n}
\bigl|\Bro(\E[\tilde{\mathsf{S}}_i^2])-\Bro(\E[(\mathsf{S}_i^{\oplus})^2])\bigr|
= \oas{n^{1/p}},\)
it suffices to require that
\begin{align}\label{condition_4}
    \sum_{k=1}^{h_n} 3^k\,\Theta_{m_k,p}
    = o\!\left(\frac{n^{2/p}}{\log n}\right).
\end{align}
This condition essentially requires that
the sequence $(m_k)$ diverges sufficiently fast so that the approximation error
$\sum_{k=1}^{h_n}3^k\Theta_{m_k,p}$ is sufficiently small. We take this condition as given for now.

\vspace{3pt}
Our final step for regularizing the variance is to control
$\max_{1\leq i \leq n}|\E[\mathsf{S}_i^2] - \E[(\mathsf{S}_i^{\oplus})^2]|$.
Again, since there is no structure such as $m$-dependence to leverage, we need to expand the variance difference and bound each resulting cross-term directly.

\begin{Prop}\label{prop:variance_regularization_2}
    Under the assumptions of Theorem \ref{thm:main}, it holds that \begin{align}
        \max_{1\leq i \leq n}|\E[\mathsf{S}_i^2] - \E[(\mathsf{S}_i^{\oplus})^2]|  = o(n^{2/p}).
    \end{align}
\end{Prop}

\begin{proof}
    Within the proof of Proposition~\ref{prop:truncation_error}, we can obtain that
    $\left( \sum_{i=1}^n \E[X_i^{\oplus}] \right)^2 = o(n^{2/p})$,
    so it reduces to controlling
    $\max_{1\leq r\leq n}|\E[\mathsf{S}_r^2  - (\sum_{i=1}^r X_i^{\oplus})^2]|$.
    By bounding $\max_{1\leq r\leq n}\E[\mathsf{S}_r^2  - (\sum_{i=1}^r X_i^{\oplus})^2]$ as
    $\sum_{i_1=1}^n \sum_{i_2=1}^n |\E[X_{i_1}X_{i_2} - X_{i_1}^{\oplus} X_{i_2}^{\oplus}]|$,
    it remains to bound the generic term, namely
    \begin{align*}
       |\E[X_{i_1}X_{i_2}] - \E[X_{i_1}^{\oplus} X_{i_2}^{\oplus}]| \leq |\E[(X_{i_1} - X_{i_1}^{\oplus}) X_{i_2}]| + |\E[X_{i_1}^{\oplus}(X_{i_2} - X_{i_2}^{\oplus})]|. 
    \end{align*}
    We treat the two terms on the right-hand side separately.

    \vspace{3pt}
    \noindent Without loss of generality, assume that $i_2 \geq i_1$. Recall the causal representation \eqref{eq:g_function} for $X_{i_2}$, and define
    $X_{i_2,[i_1]} = \varphi_{i_2}(\eps_{i_2},\dots,\eps_{i_1+1},\eps_{i_1}',\eps_{i_1-1}',\dots)$,
    where $(\eps_{i}')_{i\in\mathbb{Z}}$ is an independent copy of $(\eps_{i})_{i\in\mathbb{Z}}$.
    Then $X_{i_2,[i_1]}$ is mean-zero and independent of $X_{i_1}$.
    Let $q = p/(p-1)$. By H\"older's inequality,
    \begin{align*}
        |\E[(X_{i_1} - X_{i_1}^{\oplus}) X_{i_2}]|
        = |\E[(X_{i_1} - X_{i_1}^{\oplus}) (X_{i_2} - X_{i_2,[i_1]})]| \leq \norm{X_{i_1} - X_{i_1}^{\oplus}}_q \norm{X_{i_2} - X_{i_2,[i_1]}}_p.
    \end{align*}

    \noindent By employing the same argument as in the proof of Lemma~\ref{lem:covariance}, we have
    $\norm{X_{i_2} - X_{i_2,[i_1]}}_p \leq \Theta_{i_2-i_1,p}$.
    To bound $\norm{X_{i_1} - X_{i_1}^{\oplus}}_q$, assume that $i_1\in\mathcal{D}_k$.
    Then, by the definition of truncation,
    $|X_{i_1} - X_{i_1}^{\oplus}| \leq |X_{i_1}|\1\{|X_{i_1}| > 3^{k/p}\}$.
    Moreover, for any $a>0$, $\E[\abs{X}^q\1\{\abs{X}>a\}] \leq a^{q-p}\E[\abs{X}^p\1\{\abs{X}>a\}]$.
    Applying Assumption~\ref{asm:UI} yields
    \[
    \sup_{i\in\mathbb{Z}}\E[|X_{i}|^p\1\{|X_{i}|\geq 3^{k/p}\}]
    \lesssim \E[Q^p \1\{Q\geq 3^{k/p}\}]=: u_k.
    \]
    Consequently, $\sup_{i_1\in\mathcal{D}_k}\norm{X_{i_1} - X_{i_1}^{\oplus}}_q^q \lesssim u_k(3^{k/p})^{q-p}$, and hence
    $\sup_{i_1\in\mathcal{D}_k}\norm{X_{i_1} - X_{i_1}^{\oplus}}_q \lesssim 3^{k(2/p-1)}u_k^{1/q}$.
    Combining the above bounds and using Assumption \ref{asm:dep_decay}, we have \begin{align*}
        \sum_{i_2 = i_1}^n |\E[(X_{i_1} - X_{i_1}^{\oplus}) X_{i_2}]| &\leq \left( \sup_{i_1\in\mathcal{D}_k}\norm{X_{i_1} - X_{i_1}^{\oplus}}_q \right) \left( \sum_{i_2=1}^n \norm{X_{i_2} - X_{i_2,[i_1]}}_p \right) \\ &\lesssim 3^{k(2/p-1)}u_k^{1/q}\left( \sum_{i=1}^{\infty} \Theta_{i,p} \right) \lesssim 3^{k(2/p-1)}u_k^{1/q}.
    \end{align*}

    \noindent Summing over $i_1$, and using the fact that $u_k\to 0$ as $k\to \infty$, we obtain \begin{align}\label{eq:final_term_1}
        \sum_{i_1 = 1}^n \sum_{i_2 = i_1}^n  |\E[(X_{i_1} - X_{i_1}^{\oplus}) X_{i_2}]| &\leq \sum_{k\leq h_n} \sum_{i_1\in\mathcal{D}_k}\left(\sum_{i_2 = i_1}^n |\E[(X_{i_1} - X_{i_1}^{\oplus}) X_{i_2}]|\right) \nonumber \\ & \lesssim \sum_{k\leq h_n} |\mathcal{D}_k|3^{k(2/p-1)}u_k^{1/q} \leq \sum_{k\leq h_n} 3^{2k/p}u_k^{1/q} = o(n^{2/p}).
    \end{align}

    \noindent For controlling $\E[X_{i_1}^{\oplus}(X_{i_2} - X_{i_2}^{\oplus})]$, for notational convenience, we denote the remainder term $X_{i} - X_i^{\oplus}$ as $r(X_i)$. Recall that $X_{i_2,[i_1]}$ has been defined above; in particular, $X_{i_2,[i_1]}$ is independent of $X_{i_1}^{\oplus}$ and has the same marginal distribution as $X_{i_2}$. Using $\E[r(X_i)] = -\E[X_i^{\oplus}]$, we have \begin{align*}
        |\E[X_{i_1}^{\oplus}(X_{i_2} - X_{i_2}^{\oplus})]| &= |\E[X_{i_1}^{\oplus}(r(X_{i_2}) - r(X_{i_2,[i_1]}))] + \E[X_{i_1}^{\oplus} r(X_{i_2,[i_1]})]| \\&= |\E[X_{i_1}^{\oplus}(r(X_{i_2}) - r(X_{i_2,[i_1]}))] + \E[X_{i_1}^{\oplus}]\E[r(X_{i_2})]| \\ & \leq |\E[X_{i_1}^{\oplus}(r(X_{i_2}) - r(X_{i_2,[i_1]}))]| + |\E[X_{i_1}^{\oplus}]\E[X_{i_2}^{\oplus}]|.
    \end{align*}
    \noindent In particular, when summing over indices, the second term contributes
    $\sum_{i_1=1}^n \sum_{i_2=1}^n |\E[X_{i_1}^{\oplus}]\E[X_{i_2}^{\oplus}]| = \left(\sum_{i= 1}^n |\E[X_{i}^{\oplus}]| \right)^2 = o(n^{2/p})$.
    For the first term, by H\"older's inequality, we have  $|\E[X_{i_1}^{\oplus}(r(X_{i_2}) - r(X_{i_2,[i_1]}))]| \leq \norm{X_{i_1}^{\oplus}}_p \norm{r(X_{i_2}) - r(X_{i_2,[i_1]})}_q$. By exploring the structure of function $r(x)-r(y)$, and using H\"older's inequality again (assuming $i_2\in\mathcal{D}_k$) \begin{align*}
        \E[|r(X_{i_2}) - r(X_{i_2,[i_1]})|^q] & \leq \E\left[ |X_{i_2}-X_{i_2,[i_1]}|^q\left( \1\{|X_{i_2}|\geq 3^{k/p}\} + \1\{|X_{i_2,[i_1]}|\geq 3^{k/p}\}\right)\right] \\ & \leq \left( \E[|X_{i_2}-X_{i_2,[i_1]}|^p]\right)^{1/(p-1)}\cdot \norm{\1\{|X_{i_2}|\geq 3^{k/p}\} + \1\{|X_{i_2,[i_1]}|\geq 3^{k/p}\}}_{\frac{p-1}{p-2}} \\ & \lesssim \Theta_{i_2 - i_1,p}^q \norm{\1\{|X_{i_2}|\geq 3^{k/p}\}}_{\frac{p-1}{p-2}} \leq \Theta_{i_2 - i_1,p}^q \left(\P\left[ |X_{i_2}| \geq 3^{k/p} \right] \right)^{\frac{p-2}{p-1}} \\ & \leq \Theta_{i_2 - i_1,p}^q \left( 3^{-k} u_k\right)^{\frac{p-2}{p-1}},
    \end{align*}
    where the last inequality is given by Markov's inequality and Assumption \ref{asm:UI}. Thus we have \begin{align*}
        |\E[X_{i_1}^{\oplus}(r(X_{i_2}) - r(X_{i_2,[i_1]}))]| \leq \norm{X_{i_1}^{\oplus}}_p \norm{r(X_{i_2}) - r(X_{i_2,[i_1]})}_q \leq \Theta_{i_2- i_1,p} 3^{k(2/p-1)} u_k^{1-2/p}.
    \end{align*}

    \noindent Repeating the same summation argument as in \eqref{eq:final_term_1}, we can conclude that $\sum_{i_1=1}^n \sum_{i_2= i_1}^n |\E[X_{i_1}^{\oplus}(r(X_{i_2}) - r(X_{i_2,[i_1]}))]| = o(n^{2/p})$. Combining this result with \eqref{eq:final_term_1} and $\sum_{i_1=1}^n \sum_{i_2=1}^n |\E[X_{i_1}^{\oplus}]\E[X_{i_2}^{\oplus}]| = o(n^{2/p})$ completes the proof. 
\end{proof}

By applying Lemma~\ref{lem:increment_property}, it follows that
\(\max_{1\le i\le n}
\bigl|\Bro(\E[{\mathsf{S}}_i^2])-\Bro(\E[(\mathsf{S}_i^{\oplus})^2])\bigr|
= \oas{n^{1/p}\sqrt{\log n}}\). Thus, combining the main results established above, we conclude that under Conditions~\eqref{condition_1}, \eqref{condition_2} and \eqref{condition_3}, and under the setting of Theorem~\ref{thm:main}, we have the following strong Gaussian approximation results for nonstationary processes:
\begin{align}
    \max_{1\le i\le n}
    \bigl|\mathsf{S}_i-\Bro(\E[(\mathsf{S}_i^{\oplus})^2])\bigr|
    = \oas{n^{1/p}},\quad \max_{1\le i\le n}
    \bigl|\mathsf{S}_i-\Bro(\E[(\mathsf{S}_i)^2])\bigr|
    = \oas{n^{1/p}\sqrt{\log n}}.
\end{align}

\noindent One may notice that our strong Gaussian approximation maintains the optimal rate $n^{1/p}$ up to Proposition~\ref{prop:variance_regularization_2}, while the final step introduces an additional $\oas{\sqrt{\log n}}$ factor through the variance regularization. Removing this $\sqrt{\log n}$ term generally requires an additional non-singularity condition on the partial sum process; see \cite{bonnerjee2024gaussian}. In contrast, for stationary processes, such a non-singularity condition is automatically satisfied under Assumption~\ref{asm:dep_decay}, and hence \cite{berkes2014kmt} achieve the optimal approximation. 

\vspace{3pt}
On the other hand, we still obtain an optimal strong Gaussian approximation at the rate $\oas{n^{1/p}}$ when the Brownian time change is indexed by the truncated partial-sum variance $\E[(\mathsf{S}_i^{\oplus})^2]$. This form is also practically useful: the truncation level is pre-specified, and therefore $\E[(\mathsf{S}_i^{\oplus})^2]$ can be estimated from the sample in a straightforward manner.

\vspace{3pt}
In the above proof, besides the assumptions of Theorem~\ref{thm:main}, we also introduced three additional conditions \eqref{condition_1}, \eqref{condition_2} and \eqref{condition_4}, which impose further restrictions on the dependency decay rate $A$ in Assumption~\ref{asm:dep_decay}. At this stage, we have not verified whether these conditions are implied by the assumptions of Theorem~\ref{thm:main}, so they should be viewed as auxiliary conditions required by the proof. In the next step, we will determine a weak restriction on $A$ (i.e., the strongest admissible dependence) under which these three conditions can be satisfied.

\subsubsection{(c) - Dependency Decay Rate Calibration}

\noindent In this last step of the proof of Theorem~\ref{thm:main}, we determine a range of $A$ under which Conditions \eqref{condition_1}, \eqref{condition_2}, and \eqref{condition_4} can all be satisfied. For simplicity, we take $m_k = 3^{\beta k}$ with $\beta\in(0,1)$. Under this choice, we will derive the admissible ranges of $\beta$ and $A$ jointly.

\vspace{3pt}
Substituting $m_k$ into Condition \eqref{condition_4}, we obtain the equivalent constraint
\begin{align}\label{constraint_1}
    \frac{2}{\beta}\left(\frac{1}{2} - \frac{1}{p} \right) < A. 
\end{align}
On the other hand, substituting the same $m_k$ into Condition \eqref{condition_1} yields the constraint $(1/2 - 1/p)<A\beta$. This requirement is already implied by \eqref{constraint_1}, and hence it suffices to work with \eqref{constraint_1} in what follows.

\vspace{3pt}
Next, substituting $m_k$ into Condition \eqref{condition_2}, the constraint $\Lambda_p(\alpha)<\infty$ is equivalent to
\begin{align}\label{constraint_2}
    1 - \frac{\alpha}{p} + \beta\left( \frac{\alpha}{2} - 1\right) < 0.
\end{align}
To ensure $\Xi_p(\alpha)<\infty$, we relate it to $\Theta_{m,p}$. Under Assumption~\ref{asm:dep_decay}, an upper bound for $\Xi_p(\alpha)$ can be obtained via H\"older's inequality:
\begin{align*}
    \Xi_p(\alpha) &\leq \sum_{k = 1}^{\infty} 3^{k(1/2-1/\alpha)} \sum_{h=3^{k-1}}^{3^k-1} \delta_p(h)^{p/\alpha} \leq \sum_{k = 1}^{\infty} 3^{k(1/2-1/\alpha)} 3^{(k-1)(1-p/\alpha)} \left(\sum_{h=3^{k-1}}^{3^k-1} \delta_p(h)\right)^{p/\alpha} \\ 
    &\leq \sum_{k=1}^{\infty} 3^{k(3/2 - 1/\alpha - p/\alpha)} \Theta_{3^{k-1},p}^{p/\alpha}\lesssim \sum_{k=1}^{\infty} 3^{k(3/2 - 1/\alpha - p/\alpha - Ap/\alpha)}.
\end{align*}
Therefore, a sufficient condition for $\Xi_p(\alpha)<\infty$ is
\begin{align}\label{constraint_3}
    \alpha < \frac{2(1+p+Ap)}{3}.
\end{align}

\vspace{3pt}
Combining \eqref{constraint_2} and \eqref{constraint_3}, we obtain the admissible range for $\alpha$:
\begin{align*}
    (1-\beta)/\left(\frac{1}{p} - \frac{\beta}{2} \right) < \alpha < \frac{2(1+p+Ap)}{3}.
\end{align*}
Notice that the lower bound of $\alpha$ is an increasing function of $\beta$ when $\beta \in(0,2/p)$. Consequently, the above constraint yields the following upper bound on $\beta$:
\begin{align*}
    \beta < \frac{2Ap + 2-p}{p(p+Ap-2)}.
\end{align*}
Substituting this bound into \eqref{constraint_1} and solving for the range of $A$, we obtain
\begin{align*}
    A > A_0:= \max\left\{ \frac{p^2 - p - 2 + (p-2)\sqrt{p^2+10p+1}}{4p}, 1\right\}.
\end{align*}

\subsection{Proofs of Lemmas from \ref{sec:proof_ASIP}}
\label{sec_asip: FDM lemmas}

\begin{proof}[Proof of Lemma \ref{lem:inequality_1}]
    The proof is analogous to to the proof of Lemma A1 of \cite{LiuLin2009}. Although their result is formulated for stationary random variables, the arguments extend directly to the non-stationary case.
\end{proof}
\vspace{6pt}

\begin{proof}[Proof of Lemma \ref{lem:rosenthal}]
By applying the Rosenthal-type bound in \cite{LiuXiaoWu2013}, we obtain
    \begin{align*}
        \max_{1\leq j\leq q_k} \norm{\max_{3jm_k < l\leq 3(j+1)m_k} |\tilde{\mathcal{S}}_{3jm_k,l}^{(k)}|}_{\alpha} \lesssim & \  m_k^{1/2}\left(\sum_{h=1}^{3m_k} \delta_{2}(h)  + \sum_{h = 3m_k+1}^{\infty} \tilde{\delta}_{k,\alpha}(h) + \sup_{i \in \mathbb{Z}} \norm{\mathcal{T}_{3^{k/p}}[X_i]}\right) \\ & + m_k^{1/\alpha} \left( \sum_{h=1}^{3m_k} h^{1/2 - 1/\alpha} \tilde{\delta}_{k,\alpha}(h) +  \sup_{i\in\mathbb{Z}}\norm{\mathcal{T}_{3^{k/p}}[X_i]}_{\alpha}\right),
    \end{align*}
    where $\tilde{\delta}_{k,\alpha}(h)$ is the lag-$h$ functional dependence measure defined for $\tilde{X}_{k,i}$.
    Noting that $m_{k}^{1/2} < m^{1/\alpha}_k h^{1/2 - 1/\alpha}$ for $h > m_k$, the above quantity can be further bounded by
    \begin{align*}
        \underbrace{m_k^{1/2}\left( \sum_{h=1}^{\infty} \delta_2(h) + \sup_{i\in\mathbb{Z}}\norm{X_i}\right)}_{=: \mathcal{I}_{1,k}} + \underbrace{m_{k}^{1/\alpha} \sum_{h=1}^{\infty} h ^{1/2 - 1/\alpha} \tilde{\delta}_{k,\alpha}(h)}_{=:\mathcal{I}_{2,k}} + \underbrace{m_{k}^{1/\alpha} \sup_{i\in\mathbb{Z}} \norm{\mathcal{T}_{3^{k/p}}[X_i]}_{\alpha}}_{=:\mathcal{I}_{3,k}}.
    \end{align*}
    It is therefore enough to control $\sum_{k\geq K_0} 3^{k - k\alpha/p} m_k^{-1}\left(\mathcal{I}_{1,k}^{\alpha} + \mathcal{I}_{2,k}^{\alpha} + \mathcal{I}_{3,k}^{\alpha}\right)$.
    Since $\sum_{h=1}^{\infty} \delta_2(h) + \sup_{i\in\mathbb{Z}} \norm{X_i} \leq 2\Theta_{0,p}$, we have
    \begin{align*}
        \sum_{k\geq K_0} 3^{k - k\alpha/p} m_k^{-1}\mathcal{I}_{1,k}^{\alpha} \lesssim \sum_{k\geq K_0} 3^{k - k\alpha/p} m_k^{\alpha/2-1} \Theta_{0,p}^{\alpha} = \Lambda_p(\alpha) \Theta_{0,2}^{\alpha}.
    \end{align*}
    For the third term, noting that $\min\{ \abs{v}^{\alpha}, 1 \} \leq \min\{ \abs{v}^{\alpha}, v^2\}$, we have
    \begin{align}
         &\sum_{k\geq K_0}3^{k - k\alpha/p }\sup_{i\in\mathbb{Z}}\norm{ \mathcal{T}_{3^{k/p}}[X_i]}_{\alpha}^{\alpha} \leq \sum_{k\geq K_0} \sup_{i\in\mathbb{Z}}3^k \E\left[ \min \left\{ \abs{\frac{X_i}{3^{k/p}}}^{\alpha} ,1\right\} \right] \nonumber \\ \lesssim & \sum_{k\geq 1}3^k \E\left[ \min \left\{ \abs{\frac{Q}{3^{k/p}}}^{\alpha} , \left(\frac{Q}{3^{k/p}} \right)^2\right\} \right] \lesssim \norm{Q}_p^p, \label{eq:interpo_ineq}
    \end{align}
    where we omit some elementary calculations in the last inequality. For the second term, let $\beta = \alpha/(\alpha - 1)$ and $\lambda_h = (h^{1/2-1/\alpha} \delta_p(h)^{p/\alpha})^{-1/\beta}$. By Hölder's inequality, we have
    \begin{align*}
        \left(\sum_{h=1}^{\infty} h^{1/2 - 1/\alpha} \tilde{\delta}_{k,\alpha}(h)\right)^{\alpha} \leq \Xi_p(\alpha)^{\alpha/\beta} \sum_{h=1}^{\infty} \lambda_h^{\alpha}(h^{1/2-1/\alpha} \tilde{\delta}_{k,\alpha}(h))^{\alpha}.
    \end{align*}
    Thus we have \begin{align*}
        \sum_{k\geq K_0} 3^{k - k\alpha/p}m_k^{-1} \mathcal{I}_{3,k}^{\alpha} \leq \Xi_p(\alpha)^{\alpha/\beta} \sum_{h=1}^\infty \lambda_h^{\alpha} h^{\alpha/2-1} \left(\sum_{k=1}^{\infty} 3^{k - k\alpha/p} \tilde{\delta}_{k,\alpha}(h)^{\alpha} \right).
    \end{align*}
    Thus, it suffices for us to show $\sum_{k=1}^{\infty} 3^{k - k\alpha/p} \tilde{\delta}_{k,\alpha}(h)^{\alpha} \lesssim \delta_p(h)^p$, since $\Xi_p(\alpha)^{\alpha/\beta} \sum_{h=1}^{\infty} \lambda_h^{\alpha} h^{\alpha/2-1} \delta_p(h)^p = \Xi_p(\alpha)^{\alpha}$. By the Lipschitz continuity of the truncation operator, we have
    \begin{align*}
        \sum_{k=1}^{\infty} 3^{k-k\alpha/p}\tilde{\delta}_{k,\alpha}(h)^{\alpha}
        &\lesssim \sum_{k=1}^{\infty} 3^{k-k\alpha/p}\sup_{i\in\mathbb{Z}}
        \left\|\mathcal{T}_{3^{k/p}}[X_i]-\mathcal{T}_{3^{k/p}}[X_{i,\{i-h\}}]\right\|_\alpha^\alpha \\
        &\lesssim \sum_{k=1}^{\infty} 3^{k-k\alpha/p}\sup_{i\in\mathbb{Z}}
        \E\left[\min\left(2\cdot 3^{k/p},|X_i-X_{i,\{i-h\}}|\right)^\alpha\right] \\
        &\lesssim \sum_{k=1}^{\infty} 3^k \sup_{i\in\mathbb{Z}}
        \E\left[\min\left\{\left|\frac{X_i-X_{i,\{i-h\}}}{3^{k/p}}\right|^\alpha,1\right\}\right] \lesssim \delta_p(h)^p,
    \end{align*}
    where we used the inequality $|\mathcal{T}_a(x) - \mathcal{T}_a(y)| \leq \min\{ 2a, |x-y|\}$ and the same argument as in \eqref{eq:interpo_ineq}.
\end{proof}
\vspace{6pt}

\begin{proof}[Proof of Lemma \ref{lem:variance_A}]
Recall that $\mathcal{A}_{k,j} = \sqrt{\mathbb{V}^{\circ}_{k,j+1}(\bm{\eta}_{k,3j+3})}\,Z_{k,j} + \mathbb{M}^{(j)}_{k,2}(\bm{\eta}_{k,3j+3}) + \mathbb{M}^{(j+1)}_{k,0}(\bm{\eta}_{k,3j+3})$, where $(Z_{k,j})$ is independent of $(\eps_i)_{i\in\mathbb{Z}}$ and has mean $0$. Hence the Gaussian part is orthogonal to the drift part, and $\|\mathcal{A}_{k,j}\|^2
=
\E\!\left[\mathbb{V}^{\circ}_{k,j+1}(\bm{\eta}_{k,3j+3})\right]
+
\|\mathbb{M}^{(j)}_{k,2}(\bm{\eta}_{k,3j+3})+\mathbb{M}^{(j+1)}_{k,0}(\bm{\eta}_{k,3j+3})\|^2$, where
\begin{align*}
    \E\left[\mathbb{V}^{\circ}_{k,j+1}(\bm{\eta}_{k,3j+3})\right]
    &=
    \Bigl\|\mathbf{B}_{k,3j+1}-\mathbf{C}_{k,3j+1}(\bm{\eta}_{k,3j+1})
    +\mathbf{B}_{k,3j+2}-\mathbb{M}^{(j)}_{k,2}(\bm{\eta}_{k,3j+3})\Bigr\|^2 \\
    &\quad
    +
    \Bigl\|\mathbf{B}_{k,3j+3}-\mathbb{M}^{(j+1)}_{k,0}(\bm{\eta}_{k,3j+3})
    +\mathbf{C}_{k,3j+4}(\bm{\eta}_{k,3j+4})\Bigr\|^2.
\end{align*}

\noindent We now expand the two squared terms. First, since
$\mathbf{C}_{k,3j+1}(\bm{\eta}_{k,3j+1})=\E[\mathbf{B}_{k,3j+1}| \bm{\eta}_{k,3j+1}]$,
we have $\E[(\mathbf{B}_{k,3j+1}-\mathbf{C}_{k,3j+1}(\bm{\eta}_{k,3j+1}))\,\mathbf{C}_{k,3j+1}(\bm{\eta}_{k,3j+1})]=0$, and thus
\begin{align*}
\|\mathbf{B}_{k,3j+1}-\mathbf{C}_{k,3j+1}(\bm{\eta}_{k,3j+1})\|^2
=
\|\mathbf{B}_{k,3j+1}\|^2-\|\mathbf{C}_{k,3j+1}(\bm{\eta}_{k,3j+1})\|^2.
\end{align*}
Similarly, because $\mathbb{M}^{(j)}_{k,2}(\bm{\eta}_{k,3j+3})=\E[\mathbf{B}_{k,3j+2}| \bm{\eta}_{k,3j+3}]$ and
$\mathbb{M}^{(j+1)}_{k,0}(\bm{\eta}_{k,3j+3})=\E[\mathbf{B}_{k,3j+3}|\bm{\eta}_{k,3j+3}]$,
we have
\begin{align*}
\|\mathbf{B}_{k,3j+2}-\mathbb{M}^{(j)}_{k,2}(\bm{\eta}_{k,3j+3})\|^2+\|\mathbb{M}^{(j)}_{k,2}(\bm{\eta}_{k,3j+3})\|^2
=
&\|\mathbf{B}_{k,3j+2}\|^2,\\
\|\mathbf{B}_{k,3j+3}-\mathbb{M}^{(j+1)}_{k,0}(\bm{\eta}_{k,3j+3})\|^2+\|\mathbb{M}^{(j+1)}_{k,0}(\bm{\eta}_{k,3j+3})\|^2
=&
\|\mathbf{B}_{k,3j+3}\|^2.
\end{align*}

\noindent For the cross terms, note that $\mathbf{C}_{k,3j+1}(\bm{\eta}_{k,3j+1})$ is independent of $\mathbf{B}_{k,3j+2}$, and $\mathbb{M}^{(j)}_{k,2}(\bm{\eta}_{k,3j+3})$ is independent of $\mathbf{B}_{k,3j+1}$. Therefore,
\begin{align*}
\E\!\left[(\mathbf{B}_{k,3j+1}-\mathbf{C}_{k,3j+1}(\bm{\eta}_{k,3j+1}))(\mathbf{B}_{k,3j+2}-\mathbb{M}^{(j)}_{k,2}(\bm{\eta}_{k,3j+3}))\right]
=
\E[\mathbf{B}_{k,3j+1}\mathbf{B}_{k,3j+2}].
\end{align*}
Moreover, conditional on $\bm{\eta}_{k,3j+3}$, the variables $\mathbf{B}_{k,3j+2}$ and $\mathbf{B}_{k,3j+3}$ depend on disjoint internal innovations and are conditionally independent, so
\begin{align*}
\E[\mathbf{B}_{k,3j+2}\mathbf{B}_{k,3j+3}]
=
\E\!\left[\mathbb{M}^{(j)}_{k,2}(\bm{\eta}_{k,3j+3})\,\mathbb{M}^{(j+1)}_{k,0}(\bm{\eta}_{k,3j+3})\right].
\end{align*}
Finally, since $\mathbf{C}_{k,3j+4}(\bm{\eta}_{k,3j+4})=\E[\mathbf{B}_{k,3j+4}| \bm{\eta}_{k,3j+4}]$ and
$\E[\mathbf{B}_{k,3j+3}(\mathbf{B}_{k,3j+4}-\mathbf{C}_{k,3j+4}(\bm{\eta}_{k,3j+4}))]=0$, we have $\E[\mathbf{B}_{k,3j+3}\mathbf{C}_{k,3j+4}(\bm{\eta}_{k,3j+4})]
=
\E[\mathbf{B}_{k,3j+3}\mathbf{B}_{k,3j+4}]$. We can conclude the proof by putting these identities into the expansion of $\|\mathcal{A}_{k,j}\|^2$.
\end{proof}
\vspace{6pt}

\begin{proof}[Proof of Lemma \ref{lem:maximal_ineq}]
    Let $\mathcal{P}_k[X]=\E[X|\mathcal{F}_k]-\E[X|\mathcal{F}_{k-1}]$,
    $\mathbf{R}_k=\sum_{i=1}^k X_{a+i}$, and for $h\geq 0$ define
    $\mathbf{R}_{k,h}=\sum_{i=1}^k
    \mathcal{P}_{a+i-h}[X_{a+i}]$. Then
    $\mathbf{R}_k=\sum_{h=0}^{\infty}\mathbf{R}_{k,h}$ in $L_r$.
    Fix $h\geq 0$. The sequence
    $(\mathcal{P}_{a+i-h}[X_{a+i}])_{i\geq 1}$ forms martingale
    differences, and hence by Burkholder's inequality,
    \begin{align}\label{Burkholder ineq}
        \norm{\mathbf{R}_{m,h}}_r^2 = \norm{\sum_{i=1}^m
        \mathcal{P}_{a+i-h}[X_{a+i}]}_r^2  \leq
        (r-1)\sum_{i=1}^m
        \norm{\mathcal{P}_{a+i-h}[X_{a+i}]}_r^2 \leq
        (r-1)m\,\delta_r(h)^2.
    \end{align}
    Therefore,
    $\norm{\mathbf{R}_{m,h}}_r
    \leq\sqrt{r-1}\sqrt{m}\,\delta_r(h)$.
    By Doob's maximal inequality,
    $\norm{\max_{1\leq k\leq m}|\mathbf{R}_{k,h}|}_r
    \leq r\norm{\mathbf{R}_{m,h}}_r/(r-1)$.
    Combining these bounds and summing over $h$ yields
    \begin{align*}
        \norm{\max_{1\leq k\leq m}|\mathbf{R}_k|}_r \leq
        \sum_{h=0}^{\infty} \norm{\max_{1\leq k\leq m}|\mathbf{R}_{k,h}|}_r \leq
        \frac{r}{\sqrt{r-1}}\sqrt{m} \sum_{h=0}^{\infty}\delta_r(h) =
        \frac{r}{\sqrt{r-1}}\sqrt{m}\,\Theta_{0,r}.
    \end{align*}
    Since this bound is uniform in $a$, taking the supremum over $a$
    gives the desired result. The same argument applies to
    $X_i^{\oplus}-\E[X_i^{\oplus}]$ and $\tilde{X}_{k,i}$, since
    centering leaves the functional dependence measure unchanged,
    truncation is $1$-Lipschitz, and conditional expectation is an
    $\mathcal{L}_r$ contraction.
\end{proof}
\vspace{6pt}

\begin{proof}[Proof of Lemma \ref{lem:increment_property}]
    Let $\Delta_n:=\max_{1\le i\le n}|\mathsf{V}_i-\overline{\mathsf{V}}_i|$. For each fixed $i$, the increment
    $\Bro(\mathsf{V}_i)-\Bro(\overline{\mathsf{V}}_i)$ is centered Gaussian with variance $|\mathsf{V}_i-\overline{\mathsf{V}}_i|$; hence for any $x>0$, by taking the union bound,
    \begin{align*}
        \P\left[\max_{1\le i\le n}|\Bro(\mathsf{V}_i)-\Bro(\overline{\mathsf{V}}_i)|>x\right]
        \le 2n\exp\!\left(-\frac{x^2}{2\Delta_n}\right).
    \end{align*}
    Choose $x=c\sqrt{\Delta_n\log n}$ with any constant $c>2$. Then $\P\left[\max_{1\le i\le n}|\Bro(\mathsf{V}_i) -\Bro(\overline{\mathsf{V}}_i)|>c\sqrt{\Delta_n\log n}\right] \le 2n^{\,1-c^2/2},$
    and the right-hand side is summable in $n$. By the Borel--Cantelli lemma,
    \begin{align*}
        \max_{1\le i\le n}|\Bro(\mathsf{V}_i)-       \Bro(\overline{\mathsf{V}}_i)|
        =O_{\mathrm{a.s.}}\!\big(\sqrt{\Delta_n\log n}\big).
    \end{align*}
    Since $\Delta_n=o(n^\beta)$, it follows that $\max_{1\le i\le n}|\Bro(\mathsf{V}_i)-\Bro(\overline{\mathsf{V}}_i)|
        =o_{\mathrm{a.s.}}\!\big(n^{\beta/2}\sqrt{\log n}\big).$
\end{proof}
\vspace{6pt}

\begin{proof}[Proof of Lemma \ref{lem:covariance}]
    Recall the projection operator $\mathcal{P}_k$ from the proof of Lemma~\ref{lem:maximal_ineq}. For $i\le s$, let $X_{s,\{i\}}$ be obtained by replacing $\eps_i$ with an independent copy. Since
    $\E[X_{s,\{i\}}|\mathcal{F}_i]=\E[X_s|\mathcal{F}_{i-1}]$, we have
    $\mathcal{P}_i[X_s]=\E[X_s-X_{s,\{i\}}|\mathcal{F}_i]$, and hence
    $\norm{\mathcal{P}_i[X_s]}\le \delta_p(s-i)$.
    Therefore, by orthogonality of martingale projections, for any $s,l\ge1$,
    \begin{align*}
        \abs{\E[X_sX_{s+l}]}
        =\left|\sum_{i\in\mathbb Z}\sum_{j\in\mathbb Z}
        \E[\mathcal{P}_i[X_s]\mathcal{P}_j[X_{s+l}]]\right| \leq \sum_{i=-\infty}^{s}
        \norm{\mathcal{P}_i[X_s]}
        \norm{\mathcal{P}_i[X_{s+l}]}
        \leq \sum_{h=0}^{\infty}\delta_p(h)\delta_p(h+l).
    \end{align*}
    The same argument applies to $X_i^{\oplus}-\E[X_i^{\oplus}]$, since the truncation operator is $1$-Lipschitz and hence its functional dependence measure is bounded by $\delta_p(h)$.
\end{proof}
\vspace{6pt}

\section{Proofs of Results in Section \ref{Sec: Asymp CS} and Section \ref{Sec: Robbins Gaussian Mixture Asymp CSs for (Nonstationary) Time Series}} \label{Appendix: Proof of TUCLT}
\noindent Section \ref{Appendix: Definition on Asymp CS} further discusses AsympCSs and proves the claims of two illustrative examples in Section \ref{Sec: Asymp CS}. Section~\ref{Appendix: Proof of coverage by eventual inclusion} shows that a candidate sequence of intervals attains asymptotic time-uniform coverage if it eventually contains a valid nonasymptotic CS. Section~\ref{Appendix: Proofs of Gaussian Mixture Asymp CS} derives the Gaussian-mixture AsympCS for the running average of the unconditional means. Finally, Section~\ref{Appendix: Proof of Corollary  Equivalent Asymp CS Boundary} establishes the asymptotic equivalence between the Gaussian-mixture boundary in Theorem \ref{Thm: Robbins Gaussian Mixture Asymp CS} and its simplified form.

\subsection{Discussion of Asymptotic Confidence Sequence - Definition \ref{Def: Asym CS}} 
\label{Appendix: Definition on Asymp CS}
\noindent In this section, we first discuss the proof strategies used in existing work to establish asymptotic time-uniform coverage. We then provide proofs of the two illustrative examples.

\subsubsection{General Statements} \label{Appendix: General Statements for Defn of Asymp CS}
\noindent Following \cite{bibaut2022near} and \cite{waudby2024time}, some subsequent work such as \cite{dalal2024anytime} also noted the subtle difference between an AsympCS and a sequence merely satisfying asymptotic time-uniform coverage. However, they all proceed with Theorem 2.8 of \cite{waudby2024time} in the context of \cite{robbins1970boundary} by directly obtaining a sharp asymptotic time-uniform coverage. They are thus less integrated with the general results provided by first enabling an underlying CS via martingale inequality. In contrast, our construction leverages the eventual-inclusion condition. Consequently, their proofs rely on the Condition L-3-$\eta$ in \cite{waudby2024time}, which requires a polynomial rate of convergence for variance estimation, stronger than the  logarithmic rate imposed in our Theorem \ref{Thm: Robbins Gaussian Mixture Asymp CS}. Their argument is that, in the i.i.d.\ setting, the polynomial rate condition follows from the Marcinkiewicz-Zygmund SLLN under mild moment assumptions. For time series, however, this condition is no longer automatic and can impose substantive restrictions on the choice of bandwidth and kernel. Moreover, although Marcinkiewicz-Zygmund SLLNs (\cite{marcinkiewicz1937fonctions}) are available for certain classes of dependent processes, such as stationary strongly mixing sequences (see e.g., \cite{rio1995maximal} or Corollary 3.2 (i) of \cite{rio2017asymptotic}), to the best of our knowledge, these results do not directly extend to nonstationary time series. We thus think our proof construction has its own merit in time-series settings.

\subsubsection{Proofs of Illustrative Examples in Section \ref{Sec: Asymp CS}} \label{Appendix: Proofs of Illustrative Examples} 
\begin{proof}[Proof of Illustrative Example \ref{IEg 2}]
We proceed with a two-step proof. In the first step, we show that Condition 1 of Definition \ref{Def: Asym CS} holds. In the second step, we show that the time time-uniform coverage approaches to zero in the limit. We next state a lemma that will be used in Step 1 of the main proof.  
\begin{Lem}[EFKP-LIL] \label{Lem: Counterex 2}
Under the same conditions as Illustrative Example \ref{IEg 2}, let $S_t :=\sum_{i=1}^t (X_i-\theta_0) =t(\bar X_t-\theta_0).$ It holds that 
\[
\mathbb{P}\!\left[
|S_t|> u_t(a) \ \text{i.o.}
\right]
=
\begin{cases}
1, & a\le 3,\\[1mm]
0, & a>3.
\end{cases}
\]
\end{Lem}

\begin{proof}
Observe that $\{S_t\}_{t\ge1}$ is a Gaussian random walk, and we could then apply the Kolmogorov--Erd\H{o}s--Petrowski--Feller upper--lower class integral test (EFKP-LIL, Chapter 5.2 of \cite{revesz2013random}), which states the following:
\[
\P[|S_t| \ge  \sqrt{t}\,\phi(t) \ i.o. ] = 0 \ \text{or} \ 1 \quad \text{according as} \ \int^\infty \frac{\phi(x)}{x}
\exp\!\left(-\frac{\phi(x)^2}{2}\right)\,dx <\infty \ \text{or} \ =\infty,
\]
where $\phi(\cdot)$ is a positive non-decreasing continuous test function.

\vspace{5pt}

\noindent Now for $a > 0$, let $\phi_a(t):=\sqrt{2\log\log t+a\log\log\log t}$, which is a valid test function on $[16,\infty)$. Note that $u_t(a) = \sqrt{t} \phi_a(t)$. Then,
\[
\exp\!\left(-\frac{\phi_a(t)^2}{2}\right)
=
\exp\!\left(
-\log\log t-\frac{a}{2}\log\log\log t
\right)
=
\frac{1}{\log t\,(\log\log t)^{a/2}}.
\]
Hence the integrand satisfies 
\[
\frac{\phi_a(t)}{t}
\exp\!\left(-\frac{\phi_a(t)^2}{2}\right)
\asymp
\frac{\sqrt{\log\log t}}
{t\log t\,(\log\log t)^{a/2}}
=
\frac{1}
{t\log t\,(\log\log t)^{(a-1)/2}}.
\]
Making the change of variables $y=\log\log t,\
dy=\frac{dt}{t\log t},$ the integral is asymptotically equivalent to $\int^\infty y^{(1-a)/2}\,dy,$ which converges if and only if $\frac{1-a}{2}<-1,
\ \text{or equivalently}\
a>3.$
\end{proof}

\noindent Now we proceed with the main proof. 
\paragraph{Step 1.} We first verify that $\mathbb P\!\left[ \forall t \ge 1:\, \theta_0 \in C_t^\star \right]\ge 1-\alpha$. Let $S_t :=\sum_{i=1}^t (X_i-\theta_0) =t(\bar X_t-\theta_0).$ It follows that 
\[
\theta_0\in
\left[
\bar X_t-\frac{g_t^{\star}}{t},
\,
\bar X_t+\frac{g_t^{\star}}{t}
\right]
\iff
|S_t|\le g_t^{\star},
\]
and it thus suffices to show that $\mathbb P\!\left[
\forall t\ge1:\, |S_t|\le g_t^{\star} \right] \ge 1-\alpha.$

\vspace{3pt}

\noindent For $R>0$, define the first part of event as $A_R := \left\{\max_{1\le t<T_0}|S_t| \le R \right\}.$  Since $T_0<\infty$, $\max_{1\le t<T_0}|S_t|<\infty$ almost surely. Note that $A_R$ is a sequence of non-decreasing events with respect to $R$, and by the continuity from below property of probability measure, it then follows that $\P[A_R]\uparrow1,$ as $R\to\infty$. By the definition of a limit, we have $\forall\,\eta>0,\ \exists\,R<\infty$ such that $\mathbb{P}[A_R] > 1-\eta.$ In particular, taking \(\eta=\alpha/2\), we can choose a large enough $R_{\alpha/2} > 0$, potentially depending on $\alpha$, such that $\mathbb{P}[A_R^c] \le \frac{\alpha}{2}.$

\vspace{5pt}

\noindent  Similarly, for $K>0$, define the second part of event as 
\[
B_K :=
\Bigg\{
\forall t\ge T_0:
|S_t|
\le
u_t(4)+K
\Bigg\}
=
\left\{ \sup_{t\ge T_0}
\left(|S_t|-u_t(4)\right) \le K\right\}.
\]
W.L.O.G., take $t \in \mathbb{N}$. By Lemma \ref{Lem: Counterex 2}, it then follows that $\sup_{t\ge T_0}
\left(|S_t|-u_t(4)\right) < \infty$ almost surely. 
Hence by similar arguments, $\P[B_K] \uparrow1$ as $K\to\infty$, and we can choose a large enough $K_{\alpha/2} > 0$ such that $\mathbb P(B_K^c) \le  \frac{\alpha}{2}.$ Therefore, by the union bound, we conclude with 
\begin{align*}
\P \left[\forall t\ge1: |S_t| \le
g_t^{\star} \right] = \P[A_R\cap B_K] \ge
1- \mathbb P[A_R^c] - \mathbb P[B_K^c] \ge 1-\alpha. 
\end{align*}
\noindent We then verify that $\widetilde{g}_t / g_t^\star \to 1$ almost surely. This easily follows because convergence is a tail property, and on $[16,\infty)$, $u_t(4)/u_t(2) \to 1$.

\paragraph{Step 2.} Observe that $u_t(2)$ is the lower-class boundary of LIL rate, i.e., $\mathbb{P}\!\left[
|S_t|> u_t(2) \ \text{i.o.} \right] = 1$. Then for any  $m \ge T_0$, 
it holds that 
\[
\mathbb{P}\!\left[\forall\, t \ge m:\ \theta_0 \in \widetilde C_t\right]
=
\mathbb{P}\!\left[\forall\, t \ge m:\ |S_t| \le u_t(2)\right]
=0.
\]
Therefore, no matter how large the deterministic burn-in time point $m$ is chosen, the process will almost surely cross the boundary $u_t(2)$ again after time $m$. Consequently, we obtain 
\[
\lim_{m\to\infty}
\mathbb{P}\!\left[\forall\, t \ge m:\ \theta_0 \in \widetilde C_t\right]
=0.
\]
\end{proof}

\begin{proof}[Proof of Illustrative Example \ref{IEg 1}]
By Hoeffding's inequality (Section 2.2 of \cite{vershynin2018high}),
\begin{align*}
\P[\mu \notin C_t^\star] =
\P\left[
\left|\bar X_t-\mu\right|
>
\sqrt{\frac{\log\!\left(\pi^2 t^2/(3\alpha')\right)}{2t}}
\right] \le 2\exp\left( -2t\cdot
\frac{\log\!\left(\pi^2 t^2/(3\alpha')\right)}{2t}
\right) =
\frac{6\alpha'}{\pi^2 t^2}.
\end{align*}
Therefore by a simple union bound, it follows that 
$\P\!\left[\exists t\ge1:\mu\notin C_t^\star\right]
\le
\sum_{t=1}^\infty
{6\alpha'}/{\pi^2t^2} = \alpha'.$
Moreover, for every $m\ge1$, by the tail union bound,  $\P\!\left[\exists t\ge m:\mu\notin C_t^\star\right]
\le
\sum_{t=m}^\infty
{6\alpha'}/{\pi^2t^2} < \infty, $
so that
\[
\lim_{m\to\infty}
\P\!\left[\exists t\ge m:\mu\notin C_t^\star\right]
=0.
\]    
Taking the complement concludes the proof. 
\end{proof}

\subsection{Proof of Asymptotic Time-Uniform Coverage via Eventual Inclusion} \label{Appendix: Proof of coverage by eventual inclusion}
\begin{proof}
For each deterministic $m\in\mathbb{N}$, define
$A_m:=\bigcap_{t=m}^{\infty}\{\theta_0\in C_t^{\star}\}$,
$B_m:=\bigcap_{t=m}^{\infty}\{C_t^{\star}\subseteq\widetilde{C}_t\}$, and
$D_m:=\bigcap_{t=m}^{\infty}\{\theta_0\in\widetilde{C}_t\}$.
Clearly, $A_m\cap B_m\subseteq D_m$. Moreover, since $A_1\subseteq A_m$, the validity of $(C_t^{\star})$ under Condition 1 in Definition~\ref{Def: Asym CS} gives
$\mathbb{P}[A_m]\geq\mathbb{P}[A_1]\geq1-\alpha$. Therefore, by the union bound,
\begin{align*}
    \mathbb{P}[D_m]
    \geq \mathbb{P}[A_m\cap B_m]
    \geq 1-\alpha-\mathbb{P}[B_m^{\complement}].
\end{align*}
The events $(B_m)$ are increasing, and
\begin{align*}
    \bigcup_{m=1}^{\infty}B_m
    =
    \left\{
        \exists\,M<\infty:
        \ C_t^{\star}\subseteq\widetilde{C}_t
        \text{ for all }t\geq M
    \right\},
\end{align*}
where $M$ is an almost surely finite, sample-dependent index. Hence, the eventual-inclusion assumption and continuity from below yield
$\mathbb{P}[B_m]\to
\mathbb{P}[\bigcup_{m=1}^{\infty}B_m]=1$.
Thus, $\mathbb{P}[B_m^{\complement}]\to0$, and consequently,
\begin{align*}
    \liminf_{m\to\infty}\mathbb{P}[D_m]\geq1-\alpha.
\end{align*}
Finally, since $D_m\subseteq D_{m+1}$, the sequence $(D_m)$ is
non-decreasing and the limit exists. Therefore,
\begin{align*}
    \lim_{m\to\infty}
    \mathbb{P}
    \left[
        \forall t\geq m:
        \theta_0\in\widetilde{C}_t
    \right]
    \geq1-\alpha.
\end{align*}
\end{proof}

\subsection{Proof of Theorem \ref{Thm: Robbins Gaussian Mixture Asymp CS}}  \label{Appendix: Proofs of Gaussian Mixture Asymp CS}

\noindent In this section, we first establish auxiliary lemmas and then use them to prove Theorem \ref{Thm: Robbins Gaussian Mixture Asymp CS}.

\vspace{12pt}
\noindent \textbf{Auxiliary Lemmas.} As stated in the main text, a key technical caveat is that the process $\V_t$ in Theorem \ref{Thm: nonstationary ASIP} is not necessarily non-decreasing. Consequently, the canonical Brownian filtration, defined as $\mathcal{G}_t := \sigma(\Bro(s): 0 \le s \le \V_t)$, is not valid in a manner consistent with the martingale structure. Therefore, to recover the validity of standard filtration-based arguments, we upper bound the process $\V_t$ by a deterministic non-decreasing sequence. The minimal such sequence is given by its running maximum.

\begin{Lem}[Construction of Nonnegative Martingale] \label{Lemma: Nonnegative martingale construction}
Let $\{\U_t\}_{t\geq0}$ be a deterministic and non-decreasing sequence with $\U_0 = 0$ and $\Bro(\cdot)$ denote the standard Brownian motion. Then 
for any $\lambda \in \R$, 
\begin{align}\label{eq: expoential_martingale}
    \mathcal{E}_t(\lambda)
:= \exp\left( \lambda \Bro(\U_t) - \frac{\lambda^2}{2} \U_t \right)
\end{align}
is a nonnegative martingale starting at $\mathcal{E}_0= 1$ with respect to the Brownian filtration $\mathcal{G}_t := \sigma(\Bro(s): 0 \le s \le \U_t)$.  
\end{Lem}

\begin{proof}
This proof is standard, so we omit it here. 
\end{proof}

\begin{Prop}[Order of Running Max Difference] \label{Prop: Order of Running Max Difference}
Under the same notation and assumptions as in Theorem \ref{Thm: nonstationary ASIP}, let $\U_t:=\max_{0\le s\le t}\V_s$ with $\U_0=0$ be the running max of the cumulative variance process.
The difference between running max $\U_t$ and $\V_t$ is upper bounded by
\[
    \U_t-\V_t\lesssim 1.
\]
\end{Prop}

\begin{proof}
Note that
$\V_t=\Var\left[\sum_{i=1}^t(Y_i-\mu_i)\right]$. Let $\Sf_t=\sum_{i=1}^t(Y_i-\mu_i)$ so that $\E[\Sf_t]=0$ and $\V_t=\Var[\Sf_t]=\E[\Sf_t^2]$.
For any $s\le t$, using $\Sf_t=\Sf_s+(\Sf_t-\Sf_s)$, we have the identity
$\E[\Sf_s^2]=\E[\Sf_t^2]-\E[(\Sf_t-\Sf_s)^2]-2\E[\Sf_s(\Sf_t-\Sf_s)]$, which implies
$\E[\Sf_s^2]\le \E[\Sf_t^2]+2|\E[\Sf_s(\Sf_t-\Sf_s)]|$. Taking the maximum on both sides, we have
\begin{align*}
    \max_{0\le s\le t}\E[\Sf_s^2]
    \le \E[\Sf_t^2]
    +2\max_{0\le s\le t}\abs{\E[\Sf_s(\Sf_t-\Sf_s)]}.
\end{align*}
It therefore remains to control the cross term
$\max_{0\le s\le t}\abs{\E[\Sf_s(\Sf_t-\Sf_s)]}$.
Notice that for any $s\le t$, by Lemma \ref{lem:covariance}, we have
\begin{align*}
    \abs{\E[\Sf_s(\Sf_t-\Sf_s)]}
    &\le \sum_{i_1=1}^s\sum_{i_2=s+1}^t
    \sum_{j=0}^{\infty}\delta_p(j)\delta_p(j+i_2-i_1)\\
    &\le \sum_{h=1}^{t-1}\min\{h,t-h\}
    \sum_{j=0}^{\infty}\delta_p(j)\delta_p(j+h)\\
    &\le \sum_{h=1}^{\infty}h
    \sum_{j=0}^{\infty}\delta_p(j)\delta_p(j+h).
\end{align*}
Since all terms are nonnegative,
\begin{align*}
    \sum_{h=1}^{\infty}h
    \sum_{j=0}^{\infty}\delta_p(j)\delta_p(j+h)
    &=\sum_{r=1}^{\infty}\sum_{j=0}^{\infty}
    \delta_p(j)\Theta_{j+r,p}\\
    &\le \Theta_{0,p}\sum_{r=1}^{\infty}\Theta_{r,p}
    \lesssim \sum_{r=1}^{\infty}r^{-A}<\infty,
\end{align*}
where the last inequality follows from Assumption \ref{asm:dep_decay} and $A>1$.
Combining these results together, we conclude the proof.
\end{proof}

The running maximum $\U_t$ is non-decreasing and therefore provides a valid time change for the Brownian filtration and the subsequent martingale arguments. Proposition~\ref{Prop: Order of Running Max Difference} shows that replacing $\V_t$ by $\U_t$ incurs only a bounded error. The next lemma further shows that the Brownian time changes $\mathbb{B}(\U_t)$ and $\mathbb{B}(\V_t)$ are asymptotically equivalent at the scale
of the confidence-sequence boundary. We may therefore construct the Gaussian-mixture martingale using the monotone process $\U_t$ without changing the relevant asymptotic behavior associated with $\V_t$.

\begin{Lem} \label{Lemma: Brownian motion difference}
Under the same notation and assumptions as in Proposition \ref{Prop: Order of Running Max Difference}, for some $\beta \in (0,1)$, it holds that 
$$\abs{\Bro(\V_t)  - \Bro(\U_t)} =  \oas{t^{\beta/2}\sqrt{\log t}}.$$
\end{Lem}

\begin{proof}
By Proposition \ref{Prop: Order of Running Max Difference}, we know $\U_t - \V_t = O(1)$. For any $\beta \in (0,1),$  it holds that $O(1) = o(t^{\beta})$. Then by Lemma \ref{lem:increment_property}, the desired result immediately follows after scaling over $t$.
\end{proof}

\vspace{12pt}
\noindent \textbf{Main Proof of Theorem \ref{Thm: Robbins Gaussian Mixture Asymp CS}.}   
\noindent As discussed in the main text, our Definition~\ref{Def: Asym CS} differs from the original definition proposed in \cite{waudby2024time}. We therefore proceed with two parts. Part I follows the general recipe (Condition G-1 to Condition G-4) of \cite{waudby2024time} to establish Condition 1, the almost-sure relative approximations, and Part II verifies Condition 2, the eventual inclusion. We proceed Part I with the following steps:
\begin{enumerate}
    \item \textbf{Verify G-1.} Theorem~\ref{Thm: nonstationary ASIP} couples the centered partial-sum process with a time-changed Brownian motion, providing the strong approximation required by Condition G-1.

    \item \textbf{Verify G-2.} We construct a Gaussian-mixture martingale based on $\mathbb{B}(\U_t)$ and apply Ville's inequality to obtain a nonasymptotic boundary for the Brownian proxy.

    \item \textbf{Verify G-3.} We show that the strong approximation error is negligible relative to the nonasymptotic boundary from G-2. Proposition~\ref{Prop: Order of Running Max Difference} allows the boundary rate based on $\V_t$ to be replaced by the corresponding rate based on $\U_t$.

    \item \textbf{Verify G-4.} Using $\widehat{\V}_t/\V_t=1+\oas{1/\log t}$ and $\U_t/\V_t\to1$, we show that the boundaries based on $\widehat{\V}_t$ and $\U_t$ have a ratio converging to one, as required by Condition G-4.
\end{enumerate}

\vspace{6pt}
\begin{proof}[Proof of Part I]  
\noindent Condition G-1 is satisfied directly by Theorem~\ref{Thm: nonstationary ASIP}. To verify Condition G-2, we next derive a nonasymptotic boundary for the process $\Bro(\U_t)$. Recall from Equation~\eqref{eq: expoential_martingale} that, for each $\lambda\in\mathbb{R}$, $\mathcal{E}_t(\lambda)$ is a nonnegative martingale with initial value one. Let $f(\lambda;0,\rho^2)$ denote the Gaussian density with mean zero and variance $\rho^2>0$. Define
\begin{align*}
    \mathcal{E}_t
    &:=
    \int_{\mathbb{R}}
    \mathcal{E}_t(\lambda)f(\lambda;0,\rho^2)\,d\lambda \\
    &=
    \exp\left(
    \frac{\rho^2(\Bro(\U_t))^2}{2(\U_t\rho^2+1)}
    \right)
    (\U_t\rho^2+1)^{-1/2},
\end{align*}
where the closed form of the mixture exponential Wiener process follows from construction in \cite{waudby2024time}. By Fubini's theorem, $\mathcal{E}_t$ is also a nonnegative martingale with initial value one (\cite{robbins1970statistical}). Ville's inequality (\cite{ville1939etude}) therefore gives $\P[\forall t\geq1:\mathcal{E}_t<1/\alpha]\geq1-\alpha$. Substituting the closed-form expression and taking logarithms yield
\begin{align*}
    \P\left[
    \forall t\geq1:
    \frac{\rho^2(\Bro(\U_t))^2}{2(\U_t\rho^2+1)}
    <
    \log\left(\frac{1}{\alpha}\right)
    +
    \frac{1}{2}\log(\U_t\rho^2+1)
    \right]
    \geq1-\alpha.
\end{align*}
Rearranging the inequality gives
\begin{equation} \label{Eq: Direct Boundary from Villes}
    \P\left[
    \forall t\geq1:
    \left|
    \frac{\Bro(\U_t)}{t}
    \right|
    <
    \sqrt{
    \frac{\U_t\rho^2+1}{t^2\rho^2}
    \log\left(
    \frac{\U_t\rho^2+1}{\alpha^2}
    \right)
    }
    \right]
    \geq1-\alpha.
\end{equation}
Thus, Condition G-2 is satisfied.

\vspace{12pt}
\noindent We next compare the strong approximation and nonasymptotic boundary rates to verify Condition G-3. Specifically, we show that the strong approximation error in Theorem~\ref{Thm: nonstationary ASIP} is asymptotically negligible relative to the nonasymptotic boundary constructed above. Since $p>2$ and $\V_t\asymp t$,
\begin{align*}
    t^{1/p}\sqrt{\log t}
    =
    o\left(
    \sqrt{\V_t \log\V_t}
    \right).
\end{align*}
By Proposition \ref{Prop: Order of Running Max Difference} and $V_t \asymp t,$
\begin{equation} \label{Eq: order of Ut/V_t}
   {\U_t}/{\V_t}-1
    = O(1/t)
    =
    o\left({1}/{\log t}\right).
\end{equation}
Thus, $\U_t/\V_t\to1$ and $\log\U_t/\log\V_t\to1$. Combining these limits with the preceding rate comparison gives $t^{1/p}\sqrt{\log t} = o\left(\sqrt{\U_t \log\U_t} \right),$ which verifies Condition G-3.

 
\vspace{12pt}
\noindent We next derive the almost-sure approximate boundary required by Condition G-4. Let $\widetilde{\mu}_t:=t^{-1}\sum_{i=1}^t\mu_i$ denote the estimand, with estimator $\widehat{\mu}_t:=t^{-1}\sum_{i=1}^tY_i$ as in Theorem~\ref{Thm: Robbins Gaussian Mixture Asymp CS}. We verify that
\begin{align*}
    [\widehat{\mu}_t - \widetilde{L}_t, \widehat{\mu}_t + \widetilde{U}_t]
    :=
    \left[
    \widehat{\mu}_t
    \pm
    \left(
    \sqrt{
    \frac{\widehat{\V}_t\rho^2+1}{t^2\rho^2}
    \log\left(
    \frac{\widehat{\V}_t\rho^2+1}{\alpha^2}
    \right)
    }
    \right)
    \right]
\end{align*}
forms a $(1-\alpha)$ AsympCS for $\widetilde{\mu}_t$.

\vspace{3pt}
\noindent The assumed relative consistency
$\widehat{\V}_t/\V_t=1+\oas{1/\log t}$ implies
$\V_t/\widehat{\V}_t\xrightarrow{\mathrm{a.s.}}1$. Together with the preceding result that $\U_t/\V_t\to1$, this gives $\U_t/\widehat{\V}_t\xrightarrow{\mathrm{a.s.}}1$. Since $\V_t\asymp t$ and $\V_t/\widehat{\V}_t\xrightarrow{\mathrm{a.s.}}1$, we have $\widehat{\V}_t\to\infty$ almost surely. Hence, $(\U_t\rho^2+1)/(\widehat{\V}_t\rho^2+1) \xrightarrow{\mathrm{a.s.}}1$.



\vspace{3pt}
\noindent For the logarithmic factors,
\begin{align*}
    \frac{
    \log((\U_t\rho^2+1)/\alpha^2)
    }{
    \log((\widehat{\V}_t\rho^2+1)/\alpha^2)
    }
    -1
    =
    \frac{
    \log((\U_t\rho^2+1)/(\widehat{\V}_t\rho^2+1))
    }{
    \log((\widehat{\V}_t\rho^2+1)/\alpha^2)
    }
    \xrightarrow{\mathrm{a.s.}}
    0.
\end{align*}
Indeed, the numerator converges to zero, whereas the denominator diverges to infinity. Combining the two factors yields
\begin{align*}
    \frac{
    (\U_t\rho^2+1)
    \log((\U_t\rho^2+1)/\alpha^2)
    }{
    (\widehat{\V}_t\rho^2+1)
    \log((\widehat{\V}_t\rho^2+1)/\alpha^2)
    }
    \xrightarrow{\mathrm{a.s.}}
    1.
\end{align*}
Thus, the ratio of the two boundary terms converges almost surely to one. Moreover, the boundary based on $\widehat{\V}_t$ divided by $\sqrt{\U_t\log\U_t}/t$ converges almost surely to one. Therefore,
\begin{align}
    \sqrt{
    \frac{\U_t\rho^2+1}{t^2\rho^2}
    \log\left(
    \frac{\U_t\rho^2+1}{\alpha^2}
    \right)
    }
    =
    \sqrt{
    \frac{\widehat{\V}_t\rho^2+1}{t^2\rho^2}
    \log\left(
    \frac{\widehat{\V}_t\rho^2+1}{\alpha^2}
    \right)
    }
    +
    o_{\mathrm{a.s.}}\left(\frac{\sqrt{\U_t\log\U_t}}{t}\right)
    \label{Eq: plain rho case}
\end{align}

\noindent Finally, let
\begin{align*}
    [\widehat{\mu}_t - L_t, \widehat{\mu}_t + U_t]
    :=
    \left[
    \widehat{\mu}_t
    \pm
    \left(
    \sqrt{
    \frac{\U_t\rho^2+1}{t^2\rho^2}
    \log\left(
    \frac{\U_t\rho^2+1}{\alpha^2}
    \right)
    }
    \right)
    \right].
\end{align*}
Equation \eqref{Eq: plain rho case} implies
$\widetilde{L}_t/L_t\xrightarrow{\mathrm{a.s.}}1$ and
$\widetilde{U}_t/U_t\xrightarrow{\mathrm{a.s.}}1$,
satisfying Condition G-4.

\vspace{12pt}
\noindent \textbf{Supplementary Step.}\footnote{For interested readers, our notation for the (centered) boundaries, i.e., $[L_t^{\star},U_t^{\star}]$, $[L_t,U_t]$, and $[\widetilde{L}_t,\widetilde{U}_t]$, is not exactly same as ones in \cite{waudby2024time}. We order them in this way for the consistency of our proof.} One can further verify directly that Condition 1 in Definition~\ref{Def: Asym CS} is satisfied. We provide this verification for completeness and to make the argument transparent to readers encountering this proof strategy for the first time. By the triangle inequality, 
$$
\abs{\frac{1}{t} \sum_{i=1}^{t} (Y_i - \mu_i)} \le \abs{\frac{1}{t} \sum_{i=1}^{t} (Y_i - \mu_i) - \Bro(\V_t)/t}  + \abs{\Bro(\V_t)/t  - \Bro(\U_t)/t} + \abs{\Bro(\U_t)/t},
$$
and thus by Theorem \ref{Thm: nonstationary ASIP}, Lemma \ref{Lemma: Brownian motion difference}, and Step 1, we have with probability at least $(1-\alpha)$ that
$$
\forall t \ge 1, \quad \abs{\frac{1}{t} \sum_{i=1}^{t} (Y_i - \mu_i)}  \le \abs{\eps_t} +  \sqrt{
\frac{\U_t \rho^2 + 1}{t^2 \rho^2}
\log\!\left(
\frac{\U_t \rho^2 + 1}{\alpha^2}
\right)
},
$$
where, with a slight abuse of notation in indexing the error sequence in a finite $t$ statement, we define $\eps_t := \oas{t^{1/p -1}\sqrt{\log t}} + \oas{t^{\beta/2 -1}\sqrt{\log t}}$ with $p > 2$ and $\beta \in (0,1)$. In particular, letting 
\begin{equation} \label{Eq: exact unknown CS boundary}
C_t^{\star} = [\widehat{\mu}_t - L_t^{\star}, \widehat{\mu}_t + U_t^{\star}] := \left[\widehat{\mu}_t \pm \left(\sqrt{
\frac{\U_t \rho^2 + 1}{t^2 \rho^2}
\log\!\left(
\frac{\U_t \rho^2 + 1}{\alpha^2}
\right)
} + |\eps_t| \right)  \right],    
\end{equation}
we then obtain the exact (but unknown due to randomness in $|\eps_t|$ and unobservability of exact values of $\U_t$) $(1-\alpha)$ non-asymptotic confidence sequence for $\widetilde{\mu}_t$. In fact, one could verify that $|\eps_t| = \oas{\sqrt{\U_t\log \U_t}/t}$, and it follows 
that we do have $\widetilde{L}_t/L_t^{\star} \xrightarrow{a.s.} 1$ and $\widetilde{U}_t/U_t^{\star} \xrightarrow{a.s.} 1$ as required in Definition \ref{Def: Asym CS}. 

\end{proof}

\noindent Part II shows that the eventual-inclusion condition holds, and the desired asymptotic time-uniform coverage follows by Appendix \ref{Appendix: Proof of coverage by eventual inclusion}. The key idea is to introduce a slack parameter $k$ and use different values of $\rho$ in $C_t^\star$ and $\widetilde{C}_t$ while preserving the almost sure consistency of the boundary endpoints established in Part I. Specifically, for a fixed $k>e$, we set $\rho_k:=\rho\log k>\rho$. This choice is only a proof device and does not restrict the pre-specified value of $\rho$ in Theorem~\ref{Thm: Robbins Gaussian Mixture Asymp CS}. With a slight abuse of notation, given any pre-specified target value $\rho$ in the empirical CS $(\widetilde{C}_t)$, it suffices to fix $k \ge e$ and set the tuning parameter in the oracle CS $(C_t^\star)$ as $\rho_0 := \rho/\log k$.  

\vspace{6pt}
\begin{proof}[Proof of Part II]
\noindent We show that there exists an almost surely finite random time $T_k<\infty$ such that $C_t^{\star}\subseteq\widetilde{C}_t$ for all $t\geq T_k$. Define
\begin{align*}
L_t(\rho):=\sqrt{\frac{\U_t\rho^2+1}{t^2\rho^2}\log\left(\frac{\U_t\rho^2+1}{\alpha^2}\right)},\quad \widetilde{L}_t(\rho_k):=\sqrt{\frac{\widehat{\V}_t\rho_k^2+1}{t^2\rho_k^2}\log\left(\frac{\widehat{\V}_t\rho_k^2+1}{\alpha^2}\right)}.
\end{align*}
For each fixed $k$, since $\rho$ and $\rho_k$ are fixed and $\widehat{\V}_t/\U_t\xrightarrow{\mathrm{a.s.}}1$, the same argument as that leading to Equation~\eqref{Eq: plain rho case} gives $L_t(\rho)=\widetilde{L}_t(\rho_k)+o_{\mathrm{a.s.}}\left(\sqrt{\U_t\log\U_t}/t\right)$. Consequently, $\widetilde{L}_t(\rho_k)/L_t(\rho)\xrightarrow{\mathrm{a.s.}}1$ for each fixed $k$.\footnote{A caveat is that the eventual coverage guarantee needs to be given with respect to the exact (unknown) non-asymptotic confidence sequence $C_t^{\star}$ defined in Equation \eqref{Eq: exact unknown CS boundary}, instead of just on $C_t$ constructed directly from Ville's inequality on Brownian motion in Equation \eqref{Eq: Direct Boundary from Villes}. This part of proof for asymptotic equivalence by only checking $L_t, U_t$ is already sufficient as shown by the general recipe in \cite{waudby2024time} (Theorem 2.4 in Section 2.3). } Since the confidence sequence is symmetric and two-sided, it suffices to verify the eventual coverage of the lower side. Unlike the asymptotic-equivalence argument in Part I, the slackness argument requires the sharper rate
\begin{align}
\label{Eq: Specific rate of Vt estimator and Ut}
\frac{\widehat{\V}_t}{\U_t}
=
\frac{\widehat{\V}_t}{\V_t}\frac{\V_t}{\U_t}
=
1+o_{\mathrm{a.s.}}\left(\frac{1}{\log t}\right),
\end{align}
where the last equality follows from the assumed convergence rate of the cumulative variance estimator and Equation~\eqref{Eq: order of Ut/V_t}. Then for every fixed \(k>e\), we could obtain
\begin{align*}
\widetilde{L}_t(\rho_k) - L_t(\rho) =& \sqrt{
\frac{\widehat{\V}_t \rho_k^2 + 1 }{t^2 \rho_k^2}
 \log\!\left(
\frac{\widehat{\V}_t \rho_k^2 + 1}{\alpha^2} \right)
}  - \sqrt{
\frac{\U_t \rho^2 + 1}{t^2 \rho^2}
\log\!\left(
\frac{\U_t \rho^2 + 1}{\alpha^2}
\right)
} \\
=& \sqrt{\frac{\U_t}{t^2}
\left(\frac{\widehat \V_t}{\U_t}+\frac{1}{\U_t\rho_k^2}\right)
\left[\log \U_t+\log\!\left(\frac{\rho_k^2}{\alpha^2}\right)
+\log\!\left(\frac{\widehat \V_t}{\U_t}+\frac{1}{\U_t\rho_k^2}\right)\right]} \\ &- \sqrt{
\frac{\U_t \rho^2 + 1}{t^2 \rho^2}
\log\!\left(
\frac{\U_t \rho^2 + 1}{\alpha^2}
\right)
}.
\end{align*}
Expanding the first term,
\begin{align*}
&\sqrt{\frac{\U_t}{t^2}
\left(\frac{\widehat \V_t}{\U_t}+\frac{1}{\U_t\rho_k^2}\right)
\left[\log \U_t+\log\!\left(\frac{\rho_k^2}{\alpha^2}\right)
+\log\!\left(\frac{\widehat \V_t}{\U_t}+\frac{1}{\U_t\rho_k^2}\right)\right]} \\
= &  \sqrt{\frac{\U_t}{t^2}
\left(1 + o_{\mathrm{a.s.}}\left(\frac{1}{\log \U_t} \right)\right)
\log \U_t \left[  1 +\frac{\log(\rho_k^2/\alpha^2)}{\log \U_t} + o_{\mathrm{a.s.}}\left(\frac{1}{\log \U_t} \right) \right]}  \\
=& \frac{\sqrt{\U_t \log \U_t}}{t} \left[1 +\frac{\log(\rho_k^2/\alpha^2)}{2\log \U_t} + o_{\mathrm{a.s.}}\left(\frac{1}{\log \U_t} \right) \right],
\end{align*}
where the first equality uses Equation~\eqref{Eq: Specific rate of Vt estimator and Ut}, $\U_t\asymp t$, and the expansion $\log(1+x)=x+o(x)$, which give $\widehat{\V}_t/\U_t+1/(\U_t\rho_k^2)=1+\oas{1/\log\U_t}$ and $\log\left(\widehat{\V}_t/\U_t+1/(\U_t\rho_k^2)\right)=\oas{1/\log\U_t}$. The second equality follows from the expansion $\sqrt{1+x}=1+x/2+o(x)$. Similarly, expanding the second term, 
\begin{align*}
\sqrt{
\frac{\U_t \rho^2 + 1}{t^2 \rho^2}
\log\!\left(
\frac{\U_t \rho^2 + 1}{\alpha^2}
\right)
} &= \sqrt{\frac{\U_t}{t^2}
\left(1 +\frac{1}{\U_t\rho^2}\right)
\left[\log \U_t+\log\!\left(\frac{\rho^2}{\alpha^2}\right)
+\log\!\left(1+\frac{1}{\U_t\rho^2}\right)\right]} \\
&= \frac{\sqrt{\U_t \log \U_t}}{t} \left(1 +\frac{\log(\rho^2/\alpha^2)}{2\log \U_t} + o_{\mathrm{a.s.}}\left(\frac{1}{\log \U_t} \right) \right).
\end{align*}
Therefore, subtracting the two expansions and using $\rho_k/\rho=\log k$, we obtain
\begin{align*}
\widetilde{L}_t(\rho_k)-L_t(\rho)
=
\frac{\sqrt{\U_t}}{t\sqrt{\log\U_t}}
\left(
\log\left(\frac{\rho_k}{\rho}\right)+\oas{1}
\right) =
\frac{\sqrt{\U_t}}{t\sqrt{\log\U_t}}
\left(\log\log k+\oas{1}\right).
\end{align*}

\noindent Since $\log\log k>0$ for every fixed $k>e$, it follows that $\widetilde{L}_t(\rho_k)-L_t(\rho)>0$ eventually almost surely. Furthermore, define $L_t^{\star}(\rho):=L_t(\rho)+\abs{\eps_t}$ as in Equation~\eqref{Eq: exact unknown CS boundary}, and recall that $\eps_t=\oas{t^{1/p-1}\sqrt{\log t}}+\oas{t^{\beta/2-1}\sqrt{\log t}}$, where $p>2$ and $\beta\in(0,1)$. Since $\U_t\asymp t$, we have $\abs{\eps_t}=\oas{\sqrt{\U_t}/\left(t\sqrt{\log\U_t}\right)}$. Therefore,
\begin{align*}
\widetilde{L}_t(\rho_k)-L_t^{\star}(\rho)
&=\widetilde{L}_t(\rho_k)-L_t(\rho)-\abs{\eps_t}\\
&=\frac{\sqrt{\U_t}}{t\sqrt{\log\U_t}}\left[\log\log k+\oas{1}\right].
\end{align*}
Since $\log\log k>0$ for every fixed $k>e$, there exists an almost surely finite random time $T_k$ such that $L_t^{\star}(\rho)\leq\widetilde{L}_t(\rho_k)$ for all $t\geq T_k$. Hence, $C_t^{\star}\subseteq\widetilde{C}_t$ for all $t\geq T_k$, which proves the eventual inclusion.\footnote{The random time $T_k$ may depend on $k$, which is sufficient here. By exploiting the monotonicity of $\widetilde{L}_t(\rho_k)$ on a suitable region, one could obtain a uniform random time independent of $k$.} 
\end{proof}

\subsection{Proof of Corollary \ref{Corollary: Equivalent Asymp CS Boundary}} \label{Appendix: Proof of Corollary  Equivalent Asymp CS Boundary}

\begin{proof}
Under the current assumptions, denote the boundaries in Theorem \ref{Thm: Robbins Gaussian Mixture Asymp CS} and Corollary \ref{Corollary: Equivalent Asymp CS Boundary}, respectively, by
\[
\widetilde{C}^{\mathcal{G},\mathrm{oracle}}_t := \left(
\widehat{\mu}_t \pm
\sqrt{
\frac{\bigl(t\widehat{\sigma}^2_t \rho^2 + 1\bigr)}{t^2\rho^2}
\log\!\left(\frac{t\widehat{\sigma}^2_t\rho^2 + 1}{\alpha^2}\right)
}
\right), \quad 
\widetilde{C}^{\mathcal{G},\mathrm{equiv}}_t := \left(
\widehat{\mu}_t \pm \widehat{\sigma}_t
\sqrt{
\frac{t  \rho^2 + 1}{t^2\rho^2}
\log\!\left(\frac{t\rho^2 + 1}{\alpha^2}\right)
}
\right)
\]
We then first show the almost sure relative approximations between two boundaries, i.e., 
\begin{align*}
    \frac{
        (t\hat{\sigma}_t^2\rho^2+1)
        \log((t\hat{\sigma}_t^2\rho^2+1)/\alpha^2)
    }{
        \hat{\sigma}_t^2(t\rho^2+1)
        \log((t\rho^2+1)/\alpha^2)
    }
    \overset{\mathrm{a.s.}}{\longrightarrow}1.
\end{align*}
Since $\hat{\sigma}_t^2 \overset{\mathrm{a.s.}}{\longrightarrow}\sigma^2>0$ and $t\to \infty$, it follows that $(t\hat{\sigma}_t^2\rho^2+1)/\hat{\sigma}_t^2(t\rho^2+1) \overset{\mathrm{a.s.}}{\longrightarrow} 1$. For the logarithmic factor, observe that 
\begin{align*}
    \frac{\log((t\hat{\sigma}_t^2\rho^2+1)/\alpha^2)}{\log((t\rho^2+1)/\alpha^2)} - 1 & = \frac{\log\left(t\widehat{\sigma}_t^2\rho^2 + 1\right) - \log(t\rho^2+1)}{\log\left(t\rho^2 + 1\right) - \log(\alpha^2)}
\end{align*}
The numerator is $O_{\mathrm{a.s.}}(1)$ because $\hat{\sigma}_t^2 \overset{\mathrm{a.s.}}{\longrightarrow}\sigma^2$, whereas the denominator diverges to infinity as $t\to\infty$. Hence, the logarithmic factor converges to one almost surely. Next, with a slight abuse of notation, fix $\rho >0$, and observe the exact algebraic identity that 
$\widetilde{C}^{\mathcal{G},\mathrm{equiv}}_t (\rho) = \widetilde{C}^{\mathcal{G},\mathrm{oracle}}_t (\rho /\widehat{\sigma}_t)$. Then by the same slack trick in Appendix \ref{Appendix: Proofs of Gaussian Mixture Asymp CS}, for any pre-specified value $\rho > 0$ in $\widetilde{C}^{\mathcal{G},\mathrm{equiv}}_t$, fix $k > e$, and set the tuning parameter in $\widetilde{C}^{\mathcal{G},\mathrm{oracle}}_t$ as $\rho_0 := \rho/(\sigma \log k)$. The eventual-inclusion property could then be similarly established, which completes the proof. 
\end{proof}

\section{Proofs of Results in Section \ref{Sec: Strongly Consistent Estimation Strategies}} \label{Appendix: Sec_C}

\noindent This section provides supporting proofs for the strong consistency of HAC estimators in the FDM framework. Section~\ref{Proof of Prop: A.S. Oracle HAC Estimator} proves the strong consistency of the oracle HAC estimator. Section~\ref{Proof of Thm: A.S. Plug-in HAC Estimator} shows that under regularity conditions, replacing the unknown centered process with residuals based on the running mean preserves the strong consistency and convergence rate of the HAC estimator.

\subsection{Proof of Proposition \ref{Prop: Strong Consistency of Oracle HAC Estimator}} \label{Proof of Prop: A.S. Oracle HAC Estimator}
\noindent We begin by stating a lemma showing that the bandwidth choice $b_t$ in Condition~\ref{Cond: primitive-bandwidth-rate} is sufficient for the summability conditions used later in the main proof.

\begin{Lem} \label{Lem: Suff Cond for Oracle HAC Bandwidth}
For $p >2$, denote \(q:=p/2>1\) and $q_0:= \min\{q,2\}.$ Then define the dyadic block bandwidth envelope $b_\ell^{\star}
:=
\max_{2^\ell \le t < 2^{\ell+1}} \lfloor b_t \rfloor$, and then the bandwidth condition in Condition \ref{Cond: primitive-bandwidth-rate} is sufficient to obtain the dyadic condition $\sum_{\ell=1}^{\infty}
(b_\ell^\star)^q2^{-\ell q(1-1/q_0)} < \infty.$
\end{Lem}

\begin{proof}
\noindent For $2^\ell \le t < 2^{\ell+1}$, we have $t^{1-1/q_0} \le 2^{(\ell+1)(1-1/q_0)}$, and $\log t \ge \ell \log 2.$ Hence,
\begin{align*}
b_t 
\le 
C t^{1-1/q_0}(\log t)^{-\lambda}
\le
C \frac{2^{(\ell+1)(1-1/q_0)}}{(\ell\log 2)^\lambda}.
\end{align*}
Note that this upper bound of $b_t$ now does not depend on $t$. Taking the maximum over the dyadic block yields
\begin{align*}
b_\ell^\star
:=
\max_{2^\ell \le t < 2^{\ell+1}} b_t
\le
C
\frac{2^{(\ell+1)(1-1/q_0)}}{(\ell\log 2)^\lambda}.
\end{align*}
Therefore, we have $(b_\ell^\star)^q 2^{-\ell q(1-1/q_0)} \lesssim
\ell^{-\lambda q}.$ Since $\lambda>1/q$, we have $\lambda q>1$, and thus $\sum_{\ell=1}^{\infty}\ell^{-\lambda q}<\infty.$ Consequently the dyadic condition is satisfied: $\sum_{\ell=1}^{\infty}
(b_\ell^\star)^q 2^{-\ell q(1-1/q_0)} < \infty.$
\end{proof}

\vspace{12pt}
\noindent \textbf{Main proof of Proposition~\ref{Prop: Strong Consistency of Oracle HAC Estimator}.} For notational simplicity, denote \(q:=p/2>1\), and $q_0:= \min\{q,2\}.$ To distinguish the
finite-sample variance from its limit \(\sigma^2\), define
\(\sigma_t^2:=\mathsf{V}_t/t\) as in Corollary \ref{Corollary: Equivalent Asymp CS Boundary}. We divide the proof into two parts corresponding to the following decomposition:
\begin{align*}
\widetilde{\sigma}_t^2-\sigma_t^2
=
\underbrace{\left(
\widetilde{\sigma}_t^2-\E[\widetilde{\sigma}_t^2]
\right)}_{\text{the first (stochastic) part}}
+
\underbrace{\left(
\E[\widetilde{\sigma}_t^2]-\sigma_t^2
\right).}_{\text{the second (deterministic) part}}
\end{align*}
We first control the stochastic fluctuation
\(\widetilde{\sigma}_t^2-\E[\widetilde{\sigma}_t^2]\), and then bound the deterministic bias
\(\E[\widetilde{\sigma}_t^2]-\sigma_t^2\). For strong consistency, it suffices to show that both parts converge almost surely to zero. We derive the corresponding specific convergence rates at the end of the proof.
\vspace{5pt}

\begin{proof}[Proof of first (stochastic) part:]
Since \(K\) is even, \(K(u)=0\) for \(|u|>1\), and $K(0)=1$, we can re-index the double sum and rewrite the first stochastic part: 
\begin{align*}
\widetilde{\sigma}_t^2-\E[\widetilde{\sigma}_t^2]
&=\frac{1}{t}
\sum_{i=1}^{t}
\sum_{j=1}^{t}
K\!\left(\frac{i-j}{b_t}\right)
\left(X_iX_j-\E[X_iX_j] \right) \\
&=  \frac{1}{t}\sum_{i=1}^{t}\left(X_i^2-\E[X_i^2]\right)
+
\frac{2}{t}\sum_{h=1}^{\lfloor b_t\rfloor}
K\!\left(\frac{h}{b_t}\right)
\sum_{i=1}^{t-h}
\left(X_iX_{i+h}-\E[X_iX_{i+h}]\right) \\
&= 
\frac1t\sum_{i=1}^{t}\xi_{i,0}
+
\frac2t\sum_{h=1}^{\lfloor b_t\rfloor}
K\left(\frac{h}{b_t}\right)\sum_{i=1}^{t-h}\xi_{i,h},
\end{align*}
where we define\footnote{Note that $X_i$ is assumed to be mean zero in the oracle HAC estimator, but not the second moment, namely the product \(X_iX_{i+h}\).} for \(h\ge 0\), $\xi_{i,h} := X_iX_{i+h}-\E[X_iX_{i+h}].$

\vspace{3pt}
On the dyadic block $2^\ell \le t < 2^{\ell+1},$  we have $t^{-1} \le 2^{-\ell}$ and $\lfloor b_t\rfloor \le b_\ell^\star:=
\max_{2^\ell \le t < 2^{\ell+1}}
\lfloor b_t \rfloor.$
Hence, for every \(t\in[2^\ell,2^{\ell+1})\), the inner partial sum has length at most \(2^{\ell+1}\). Therefore, for $h \ge 0$,
\[
\left|
\sum_{i=1}^{t-h}\xi_{i,h}
\right|
\le
\max_{1\le m\le 2^{\ell+1}}
\left|
\sum_{i=1}^{m}\xi_{i,h}
\right|.
\]
It follows that
\begin{align} \label{Step: Dyadic blocking inner sum}
\max_{2^\ell\le t<2^{\ell+1}}
\left|
\widetilde{\sigma}_t^2-\E[\widetilde{\sigma}_t^2]
\right|
\le
C\,2^{-\ell}
\sum_{h=0}^{b_\ell^\star}
\max_{1\le m\le 2^{\ell+1}}
\left|
\sum_{i=1}^{m}\xi_{i,h}
\right|,
\end{align}
where we assume the kernel function to be bounded so that $|K(h/b_t)|\le C.$

\vspace{3pt}
As in the proof of Lemma~\ref{lem:maximal_ineq}, let $\mathcal{F}_j:=\sigma(\eps_j,\eps_{j-1},\ldots)$ and define the projection operator $\mathcal{P}_j[Z]:=\E[Z\mid\mathcal{F}_j]-\E[Z\mid\mathcal{F}_{j-1}]$. For fixed $i,h$, the random variable $\xi_{i,h}$ is $\mathcal{F}_{i+h}$-measurable. For every finite $R$, the telescoping identity gives $\sum_{r=0}^{R}\mathcal{P}_{i+h-r}[\xi_{i,h}]=\xi_{i,h}-\E[\xi_{i,h}\mid\mathcal{F}_{i+h-R-1}]$. Assumption~\ref{asm:dep_decay} gives $\sup_{i\in\mathbb{Z}}\norm{X_i}_p\leq\Theta_{0,p}$ and hence $\xi_{i,h}\in\mathcal{L}_q$. By the reverse martingale convergence theorem, $\E[\xi_{i,h}\mid\mathcal{F}_{i+h-R-1}]\to\E[\xi_{i,h}\mid\mathcal{F}_{-\infty}]=\E[\xi_{i,h}]=0$ in $\mathcal{L}_q$, where $\mathcal{F}_{-\infty}:=\bigcap_{j\in\mathbb{Z}}\mathcal{F}_j$ is trivial by the Kolmogorov zero-one law. Letting $R\to\infty$ yields
\begin{equation}\label{eq: xi_{i,h}}
\xi_{i,h}
=
\sum_{r=0}^{\infty}
\mathcal{P}_{i+h-r}[\xi_{i,h}]
\qquad\text{in }\mathcal{L}_q.
\end{equation}

For each fixed $h,r$, the sequence $(\mathcal{P}_{i+h-r}[\xi_{i,h}])_{i\geq1}$ forms martingale differences with respect to $(\mathcal{F}_{i+h-r})_{i\geq1}$. We use $N$ as the terminal index here, reserving the dyadic-block index for later. Doob's maximal inequality, the von Bahr–Esseen inequality (\cite{von1965inequalities}, for $q<2$) and Burkholder moment inequality (\cite{burkholder1973distribution}, for $q \ge 2$) imply
\begin{equation} \label{Step: DB ineq}
\left\|
\max_{1\leq m\leq N}
\left|
\sum_{i=1}^{m}
\mathcal{P}_{i+h-r}[\xi_{i,h}]
\right|
\right\|_q
\leq
\frac{q}{q-1}
\left\|
\sum_{i=1}^{N}
\mathcal{P}_{i+h-r}[\xi_{i,h}]
\right\|_q 
\lesssim
\left(
\sum_{i=1}^{N}
\norm{\mathcal{P}_{i+h-r}[\xi_{i,h}]}_q^{q_0}
\right)^{1/q_0}    
\end{equation}
It remains to bound $\norm{\mathcal{P}_{i+h-r}[\xi_{i,h}]}_q$. Let $X_u^{*(j)}$ denote the coupled version of $X_u$ obtained by replacing $\eps_j$ with an independent copy, and let $\xi_{i,h}^{*(j)}$ denote the corresponding coupled version of $\xi_{i,h}$. Since
$\E[\xi_{i,h}\mid\mathcal{F}_{i+h-r-1}]
=\E[\xi_{i,h}^{*(i+h-r)}\mid\mathcal{F}_{i+h-r-1}]
=\E[\xi_{i,h}^{*(i+h-r)}\mid\mathcal{F}_{i+h-r}]$, we have
\begin{align*}
\norm{\mathcal{P}_{i+h-r}[\xi_{i,h}]}_q
=
\left\|
\E\left[
\xi_{i,h}-\xi_{i,h}^{*(i+h-r)}
\mid
\mathcal{F}_{i+h-r}
\right]
\right\|_q.
\end{align*}
Since conditional expectation is an $\mathcal{L}_q$ contraction, the triangle inequality and H\"older's inequality with $q=p/2$ give
\begin{align*}
\norm{\mathcal{P}_{i+h-r}[\xi_{i,h}]}_q
&\leq
\norm{\xi_{i,h}-\xi_{i,h}^{*(i+h-r)}}_q\\
&=
\norm{
X_iX_{i+h}
-
X_i^{*(i+h-r)}X_{i+h}^{*(i+h-r)}
}_q\\
&\leq
\norm{\left(X_i-X_i^{*(i+h-r)}\right)X_{i+h}}_q
+
\norm{
X_i^{*(i+h-r)}
\left(X_{i+h}-X_{i+h}^{*(i+h-r)}\right)
}_q\\
&\lesssim
\norm{X_i-X_i^{*(i+h-r)}}_p
+
\norm{X_{i+h}-X_{i+h}^{*(i+h-r)}}_p.
\end{align*}

\noindent By the definition of the uniform functional dependence measure in Equation~\eqref{eq:f_measure},
\begin{align*}
\norm{\mathcal{P}_{i+h-r}[\xi_{i,h}]}_q
\leq
C\left[
\mathbf{1}_{\{r\geq h\}}\delta_p(r-h)
+
\delta_p(r)
\right].
\end{align*}
Consequently,
\begin{equation}\label{Step: sum sup proj operator}
\sum_{r=0}^{\infty}
\sup_i
\norm{\mathcal{P}_{i+h-r}[\xi_{i,h}]}_q
\leq
C
\sum_{r=0}^{\infty}
\left[
\delta_p(r)
+
\mathbf{1}_{\{r\geq h\}}\delta_p(r-h)
\right]
\leq
C\Theta_{0,p},
\end{equation}
where the last inequality follows by the change of variables $k=r-h$ and the definition of $\Theta_{0,p}$.

\vspace{3pt}
We next bound the inner sum on the right-hand side of Equation~\eqref{Step: Dyadic blocking inner sum}. Using the decomposition in Equation~\eqref{eq: xi_{i,h}}, Minkowski's inequality, and Equation \eqref{Step: DB ineq}, we obtain 
\begin{align}\label{Step: N^power control}
\Biggl\|
\max_{1\leq m\leq N}
\Bigl|
\sum_{i=1}^{m}\xi_{i,h}
\Bigr|
\Biggr\|_q
&\leq
\sum_{r=0}^{\infty}
\Biggl\|
\max_{1\leq m\leq N}
\Bigl|
\sum_{i=1}^{m}
\mathcal{P}_{i+h-r}[\xi_{i,h}]
\Bigr|
\Biggr\|_q
\nonumber\\
&\lesssim
\sum_{r=0}^{\infty}
\left(
\sum_{i=1}^{N}
\norm{\mathcal{P}_{i+h-r}[\xi_{i,h}]}_q^{q_0}
\right)^{1/q_0}
\nonumber\\
&\leq
C N^{1/q_0}
\sum_{r=0}^{\infty}
\sup_i
\norm{\mathcal{P}_{i+h-r}[\xi_{i,h}]}_q
\leq
C N^{1/q_0},
\end{align}
where the last inequality follows from Equation~\eqref{Step: sum sup proj operator}.

\vspace{3pt}
We now return to the outer sum in Equation~\eqref{Step: Dyadic blocking inner sum}. Setting $N=2^{\ell+1}$ and applying Minkowski's inequality and Equation~\eqref{Step: N^power control} give
\begin{align*}
\left\|
\max_{2^\ell \leq t < 2^{\ell+1}}
\left|
\widetilde{\sigma}_t^{\,2}
-
\E[\widetilde{\sigma}_t^{\,2}]
\right|
\right\|_q
&\leq
C\,2^{-\ell}
\left\|
\sum_{h=0}^{b_\ell^\star}
\max_{1\leq m\leq 2^{\ell+1}}
\left|
\sum_{i=1}^{m}\xi_{i,h}
\right|
\right\|_q\\
&\leq
C\,2^{-\ell}
\sum_{h=0}^{b_\ell^\star}
\left\|
\max_{1\leq m\leq 2^{\ell+1}}
\left|
\sum_{i=1}^{m}\xi_{i,h}
\right|
\right\|_q\\
&\leq
C\,2^{-\ell}
\sum_{h=0}^{b_\ell^\star}
(2^{\ell+1})^{1/q_0}\leq
C b_\ell^\star 2^{-\ell(1-1/q_0)}.
\end{align*}
Therefore,
\begin{align}\label{Eq: Dyadic Lq norm}
\sum_{\ell=1}^{\infty}
\left\|
\max_{2^\ell \leq t < 2^{\ell+1}}
\left|
\widetilde{\sigma}_t^{\,2}
-
\E[\widetilde{\sigma}_t^{\,2}]
\right|
\right\|_q^q
\leq
C\sum_{\ell=1}^{\infty}
(b_\ell^\star)^q2^{-\ell q(1-1/q_0)}
<
\infty,
\end{align} 
where the last inequality follows from Condition~\ref{Cond: primitive-bandwidth-rate} and Lemma~\ref{Lem: Suff Cond for Oracle HAC Bandwidth}. For every $\varepsilon>0$, Equation~\eqref{Eq: Dyadic Lq norm} and Markov's inequality imply
\begin{align}\label{Eq: Dyadic Borel-Cantelli}
\sum_{\ell=1}^{\infty}
\P\left[
\max_{2^\ell \leq t < 2^{\ell+1}}
\left|
\widetilde{\sigma}_t^{\,2}
-
\E[\widetilde{\sigma}_t^{\,2}]
\right|
>
\varepsilon
\right]
<
\infty.
\end{align}
The Borel--Cantelli lemma therefore gives
$\max_{2^\ell\leq t<2^{\ell+1}}
\left|\widetilde{\sigma}_t^{\,2}-\E[\widetilde{\sigma}_t^{\,2}]\right|
\to0$ almost surely. Since every $t$ belongs to a dyadic block, $\left|
\widetilde{\sigma}_t^{\,2}
-
\E[\widetilde{\sigma}_t^{\,2}]
\right|
\leq
\max_{2^\ell\leq s<2^{\ell+1}}
\left|
\widetilde{\sigma}_s^{\,2}
-
\E[\widetilde{\sigma}_s^{\,2}]
\right|
\to0
\ \mathrm{a.s.}$ This completes the proof of the stochastic part. 
\end{proof}

\vspace{5pt}
\begin{proof}[Proof of second (deterministic) part:]
Define the uniform covariance envelope $\gamma_2(h):=\sup_i\abs{\E[X_iX_{i+h}]}$. Following Equation~(39) of \cite{wu2012covariance}, we can decompose the deterministic error into two components, corresponding to the kernel bias and the truncation bias, respectively, as follows:
\begin{equation} \label{Eq: kernel bias + truncation bias}
\abs{\E[\widetilde{\sigma}_t^2]-\sigma_t^2}
\leq
2\sum_{h=1}^{\lfloor b_t\rfloor}
\abs{1-K(h/b_t)}\gamma_2(h)
+
2\sum_{h=\lfloor b_t\rfloor+1}^{t-1}
\gamma_2(h), \quad \gamma_2(h)\leq\sum_{r=0}^{\infty}\delta_2(r)\delta_2(r+h)\footnote{Equation (39) of \cite{wu2012covariance} applies to both $2<p<4$ and $p \ge 4$, since the deterministic bias bound itself
uses only second-order projection arguments. The covariance envelope inequality follows from the projection argument in Proposition 8, Equations (31) and (34) of \cite{xiao2011asymptotic}, with the bounds taken uniformly to accommodate nonstationarity. }.
\end{equation}
The short-range dependence condition from Assumption \ref{Assump:dep_decay}, together with $p > 2$ gives $\sum_{r=0}^{\infty}\delta_2(r)<\infty$ by the monotonicity of $\mathcal{L}_p$ norms. Hence,
\begin{equation} \label{Eq: summability}
\sum_{h=0}^{\infty}\gamma_2(h)
\leq
\sum_{h=0}^{\infty}
\sum_{r=0}^{\infty}
\delta_2(r)\delta_2(r+h)\leq
\left(
\sum_{r=0}^{\infty}\delta_2(r)
\right)^2
<
\infty.
\end{equation}
For each fixed $h$, continuity of $K$ at $0$ and $b_t\to\infty$ imply $K(h/b_t)\to K(0)=1$. Since $K$ is bounded and $\sum_{h=0}^{\infty}\gamma_2(h)<\infty$, the dominated convergence theorem gives $\sum_{h=1}^{\lfloor b_t\rfloor}
\abs{1-K(h/b_t)}
\gamma_2(h)
\to0.$ The summability of $\gamma_2(h)$ and $b_t\to\infty$ also imply $\sum_{h=\lfloor b_t\rfloor+1}^{t}
\gamma_2(h)
\leq
\sum_{h=\lfloor b_t\rfloor+1}^{\infty}
\gamma_2(h)
\to0.$
Therefore, $\E[\widetilde{\sigma}_t^2]-\sigma_t^2\to0$.
\end{proof}

\vspace{8pt}

\begin{proof}[Proof of the specific rate of strong consistency in Corollary \ref{Corollary: Equivalent Asymp CS Boundary}]
Condition~\ref{Cond: o(1/logt) oracle HAC bandwidth} strengthens the bandwidth requirements in Condition~\ref{Cond: primitive-bandwidth-rate}, so Lemma~\ref{Lem: Suff Cond for Oracle HAC Bandwidth} remains applicable. By the triangle inequality, it suffices to verify the desired rate for the stochastic and deterministic parts separately.

\vspace{3pt}
\noindent We first consider the stochastic part. On the dyadic block $2^\ell\leq t<2^{\ell+1}$, we have $\log t\asymp\ell$. Therefore, for every $\varepsilon>0$, it suffices to show
\begin{align}\label{Eq: Dyadic Borel-cantelli on o(1/logt)}
\sum_{\ell=1}^{\infty}
\P\left[
\max_{2^\ell\leq t<2^{\ell+1}}
\left|
\widetilde{\sigma}_t^{\,2}
-
\E[\widetilde{\sigma}_t^{\,2}]
\right|
>
\frac{\varepsilon}{\ell}
\right]
<
\infty.
\end{align}
Recall that $q=p/2$ and $q_0:= \min\{q,2\}.$ Lemma~\ref{Lem: Suff Cond for Oracle HAC Bandwidth} still applies and gives  $(b_\ell^\star)^q 2^{-\ell q(1-1/q_0)} \lesssim
\ell^{-\lambda q}.$ By Markov's inequality, the $\ell$-th probability in Equation~\eqref{Eq: Dyadic Borel-cantelli on o(1/logt)} is therefore bounded by
$C\varepsilon^{-q}\ell^{q-\lambda q}
=C\varepsilon^{-q}\ell^{-(\lambda-1)q}$.
The series is summable whenever $(\lambda-1)q>1$, or equivalently, $\lambda>1+2/p$. The Borel--Cantelli lemma and $\log t\asymp\ell$ then yield
$\widetilde{\sigma}_t^{\,2}-\E[\widetilde{\sigma}_t^{\,2}]
=\oas{1/\log t}$.

\vspace{3pt}
\noindent It remains to consider the deterministic part. By an argument similar to that in Equations \eqref{Eq: kernel bias + truncation bias}-\eqref{Eq: summability}, for any lag cutoff $b \ge 1$, $\sum_{h>b}\gamma_2(h) \le \sum_{r=0}^{\infty}\delta_2(r)\Theta_{r+b+1,2} = O(b^{-A})$ under Assumption \ref{Assump:dep_decay}. Then by the ``local H\"older condition'' of order $\kappa$ at zero of Condition \ref{Cond: o(1/logt) oracle HAC bandwidth} and the implied boundedness property of the kernel $K$, reindexing Equation \eqref{Eq: kernel bias + truncation bias} gives
\begin{align*}
\abs{\E[\widetilde{\sigma}_t^2]-\sigma_t^2}
&\lesssim
b_t^{-\kappa}\sum_{h\leq b_t}h^\kappa\gamma_2(h)
+
\sum_{h>b_t}\gamma_2(h) =
O(b_t^{-\kappa})+O(b_t^{-A}) = O(b_t^{-\kappa})
\end{align*}
where the first bound follows from summation by parts and $\kappa<A$, which imply $\sum_{h=1}^{\infty} h^\kappa\gamma_2(h) < \infty$, and the second equality follows from $\kappa<A$.

Finally, the lower bound bandwidth condition $b_t^\kappa/\log t\to\infty$ implies
$O(b_t^{-\kappa})=o(1/\log t)$.  
\end{proof}

\subsection{Proof of Theorem \ref{Thm: Strong Consistency of Plug-in HAC Estimator}} \label{Proof of Thm: A.S. Plug-in HAC Estimator} 
\noindent Under Condition \ref{Cond: Ces\`aro Mean-Stability Condition} and the same assumptions as in Proposition \ref{Prop: Strong Consistency of Oracle HAC Estimator}, Theorem \ref{Thm: Strong Consistency of Plug-in HAC Estimator}
follows once we establish $\widehat{\sigma}_t^2 = \widetilde{\sigma}_t^2 + o_{\mathrm{a.s.}}(r_t)$
for some $r_t \to 0$. We introduce vector notation for clarity. Let
$\bm{X}_t = (X_1,\dots,X_t)'$, $\bm{\mu}_t = (\mu_1,\dots,\mu_t)'$, and
$\1_t = (1,\dots,1)'$. Define $\bm{d}_t = \bm{\mu}_t - \widetilde{\mu}_t\1_t$,
which satisfies $\bm{d}_t'\1_t = 0$ by construction, and define $\overline{X}_t := t^{-1}\sum_{i=1}^t X_i$.
By Equation \eqref{eq:additive_decomposition}, we have $\widehat{\mu}_t = \widetilde{\mu}_t + \overline{X}_t$. After residualizing 
$\widehat{X}_{i,t} = Y_i - \widehat{\mu}_t$ and stacking them into
$\widehat{\bm{X}}_t = (\widehat{X}_{1,t},\dots, \widehat{X}_{t,t})'$,
we obtain 
\begin{align*}
    \widehat{\bm{X}}_t = \bm{X}_t + \bm{d}_t - \overline{X}_t \1_t =: \bm{X}_t + \bm{u}_t.
\end{align*}
Let $\mathbf{W}_t = (K((i-j)/b_t))_{1\leq i,j\leq t}$, so that the two variance
estimators take the quadratic forms $\widehat{\sigma}_t^2 = t^{-1} \widehat{\bm{X}}_t'
\mathbf{W}_t \widehat{\bm{X}}_t$ and $\widetilde{\sigma}_t^2 = t^{-1} \bm{X}_t'
\mathbf{W}_t \bm{X}_t$. Since $K$ is a positive semi-definite kernel, $\mathbf{W}_t\succeq \bm{0}$, and by boundedness on a compact support $[-1,1]$, $\|\mathbf{W}_t\|_{\mathrm{op}} \leq C_K b_t$
for a kernel-dependent constant $C_K$. Expanding the difference and applying
the Cauchy--Schwarz inequality yields
\begin{align*}
    \abs{\widehat{\sigma}_t^2 - \widetilde{\sigma}_t^2}
    \leq \frac{2}{t} \abs{\bm{X}_t' \mathbf{W}_t \bm{u}_t}
    + \frac{1}{t} \bm{u}_t' \mathbf{W}_t\bm{u}_t
    \leq 2\widetilde{\sigma}_t \sqrt{Q_t} + Q_t,
\end{align*}
where $Q_t = \bm{u}_t'\mathbf{W}_t\bm{u}_t/t$.

\vspace{3pt}
It therefore suffices to show
$Q_t = o_{\mathrm{a.s.}}(r_t^2)$ for some $r_t\to 0$. Using $\bm{u}_t =
\bm{d}_t - \overline{X}_t\1_t$ and the elementary inequality
$(\bm{a} + \bm{b})'\mathbf{W}_t(\bm{a}+\bm{b}) \leq 2(\bm{a}'\mathbf{W}_t\bm{a}
+ \bm{b}'\mathbf{W}_t\bm{b})$, which follows from $\mathbf{W}_t\succeq \bm{0}$, we obtain
\begin{align*}
    Q_t
    \leq \frac{2}{t} \bm{d}_t' \mathbf{W}_t \bm{d}_t
    + 2\overline{X}_t^2\left( \frac{1}{t}\1_t'\mathbf{W}_t\1_t \right)
    \leq2 \norm{\mathbf{W}_t}_{\mathrm{op}}
    \left(\frac{\norm{\bm{d}_t}^2}{t} + \overline{X}_t^2 \frac{\norm{\1_t}^2}{t}\right)
    \lesssim b_t \left(A_{\mu,t} + \overline{X}_t^2\right),
\end{align*}
where $A_{\mu,t} = t^{-1}\sum_{i=1}^t (\mu_i - \widetilde{\mu}_t)^2$ is controlled
by Condition \ref{Cond: Ces\`aro Mean-Stability Condition}. 

\vspace{3pt}
When $\mathsf{V}_t\asymp t$, Theorem~\ref{Thm: nonstationary ASIP} together with a standard Brownian maximal inequality gives $\overline{X}_t = O_{\mathrm{a.s.}}(\sqrt{\log t /t})$, and hence $Q_t \lesssim b_t\bigl(A_{\mu,t} + O_{\mathrm{a.s.}}(\log t/t)\bigr)$. Finally, taking $c_t := b_t (\log t)^2$ in Condition \ref{Cond: Ces\`aro Mean-Stability Condition} and under the bandwidth condition in Condition \ref{Cond: o(1/logt) oracle HAC bandwidth}, it follows that $Q_t = \oas{1/{(\log t)^2}}.$ Taken together with Proposition \ref{Prop: Strong Consistency of Oracle HAC Estimator}, this completes the proof.

\section{AsympCSs for SGD and Proof in Section \ref{Sec: Sampling and Optimization Algorithms}} \label{Appendix: Sampling and Optimization Algorithms}
\noindent In this section, we provide further discussion of SGD in Appendix~\ref{Appendix: Stochastic Gradient Descent}, followed by the proof of Proposition~\ref{Prop: Asymp CS for MCMC averages} in Appendix~\ref{Proof of Prop: Asymp CS for MCMC averages}.

\subsection{Stochastic Gradient Descent} \label{Appendix: Stochastic Gradient Descent}
\noindent The same Markovian framework in Section \ref{Sec: Sampling and Optimization Algorithms} also applies to stochastic gradient descent (SGD) with a constant step size\footnote{Under a constant step size regime, the variance accumulation is non-degenerate, satisfying $\V_t\asymp t$ as required in Theorem \ref{Thm: Robbins Gaussian Mixture Asymp CS}.}.  Specifically, consider
\begin{align*}
    \bm{\theta}_i
    =
    \bm{\theta}_{i-1}
    -
    \eta\nabla_{\bm{\theta}} \ell(\bm{\theta}_{i-1},\bm{\zeta}_i),
\end{align*}
where $\eta>0$ is a fixed constant, $(\bm{\zeta}_i)_{i\geq1}$ are i.i.d.\ stochastic inputs, and $\ell$ is the sample loss function.

\vspace{3pt}
Despite its lack of exact convergence because a stepsize-dependent stationary bias still persists, constant-step SGD has still gained widespread adoption in practice because of its computational advantages. In particular, it simplifies hyperparameter tuning, rapidly forgets the effect of the initialization, and naturally supports parallel computation through the aggregation of multiple independent SGD trajectories. The Markov-chain perspective on constant-stepsize SGD has been developed in \cite{dieuleveut2020bridging}, \cite{merad2025convergence} and  \cite{wei2025online}. Under strong convexity, smoothness, and moment conditions, \cite{li2024sgd} show that the SGD recursion has a geometric-moment contraction property within the functional dependence framework, so Lipschitz functionals of the iterates satisfy geometric dependence decay.

\vspace{3pt}
A convenient objective in determining whether SGD can be stopped is an observable surrogate-based stochastic gradient energy: Let $Y_i :=  \|\nabla_\theta\ell(\bm{\theta}_{i-1},\bm{\zeta}_i)\|^2
=
\eta^{-2}\|\bm{\theta}_i-\bm{\theta}_{i-1}\|^2$, and then we define the estimand and the estimator respectively as 
\begin{align} \label{Eq: surrogate stochastic gradient energy}
    \bar\mu_t := \frac{1}{t}\sum_{i=1}^t \E[Y_i], \quad 
    \widehat\mu_t := \frac{1}{t}\sum_{i=1}^t Y_i,
\end{align}
where under unbiased stochastic gradient, Jensen's inequality gives that the surrogate-based estimand of interest $\bar\mu_t$ upper bounds the average expected population-gradient energy along the original, possibly nonstationary, SGD trajectory. 

\vspace{3pt}
While the monitored quantity is not the objective gap itself, it provides a natural proxy for optimization performance. For nonconvex problems, a small average population-gradient energy certifies approximate first-order stationarity. Under stronger structural assumptions, such as convexity or the Polyak--{\L}ojasiewicz (PL) condition, it further controls the expected objective suboptimality. Thus, monitoring the average population-gradient energy also provides a certificate for the expected objective performance of the averaged iterate.

\vspace{3pt}
Under the regularity conditions as in Corollary \ref{Corollary: Equivalent Asymp CS Boundary}, applying a one-sided version of the Robbins Gaussian mixture boundary yields that for any $\rho>0$ and $\alpha\in(0,1)$,
\begin{align} \label{Eq: one-sided upper boundary}
    U_t^+(\alpha)
    :=
    \widehat\mu_t
    +
    \widehat\sigma_t
    \sqrt{
    \frac{2 (t\rho^2+1)}{t^2\rho^2}
    \log\!\left(1 + \frac{\sqrt{t\rho^2+1}}{2\alpha}\right)
    }
\end{align}
forms an upper $(1-\alpha)$-AsympCS for the estimand of interest $\bar{\mu}_t$ as in Equation \eqref{Eq: surrogate stochastic gradient energy}, where $\widehat{\sigma}_t$ is the HAC estimator on the corresponding process $Y_i$ provided in Theorem \ref{Thm: Strong Consistency of Plug-in HAC Estimator}.  We omit the proof here but this follows from Proposition B.1 of \cite{waudby2024time} along with the same tricks as in our Corollary \ref{Corollary: Equivalent Asymp CS Boundary}.\footnote{\label{fn: 1S,2S diff}Note that the one-sided $(1-\alpha)$-AsympCS resembles a two-sided $(1-2\alpha)$-AsympCS but with an additional additive 1 inside log. A similar phenomenon appears in the one- and two-sided sub-Gaussian CSs of \cite{howard2021time}.} The optimal choice of $\rho$ is provided in Appendix \ref{Appendix: optimal rho}. A natural surrogate-based (asymptotically) anytime-valid stopping rule is $\tau_\varepsilon := \inf\{t\geq T_0: U_t^+(\alpha)\leq \varepsilon\}$, which certifies that the average expected
stochastic-gradient energy is below a threshold $\varepsilon > 0$.  

\vspace{3pt}
In contrast to our approach, \cite{aolaritei2025stopping} construct finite-sample anytime-valid confidence sequences for the weighted-average suboptimality of projected SGD under general convexity, allowing decaying stepsizes and requiring neither smoothness nor strong convexity. They handle trajectory dependence by isolating a martingale-difference structure in the suboptimality decomposition. Our construction instead targets average stochastic-gradient energy in the online, constant-stepsize regime. It yields AsympCSs with time-uniform CLT-rate widths under weak dependence. Furthermore, although we only focus SGD with a constant-step size, it is possible and interesting to extend the current framework to accommodate decaying and other more general learning-rate-based SGDs in that $\V_t \asymp t$ is only a sufficient condition but not as a necessary condition to derive our Theorem \ref{Thm: Robbins Gaussian Mixture Asymp CS} (see more discussion in the corresponding footnote \ref{fn:Vt order}). We thus leave them to future work.

\subsection{Proof of Proposition \ref{Prop: Asymp CS for MCMC averages}} \label{Proof of Prop: Asymp CS for MCMC averages}

\begin{proof}
Under the contraction and initial moment conditions in
Equations~\eqref{eq:contraction_map} and
\eqref{eq: initial moment condition}, respectively, Theorem~4.1 of
\cite{wu2010dependence} (also see
\cite{wu2004limit,diaconis1999iterated}) shows that the associated
backward iterations converge almost surely and yield a strictly
stationary ``causal'' solution
\(\overline{\bm{X}}_i
:=\varphi(\bm{\varepsilon}_i,\bm{\varepsilon}_{i-1},\ldots)\)
with marginal distribution \(\pi\) and
\(\overline{\bm{X}}_0\in \mathcal{L}_p\). Therefore,
\(\overline{Y}_i:=f(\overline{\bm{X}}_i)\) is itself a stationary causal
process and falls within the functional dependence measure framework.

\vspace{3pt}
Let \(\overline{Y}_i^{(i-h)}\) be obtained by replacing
\(\bm{\varepsilon}_{i-h}\) with an independent copy, and denote the
corresponding functional dependence measure of the
\(\overline{Y}\)-process by
\(\delta_p(h):=\|\overline{Y}_i-\overline{Y}_i^{(i-h)}\|_p\).
Such a replacement affects \(\overline{\bm{X}}_i\) through \(h\)
successive contractive maps. Iterating the contraction inequality \eqref{eq:contraction_map} and
using the Lipschitz property of \(f\) give
\(\delta_p(h)\leq C L_p^h\) for some \(C<\infty\). Thus the functional
dependence measure decays exponentially fast at the geometric rate
\(L_p^h\). Consequently,
\(\Theta_{m,p}\leq C L_p^m/(1-L_p)\), and, because exponential decay
dominates any polynomial rate,
\(\sup_{m\geq0}(m+1)^A\Theta_{m,p}<\infty\) for every \(A>0\).
Moreover, \(\overline{Y}_0\in \mathcal{L}_p\), and, by stationarity, we may take
\(Q:=|\overline{Y}_0|\), so that
\(\sup_i\P[|\overline{Y}_i|>y]=\P[Q>y]\). Thus
Assumptions~\ref{Assump:dep_decay} and~\ref{Assump:UI} hold, and
Theorem~\ref{Thm: Robbins Gaussian Mixture Asymp CS} applies to the
oracle estimator
\(\widehat{\mu}_t^\star:=t^{-1}\sum_{i=1}^t\overline{Y}_i\).

\vspace{3pt}
Now let \(\bm{X}_0\in \mathcal{L}_p\) and construct the observed and stationary
chains using the same future innovations. Iterating the contraction
condition along these common innovations gives
\(\|\bm{X}_i-\overline{\bm{X}}_i\|_p
\leq L_p^i\|\bm{X}_0-\overline{\bm{X}}_0\|_p\).
This is the geometric moment contraction (GMC) property: the two chains
forget their different initializations and approach one another
geometrically fast in \(\mathcal{L}_p\). Hence passing from the stationary chain
to the observed chain incurs only a summable transient error. Indeed,
the Lipschitz property of \(f\) implies
\(\sum_{i\geq1}\|Y_i-\overline{Y}_i\|_p<\infty\), and therefore, by
Minkowski's inequality,
\(M:=\sum_{i\geq1}|Y_i-\overline{Y}_i|<\infty\) almost surely. It follows
that
\[
    |\widehat{\mu}_t-\widehat{\mu}_t^\star|
    \leq \frac{M}{t}
    =O_{\mathrm{a.s.}}(t^{-1}).
\]
Since this difference is negligible relative to the asymptotic
confidence-sequence boundary, replacing the oracle estimator by
\(\widehat{\mu}_t\) does not affect the AsympCS conclusion. The result
then follows under the stated LRV conditions.
\end{proof}

\section{Additional Discussions} \label{Appendix: Additional Discussions}
\noindent In this section, we provide several additional discussions. Section \ref{Appendix: Functional Dependence Measure Framework} provides an overview of the functional dependence measure (FDM) framework. For readers interested in the underlying probability theory, Section \ref{Appendix: Non-Degenerate Condition on the Cumulative Variance Process} briefly discusses connections to recent developments of nonasymptotic KMT approximation results and the different roles of non-degeneracy conditions on the cumulative variance process. Section \ref{Appendix: Comparison of Admissible Bandwidth Classes} compares the admissible bandwidth classes under the FDM framework and mixing conditions. Finally, Section \ref{Appendix: optimal rho} provides the optimal choice of the tuning parameter $\rho$ for the Robbins' Gaussian-mixture AsympCSs.

\subsection{Functional Dependence Measure Framework} \label{Appendix: Functional Dependence Measure Framework}
\noindent This section provides discussion of the FDM framework. We first present illustrative examples for Assumptions \ref{Assump:dep_decay} and \ref{Assump:UI}, then compare FDM with mixing conditions, and finally discuss the assumptions imposed in \cite{hafouta2023almost} under mixing conditions.

\subsubsection{Illustrative Examples for Assumptions \ref{Assump:dep_decay} and \ref{Assump:UI}} \label{Appendix: Examples Satisfying and Violating the Assumptions}
\noindent We first survey real-valued stationary and nonstationary time-series models whose functional dependence properties have been characterized in the literature, and specifically lay out ones satisfying Assumptions \ref{Assump:dep_decay} and \ref{Assump:UI}.
Models not on the list should not be simply interpreted as counterexamples. Rather, it primarily reflects that their functional dependence properties may not yet have been formally established in the literature.
\begin{table}[H]
\centering
\label{tab:stationary-functional-dependence-examples}

\small
\setlength{\tabcolsep}{6pt}
\renewcommand{\arraystretch}{1.15}

\begin{tabularx}{\textwidth}{
    @{}
    >{\raggedright\arraybackslash}X
    >{\raggedright\arraybackslash}p{0.38\textwidth}
    @{}
}
\toprule
Representative Examples
    & References \\
\midrule

\multicolumn{2}{@{}l}{\textit{1. Linear processes}} \\[0.2em]

General linear processes (e.g., finite-order MA and stable
ARMA)
    & \cite[Equation ~6]{wu2005nonlinear} \\

\addlinespace[0.5em]
\multicolumn{2}{@{}l}{\textit{2. Nonlinear time series}} \\[0.2em]

Nonlinear AR models / ARCH models 
    & \cite[Example~2.2]{berkes2014kmt} \\

Exponential AutoRegressive (EXPAR) models
    & \cite[Example~5.1]{shao2007asymptotic} \\ 

Threshold autoregressive (TAR) models
    & \cite[Example~2.2]{berkes2014kmt} \\

Random-coefficient AR models
    & \cite[Equation 5.4]{shao2007asymptotic}\\

Bilinear models
    & \cite[Example~5.3]{shao2007asymptotic} \\

\addlinespace[0.5em]
\multicolumn{2}{@{}l}{
    \textit{3. Time-varying Volatility Models}
} \\[0.2em] 

AR models with ARCH errors
    & \cite[Example~5.2]{shao2007asymptotic} \\

GARCH, APARCH, ARMA--GARCH; augmented GARCH
    & \cite[Propositions~5.1--5.2]{shao2007asymptotic} \\ 

Gaussian AR(1) log-stochastic-volatility models
    & \cite{kokoszka2025dependence} \\

\addlinespace[0.5em]
\multicolumn{2}{@{}l}{\textit{4. Nonlinear series representations}} \\[0.2em]

(discrete-time) Volterra processes
    & \cite[Example~2.3]{berkes2014kmt} \\
\bottomrule
\end{tabularx}
\caption{FDM: real-valued examples of stationary time series (also see survey in \cite{wu2011asymptotic})}
\end{table}

\begin{table}[H]
\centering
\label{tab:nonstationary-functional-dependence-examples}

\small
\setlength{\tabcolsep}{6pt}
\renewcommand{\arraystretch}{1.15}

\begin{tabularx}{\textwidth}{
    @{}
    >{\raggedright\arraybackslash}X
    >{\raggedright\arraybackslash}p{0.38\textwidth}
    @{}
}
\toprule
Representative Examples
    & References \\
\midrule

Time-varying random-coefficient and volatility models
(e.g., tvAR(p), tvARCH(p), tvTAR(1))
    & \cite[Example~2.2 and Proposition~4.4(ii)]{dahlhaus2019towards} \\

Cyclostationary process
    & \cite[Example~2.5.1]{bonnerjee2024gaussian}; \cite{parzen1979approach} \\
    
Bounded piecewise-stationary models with finitely many
structural breaks
    & Direct consequences of uniform-contraction results \\
\bottomrule
\end{tabularx}
\caption{FDM: real-valued examples of non-stationary time series} 
\end{table}

\noindent We now give out two examples not satisfying our Assumptions \ref{Assump:dep_decay}-\ref{Assump:UI}, but as discussed in the main text, we do not regard these examples as posing serious practical concerns. 

\begin{IEg}[ARFIMA$(0,d,0)$] 
For \(0<d<1/2\), consider the fractionally differenced process: $X_t=(1-L)^{-d}\varepsilon_t
    =\sum_{j=0}^{\infty}\psi_j\varepsilon_{t-j},
\
\psi_j
    =\frac{\Gamma(j+d)}{\Gamma(j+1)\Gamma(d)}
    \sim \frac{j^{d-1}}{\Gamma(d)},$
where \(L\) denotes the lag operator with $\varepsilon_t \overset{\mathrm{iid}}{\sim} \mathcal{N}(0,1)$, and 
\(\Gamma(\cdot)\) denotes the Gamma function: $\Gamma(z)=\int_0^\infty u^{z-1}e^{-u}\,du, \ x>0.$ Then the process fails to satisfy Assumption \ref{Assump:dep_decay}.
\end{IEg}
\noindent The ARFIMA$(0,d,0)$ model is a stationary long-memory process whose partial sums no longer converge to a Brownian motion limit. Instead, under an appropriate normalization, its limit is a fractional Brownian motion (\cite{taqqu1975weak}). 
\begin{IEg}[Unit-root random walk]
Consider the process: $X_t = X_{t-1} + \varepsilon_t$ with $\varepsilon_t \overset{\mathrm{iid}}{\sim} \mathcal{N}(0,1).$ Then the process fails to satisfy Assumption \ref{Assump:UI}. 
\end{IEg}
\noindent The unit-root random walk can often be handled by the standard differencing transformation as in the classic fixed-time regime, after which the resulting process satisfies our stochastic dominance condition.

\subsubsection{Comparison with Mixing Conditions} 
\noindent We refer to Chapter 5 in \cite{rio2017asymptotic} for a gentle introduction to the basics of mixing conditions. In short, mixing conditions are characterized through the dependence structure between sigma-algebras. Unlike the FDM framework, they are often difficult to verify in nonlinear models and may even fail to hold for rather simple processes (see e.g., \cite{andrews1984non} for an AR model and \cite{doukhan2009weak} for an ARCH model), but these processes still satisfy dependency conditions under \cite{wu2005nonlinear}'s framework. For related comparisons of dependence conditions in other contexts, see, e.g., Appendix B in \cite{ballarin2026impulse}. For more complex settings, we refer to \cite{han2023probability} for reviews and benefits of functional dependence measures, where there are substantially new challenges to the verification of mixing conditions.  

For readers interested in conditions and examples under which functional dependence follows from mixing conditions, and conversely when mixing conditions can be derived from functional dependence, we refer to some of the recent developments in the literature (\cite{hill2025mixingale}; \cite{heinrichs2026note}). These comparisons are not intended to be exhaustive, but rather to provide a useful point of reference.

\subsubsection{Assumptions in the Existing Nonstationary ASIP Result under Mixing} \label{Appendix: Assumptions in Existing Nonstationary ASIP Results}
\noindent We remark that our ASIP differs from that of \cite{hafouta2023almost} in terms of the approximation error rate and the variance structure of the Gaussian approximation.
We therefore present their assumptions solely to compare the conditions imposed on the underlying processes.

\begin{Ass}[Mixing Conditions] \label{Assump: nonstationary mixing coefficients condition}
Let $\{X_t : t \ge 1\}$ be a zero-mean, (univariate) real-valued process with $\alpha$-mixing coefficients $\alpha(n)$ and $\phi$-mixing coefficients $\phi(n)$. Assume that the mixing coefficients satisfy the following conditions simultaneously: (i) $\alpha(n)\;\le\; C \delta^n \ \text{for some }C<\infty,\ \delta \in (0,1)$, and (ii) $\phi(n_0) < \frac{1}{2}$ for some $n_0 \in \mathbb{N}$.
\end{Ass}

\begin{Ass}[Boundedness] \label{Assump: boundedness mixing condition}
For the same process in Assumption \ref{Assump: nonstationary mixing coefficients condition}, assume that $(X_t)_{t\ge 1}$ is uniformly bounded, i.e., there exists a constant $B > 0$ such that $\sup_{t\ge 1} |X_t| \le B$ almost surely. 
\end{Ass}
\noindent Assumption \ref{Assump: nonstationary mixing coefficients condition} and Assumption \ref{Assump: boundedness mixing condition}  impose substantially stronger conditions than our Assumption \ref{asm:dep_decay} and Assumption \ref{asm:UI}. In particular, we relax exponential dependence decay to polynomial decay and uniform boundedness to uniform boundedness in moments, thereby accommodating a substantially broader class of processes.

\subsection{KMT Approximations and Different Usage of the Non-Degenerate Condition} \label{Appendix: Non-Degenerate Condition on the Cumulative Variance Process}
\noindent KMT rates are usually also of independent interest, beyond the scope that we consider in this paper. For time-series applications, they are valuable for conducting structural break tests and constructing simultaneous confidence bands for nonparametric trend estimation in complex scenarios (\cite{wu2007inference}; \cite{zhou2010simultaneous}). Recent work on nonasymptotic KMT approximation uses a similar peeling perspective to obtain sharp path-wise couplings (\cite{waudby2025nonasymptotic}). The distinction from our result of Theorem \ref{Thm: nonstationary ASIP} lies in how the group-wise error probabilities are aggregated. In this paper, we make these probabilities summable and then pass to the limit through the Borel--Cantelli lemma, which yields an asymptotic almost sure approximation but does not keep track of the finite-sample error aggregation.  Nonasymptotic analysis keeps this accounting explicit in the union-bound step, which then enables finite-sample guarantees. Since our primary goal is to establish an ASIP under mild conditions to validate the time-uniform dependent CLT, we do not optimize the outer aggregation step, although we think such construction is feasible in time-series settings. A nonasymptotic version of the present ASIP, possibly with distribution-uniform guarantees, is thus left for future work.

Besides, we remark a difference between our ASIP/AsympCSs and those in the existing literature. In \cite{waudby2024time} and its related subsequent work, they rely on \cite{strassen1967almost}. Consequently, a non-degenerate condition on $\V_t \to \infty$ is required to establish both ASIP results for martingale difference sequence, and to ensure that the approximation error rate is dominated by the nonasymptotic boundary rate for valid construction of AsympCSs. In contrast, our approach only requires this condition for the latter purpose and in particular, we do not assume it in our Theorem \ref{Thm: nonstationary ASIP}. We could achieve this because subsequent ASIP literature and thus ours leverages finer nonasymptotic techniques from \cite{sakhanenko2006estimates}.

\subsection{Comparison of Admissible Bandwidth Classes between Functional Dependence Framework and Mixing Condition Framework} \label{Appendix: Comparison of Admissible Bandwidth Classes}
\noindent For weakly consistent HAC estimators, the classic bandwidth conditions have been shown to be $b_{t}^{\text{FDM}} = o(t)$ (Theorem 6 in \cite{wu2012covariance}) and $b_{t}^{\text{Mixing}} = o(t^{1/2-1/(p \wedge 4)})$(Theorem 2 in \cite{de2000strong}), with the latter subsequently relaxed Hansen's rate $o(t^{1-2/(p \wedge 4)})$ (\cite{davidson2020new}). This advantage of enabling more feasible bandwidth choices under the FDM framework appears to be quite general. A similar phenomenon is also observed in more general spectral density estimation (\cite{shao2007asymptotic}; \cite{liu2010asymptotics}) and in the present setting for strongly consistent HAC estimators, where Remark \ref{Rmk: A subset relation for admissible bandwidth classes} also reveals a polynomial-factor difference, up to logarithmic terms. This difference appears to reflect sharper control of the quadratic-form estimation error under functional dependence (\cite{liu2010asymptotics}), although additional structure can also relax bandwidth restrictions in the mixing framework (\cite{davidson2020new}).

\vspace{8pt}
\noindent We next complete the argument per further enlarging the admissible bandwidth choice set. As seen from the exisiting literature (\cite{de2000strong}; \cite{davidson2020new}), truncating the effective moment order at four provides a convenient simplification, with little practical loss since commonly used bandwidth choices are typically of order smaller than $t^{1/2}$. This is because when applying the von Bahr--Esseen inequality (\cite{von1965inequalities}) or the Burkholder inequality (\cite{burkholder1973distribution}), letting $q:=p/2>1$, the corresponding moment bound for the martingale difference sequence $\{X_i\}_{i=1}^{N}$ is of order
\[
\left\|\sum_{i=1}^N X_i\right\|_q
\lesssim
N^{1/(q\wedge 2)}
=
\begin{cases}
N^{2/p}, & 2<p<4,\\
N^{1/2}, & p\ge 4,
\end{cases}
\]
so that the order saturates at $N^{1/2}$ once $p\ge4$ and cannot be further improved.

However, under a bounded variation condition (Condition 3 of \cite{liu2010asymptotics}), 
the result may still be established for $p>4$ while retaining the larger admissible bandwidth, i.e., $b_t = O\!\left(t^{1-2/p}(\log t)^{-\lambda}\right), \lambda>2/p.$ This extension requires a maximal banded quadratic-form argument. In particular, Burkholder's inequality combined with Young's convolution inequality yields the required moment bound for rectangular partial sums, while the two-parameter (time index and lag index) maximal inequality of \cite{moricz1977moment} provides uniform control over both the time and lag indices. This approach would allow the admissible bandwidth to be further relaxed in terms of the leading polynomial order as higher-order moments become available. We do not pursue this approach here, as we think the additional practical gains are likely to be limited.

\subsection{Optimizing Robbins' Normal Mixture} \label{Appendix: optimal rho}
\noindent In this section, we outline how one can choose $\rho$ to optimize the Robbins' normal mixture boundary for a specific time $t^{\star}$ and type-I error level $\alpha \in (0,1)$. Specifically, we outline the exact solution and a closed-form approximation for the two-sided AsympCS as in Corollary \ref{Corollary: Equivalent Asymp CS Boundary}, followed by an approximate solution for the one-sided AsympCS in Equation \eqref{Eq: one-sided upper boundary}. The boundary in Corollary \ref{Corollary: Equivalent Asymp CS Boundary} is asymptotically equivalent to that in Theorem \ref{Thm: Robbins Gaussian Mixture Asymp CS} but is easier to optimize. This section is similar to Appendix B.2 of \cite{waudby2024time}, so we omit proof details and only present results for completeness. 
 
\subsubsection{Optimal \texorpdfstring{$\rho$}{rho} in a Two-sided AsympCS}
\noindent  \textbf{The exact solution.} Consider the boundary in Corollary \ref{Corollary: Equivalent Asymp CS Boundary} at time $t$,
\[
\widetilde{\mathfrak{B}}_{t}(\alpha)
:=
\sqrt{
\frac{t  \rho^2 + 1}{t^2\rho^2}
\log\!\left(\frac{t\rho^2 + 1}{\alpha^2}\right)
}.
\]
Let $W_{-1}$ be the lower branch of the Lambert $W$ function (\cite{corless1996lambert}). Then, at the specific time of interest $t^{\star}$, we could obtain 
\begin{equation} \label{Eq: Exact Optimal rho in 2-sided Asymp CS}
\argmin_{\rho>0} \,\widetilde{\mathfrak{B}}_{t^{\star}}(\alpha)
=
\sqrt{
\frac{
-\,W_{-1}\!\left(-\alpha^{2}\exp(-1)\right)-1
}{t^{\star}}
}.
\end{equation}

\paragraph{An approximate solution.}
We can derive a closed-form approximation to Equation ~\eqref{Eq: Exact Optimal rho in 2-sided Asymp CS} by considering the Taylor series expansion of the Lambert \(W\) function. We could then obtain 
\[
\rho^{\text{2s-approx}}(t^\star)
:=
\sqrt{
\frac{
-2\log\alpha
+ \log\!\left(-2\log\alpha+1\right)
}{t^\star}
}.
\]

\subsubsection{Optimal \texorpdfstring{$\rho$}{rho} in a One-sided AsympCS}
\noindent As noted in the footnote \ref{fn: 1S,2S diff}, the one-sided boundary has an additional additive 1 inside log, and 
this precludes an exact solution in terms of the Lambert $W$ function. We therefore recommend using a proxy for the boundary, i.e., disregarding the additional additive 1, and applying the same argument as in the two-sided case, with $\alpha$ replaced by $2\alpha$. The intuition is that the two-sided construction divides the total error budget between the upper and lower tails, whereas the one-sided construction allocates the entire error budget to a single tail, yielding a sharper boundary. In particular, it suffices to obtain the following approximate solution for the one-sided boundary: 
\[
\rho^{\text{1s-approx}}(t^\star)
:=
\sqrt{
\frac{
-2\log2\alpha
+ \log\!\left(-2\log 2 \alpha+1\right)
}{t^\star}
}.
\]

\end{document}